\documentclass[reqno,A4paper]{amsart}
\usepackage{amsmath}
\usepackage{amssymb}
\usepackage{amsthm}
\usepackage{enumerate}
\usepackage{mathrsfs} %%\mathcal classic,  \mathscr decorative
\usepackage{eqlist}
\usepackage{array}

\usepackage[english]{babel}                             %kieliasetukset sun muut
\usepackage[latin1]{inputenc}
\usepackage{comment}
\usepackage{ae}
\usepackage[leqno]{amsmath}  %kaavat vasemmalle, int.alue alas

\usepackage[dotinlabels]{titletoc}     %piste sis?llysluetteloon

\usepackage{amssymb,latexsym,verbatim} %vaihtoehtoista s?l??
\usepackage{amstext}
\usepackage{amsfonts}
\usepackage{amsbsy}
\usepackage{amsopn}
\usepackage{enumerate}
\usepackage{txfonts}
\usepackage{amsmath, amsthm, amssymb}

\usepackage{color}
\usepackage[usenames,dvipsnames,svgnames,table]{xcolor}
\usepackage[numbers,sort&compress]{natbib}
\allowdisplaybreaks
\newtheorem{theorem}{Theorem}
\newtheorem{definition}[theorem]{Definition}
\newtheorem{lemma}[theorem]{Lemma}

\newtheorem{proposition}[theorem]{Proposition}

\numberwithin{equation}{section}
\numberwithin{theorem}{section}

\renewcommand{\epsilon}{\varepsilon}
\renewcommand{\rho}{\varrho}
\renewcommand{\widetilde}{\tilde}

\DeclareMathOperator*{\esssup}{ess\,sup}
\DeclareMathOperator*{\essinf}{ess\,inf}
\DeclareMathOperator*{\essosc}{ess\,osc}
\DeclareMathOperator*{\loc}{loc}

\def\XXint#1#2#3{{\setbox0=\hbox{$#1{#2#3}{\int}$ }
\vcenter{\hbox{$#2#3$ }}\kern-.6\wd0}}

\usepackage{graphicx} % for \rotatebox macro

\def\YYint#1#2#3{{\setbox0=\hbox{$#1{#2#3}{\iint}$}
    \vcenter{\hbox{$#2#3$}}\kern-.51\wd0}}

\numberwithin{equation}{section}
\begin{document}

\title[Continuity estimates for doubly degenerate parabolic equations]%
{Continuity estimates for doubly degenerate parabolic equations with lower order terms via nonlinear potentials
}
\author[Qifan Li]%
{Qifan Li*}

\newcommand{\acr}{\newline\indent}

\address{\llap{*\,}Department of Mathematics\acr
                   School of Sciences\acr
                   Wuhan University of Technology\acr
                   430070, 122 Luoshi Road,
                   Wuhan, Hubei\acr
                   P. R. China}
\email{qifan\_li@yahoo.com,
qifan\_li@whut.edu.cn}

 %Secondary is optional

\begin{abstract}
This article studies the continuity of bounded nonnegative weak solutions to inhomogeneous doubly nonlinear parabolic equations.
A model equation is
\begin{equation*}\partial_t u-\operatorname{div}(u^{m-1}|Du|^{p-2}Du)=f\qquad \text{in}\quad\Omega\times(-T,0)\subset \mathbb{R}^{n+1}.\end{equation*}
Here, we consider the
case $m>1$ and $2<p<n$. We establish a continuity estimate for $u$ in terms of elliptic Riesz potentials of the right-hand side of the equation.
\end{abstract}
\maketitle
\section{Introduction and main results}
We are concerned in this paper with the doubly degenerate parabolic equation with lower order terms.
More precisely, we consider quasilinear parabolic equations of the form
\begin{equation}\label{parabolic}\partial_t u-\operatorname{div}A(x,t,u,Du)=f,\end{equation}
in a bounded domain $\Omega_T=\Omega\times(-T,0)$, where $\Omega\subset\mathbb{R}^n$.
Here, we assume that $f=f(x)\in L_{\loc}^1(\Omega)$ and $u$ is nonnegative.
The vector field $A(x,t,u,\xi)$ is measurable in $\Omega_T\times\mathbb{R}\times\mathbb{R}^n$
and satisfies the growth and ellipticity conditions:
 \begin{equation}\label{A}
	\begin{cases}
	 \big\langle A(x,t,u,\xi),\xi\big \rangle\geq C_0u^{m-1}|\xi|^p,\\
	|A(x,t,u,\xi)|\leq C_1u^{m-1}|\xi|^{p-1}+g(x)u^{\frac{m-1}{p}},
	\end{cases}
\end{equation}
where $C_0$, $C_1$ are given positive constants, $g\geq0$ and $g=g(x)\in L_{\loc}^{p/(p-1)}(\Omega)$.
Throughout the paper, we keep $m>1$ and $2<p<n$. In this case, the quasilinear equation \eqref{parabolic}
with the structure conditions \eqref{A} is classified as doubly degenerate.
Doubly degenerate parabolic equations play an important role in the study of the non-stationary, polytropic flow of a fluid in a porous medium
(see, for example \cite{BER}). The H\"older continuity of nonnegative weak solutions to \eqref{parabolic}-\eqref{A}
 was proved in \cite{PV, I}. In \cite{BDLS} the authors proved the local H\"older continuity of sign-changing
solutions to doubly nonlinear parabolic equations in the range $p>2$ and $m+p>2$.
Furthermore, Harnack inequality for doubly degenerate parabolic equations was first established by
Vespri \cite{V}. This result was extended in \cite{FS}
to the general quasilinear structure. Recently, B\"ogelein et al. \cite{BHSS} proved the Harnack inequality in the case $m>0$, $p>1$
and $m(p-1)>1$ which is the full range in slow diffusion
case.

The aim of this paper is to prove continuity of weak solutions to \eqref{parabolic}
in terms of nonlinear potentials of the functions $f$
and $g$.
The continuity problem related to the potentials
can be traced back to the study of the Wiener criteria of $p$-Laplace equations (see \cite{KM}).
In \cite{KM} the authors proved the local boundedness of weak solutions to $p$-Laplace equations with measure data.
Furthermore, Liskevich and Skrypnik \cite{LS, Sk} studied evolutionary $p$-Laplace equations with lower order terms from nonlinear Kato-type
classes and
established qualitative properties such as local boundedness, continuity and
Harnack estimate for weak solutions.
Subsequently, Liskevich, Skrypnik and Sobol \cite{LSS} extended the local boundedness result
via a time-dependent parabolic potential.
Similar results for porous medium equations have been obtained
in \cite{BDG2, BDG1, z2, z1}.
However, limited work has been done for the doubly nonlinear equations. In \cite{SS} the author obtained
a local boundedness result for the doubly degenerate parabolic equations in terms
of the parabolic nonlinear potentials.
Motivated by this work, we are interested in the continuity estimate for doubly degenerate parabolic equations
with lower order terms from nonlinear Kato-type classes.
Before formulating the main result, let us
introduce the following quantities
\begin{equation}\begin{split}\label{potential1}
F_1(R)=\sup_{x\in \Omega}\int_0^R\left(\frac{1}{r^{n-p}}\int_{B_r(x)}g(y)^{\frac{p}{p-1}}
\,\mathrm {d}y\right)^{\frac{1}{p}}\frac{1}{r}\,\mathrm {d}r
\end{split}\end{equation}
and
\begin{equation}\begin{split}\label{potential2}
F_2(R)=\sup_{x\in \Omega}\int_0^R\left(\frac{1}{r^{n-p}}\int_{B_r(x)}f(y)
\,\mathrm {d}y\right)^{\frac{1}{p-1}}\frac{1}{r}\,\mathrm {d}r,
\end{split}\end{equation}
where $R>0$ and $B_r(x)=\left\{y\in\Omega:|y-x|\leq r\right\}$.
These quantities are also termed as elliptic Riesz potentials (see \cite{z1}).
Next, we give the definition of weak solutions to doubly nonlinear parabolic equations.
\begin{definition}\label{weak solution}
A nonnegative measurable function $u:\Omega_T\to \mathbb{R}$ is
said to be a local weak solution to \eqref{parabolic}-\eqref{A} if
\begin{equation*}\begin{split}
u\in C_{\loc}(-T,0;L_{\loc}^2(\Omega)), \quad u^\alpha\in L_{\loc}^p(-T,0;W_{\loc}^{1,p}(\Omega)),\quad
\text{where}\quad \alpha=\tfrac{m+p-2}{p-1}
\end{split}\end{equation*}
and
for every open set
$U\subset \Omega$ and every subinterval $(t_1,t_2)\subset(-T,0)$
the identity
\begin{equation}\begin{split}\label{weaksolution}
\int_U &u(\cdot,t)\varphi(\cdot,t)dx\bigg|_{t=t_1}^{t_2}+\iint_{U\times(t_1,t_2)}-u\partial_t
 \varphi+\big\langle A(x,t,u,Du), D\varphi\big\rangle
\,\mathrm {d}x\mathrm {d}t
\\&=\iint_{U\times(t_1,t_2)}  f \varphi\,\mathrm {d}x\mathrm {d}t
\end{split}\end{equation}
holds for
all any function
$\varphi\in W_{\loc}^{1,2}(-T,0;L^2(U))\cap L_{\loc}^p(-T,0;W_0^{1,p}(U)).$
 \end{definition}
 We note that the gradient of $u$ in the structure conditions \eqref{A} is interpreted as $Du=\alpha^{-1}u^{1-\alpha}
 \chi_{\{u>0\}}Du^\alpha$.
 We are now in a position to state our main theorem.
 \begin{theorem}\label{main}
 Let $u$ be a nonnegative, locally bounded, weak solution to the doubly degenerate parabolic equation \eqref{parabolic}
 in the sense of Definition \ref{weak solution}, where the vector field $A$ fulfills the structure conditions \eqref{A}.
 Furthermore, assume that
 \begin{equation}\label{limit F_1F_2}
 \lim_{R\downarrow 0}\left[F_1(R)+F_2(R)\right]=0
 \end{equation}
 holds true. Then, $u$ is locally continuous in $\Omega_T$.
 \end{theorem}
No attempt has been made here to study the regularity problem for the time-dependent parabolic potential.
Our proof relies on the method of intrinsic scaling and follow the idea from \cite[chapter III]{Di93}.
In the context of the potential estimate,
it seems that neither De Giorgi nor Moser iteration technique work in the proof of De Giorgi type lemmas.
Our proof is based on the Kilpel\"ainen-Mal\'y technique \cite{KM, LS}, but the arguments
become more involved when we move from the $p$-Laplace to the
doubly nonlinear setting.
Contrary to \cite{BDG1}, which established a continuity result for the porous medium equations with nonnegative measure data,
our result could handle the case when the right-hand side of the equation $f$ is a sign-changing function.

The rest of the paper is organized as follows. In Sect. 2, we derive the Caccioppoli inequalities
and set up an alternative argument.
Subsequently
Sect. 3 is devoted to the
analysis of the first alternative. We obtain a decay estimate for the oscillation of the weak solution in terms of $F_1(R)$ and $F_2(R)$.
In Sect. 4 we proceed with the study of the second alternative and obtain a similar estimate for the oscillation of the weak solutions
via nonlinear potentials.
Finally the proof of the main theorem is presented in Sect. 5.
\section{Preliminary material and Caccioppoli inequalities}
In this section, we provide some preliminary lemmas. To start with,
we follow the notation used in \cite{LSS} and introduce some auxiliary functions. Throughout the paper, we define
\begin{equation}\label{G}
	G(u)=\begin{cases}
u,&\quad \text{for}\quad u>1,\\
	u^2,&\quad \text{for}\quad 0\leq u\leq1
\\
	0,&\quad \text{for}\quad u<0.
	\end{cases}
\end{equation}
Let $s\in\mathbb{R}$ and $s_+=\max\{s,0\}$. We observe that $G(s)=
\min\{s_+,s_+^2\}$. Fix $0<\lambda<p-1$, we introduce the
function
\begin{equation*}\phi_+(s)=\int_0^{s_+}(1+\tau)^{-1-\lambda}\,\mathrm {d}\tau=\tfrac{1}{\lambda}
\left(1-\left(1+s_+\right)^{-\lambda}\right).
\end{equation*}
It can be easily seen that there exist $\gamma_1$ and $\gamma_2$ depending only upon $\lambda$ such that $\gamma_1<\gamma_2$,
\begin{equation}\label{phi}\gamma_1\min\{s_+,1\}\leq\phi_+(s)
\leq \gamma_2\min\{s_+,1\}\qquad\text{and}\qquad
\gamma_1\frac{s_+}{s_++1}\leq\phi_+(s)
\leq\gamma_2\frac{s_+}{s_++1}.
\end{equation}
In this paper, we fix the parameter $\lambda$ such that $\lambda<\min\left\{(p-1)^{-1},n^{-1}\right\}$.
The statement that a constant  $\gamma$ depends only upon the data means that it can be determined a priori only in terms
of $\{n,m,p,C_0,C_1,\lambda\}$. Next, we set
\begin{equation*}
\Phi(s)=\int_0^{s_+}\phi_+(\tau)\,\mathrm {d}\tau
\end{equation*}
and it follows from \eqref{phi} that
\begin{equation}\label{PhiG}
\gamma_1G(s)=\gamma_1\min\{s_+,s_+^2\}\leq \Phi(s)\leq\gamma_2 \min\{s_+,s_+^2\}=\gamma_2G(s).
\end{equation}
Furthermore, we introduce the function
\begin{equation*}
\psi(s)=\int_0^{s_+}(1+\tau)^{-\frac{1}{p}-\frac{\lambda}{p}}\,\mathrm {d}\tau=\tfrac{p}
{p-1-\lambda}\left[(1+s_+)^{1-\frac{1+\lambda}{p}}-1\right].
\end{equation*}
Let $d>0$, $l>0$ and $u$ be the weak solution
as in the Definition \ref{weak solution}. We define the quantities
\begin{equation}\label{psi-}
\psi_-=\left(\frac{1}{d}\int_u^l\left(1+\frac{l-s}{d}\right)^{-\frac{1}{p}-\frac{\lambda}{p}}
\,\mathrm {d}s\right)_+=\psi\left(\frac{l-u}{d}\right)
\end{equation}
and
\begin{equation}\label{psi+}
\psi_+=\left(\frac{1}{d}\int_l^u\left(1+\frac{s-l}{d}\right)^{-\frac{1}{p}-\frac{\lambda}{p}}
\,\mathrm {d}s\right)_+=\psi\left(\frac{u-l}{d}\right).
\end{equation}
At this point, we state the following two lemmas which give a characterization of the quantities $\psi_-$
and $\psi_+$.
  \begin{lemma}\label{lemmainequalitypsi-}
  Let $\psi_-$ be the quantity defined in \eqref{psi-}. Then, we have
    \begin{equation}\label{inequalitypsi-}
c_1(\epsilon_1)\left(\frac{l-u}{d}\right)^{\frac{p-1-\lambda}{p}}\leq\psi_-\leq  c_2(\epsilon_1)
\left(\frac{l-u}{d}\right)^{\frac{p-1-\lambda}{p}}
\qquad\text{if}\qquad \frac{l-u}{d}\geq \epsilon_1.
\end{equation}
\end{lemma}
  \begin{lemma}\label{lemmainequalitypsi+}
  Let $\psi_+$ be the quantity defined in \eqref{psi+}. Then, we have
  \begin{equation}\label{inequalitypsi+}
c_1(\epsilon_1)\left(\frac{u-l}{d}\right)^{\frac{p-1-\lambda}{p}}\leq\psi_+\leq  c_2(\epsilon_1)
\left(\frac{u-l}{d}\right)^{\frac{p-1-\lambda}{p}}
\qquad\text{if}\qquad \frac{u-l}{d}\geq \epsilon_1
\end{equation}
\end{lemma}
The proofs of these two lemmas are not particularly difficult but will not be reproduced here.
 Let $\rho>0$, $\theta>0$ and $(x_0,t_0)\in\Omega_T$ be a fixed point.
 We set $Q_{\rho,\theta}(z_0)=B_\rho(x_0)\times(t_0-\theta,t_0)\subset \Omega_T$
and denote by
$\varphi$ a piecewise smooth function in $Q_{\rho,\theta}(z_0)$ such that
\begin{equation}\label{def zeta}0\leq\varphi\leq1, \quad|D\varphi|<\infty\quad\text{and}\quad\varphi=0\quad\text{
on}\quad\partial B_\rho(x_0),\end{equation}
where $\partial B_\rho(x_0)=\{y:|y-x_0|=\rho\}$.
If $x_0$ is the origin, then we omit in our notation the point $x_0$ and write $B_\rho$ for $B_\rho(x_0)$.
Finally, let
$$\partial_PQ_{\rho,\theta}(z_0)=
[\partial B_\rho(x_0)\times (t_0-\theta,t_0)]\cup [B_\rho(x_0)\times\{t_0-\theta\}]$$
denotes the parabolic boundary of
$Q_{\rho,\theta}(z_0)$.
The crucial result in our development of regularity property of weak solutions will be the following Caccioppoli inequality.
  \begin{lemma}\label{Cac1}
  Let $d>0$, $l>0$ and let
  $u$ be a nonnegative weak solution to the doubly degenerate parabolic equation \eqref{parabolic}
 in the sense of Definition \ref{weak solution}.
 There exists a constant $\gamma$ depending only upon the data, such that
 for every piecewise smooth cutoff function $\varphi$ satisfying \eqref{def zeta},
 we have
  \begin{equation}\begin{split}\label{Cacformula1}
  \esssup_{t_0<t<t_1}&\int_{L^-(t)}G\left(\frac{l-u}{d}\right)\varphi^k\,\mathrm {d}x+
  d^{p-2}\iint_{L^-}u^{m-1}|D\psi_-|^p\varphi^k\,\mathrm {d}x\mathrm {d}t
  \\
 \leq &\gamma\int_{L^-(t_0-\theta)}G\left(\frac{l-u}{d}\right)\varphi^k\,\mathrm {d}x
 +\gamma \iint_{L^-}\frac{l-u}{d}\left|\partial_t\varphi\right|\varphi^{k-1}\,\mathrm {d}x\mathrm {d}t
  \\&+\gamma  d^{p-2}\iint_{L^-}u^{m-1}\left(\frac{l-u}{d}\right)^{(1+\lambda)(p-1)}
  \varphi^{k-p}|D\varphi|^p\,\mathrm {d}x\mathrm {d}t
\\&+\gamma \frac{\theta}{d^2}\int_{B_\rho(x_0)}g^{\frac{p}{p-1}}\,\mathrm {d}x
  +\gamma \frac{\theta}{d}\int_{B_\rho(x_0)}|f|\,\mathrm {d}x,
   \end{split}\end{equation}
   where $k>p$, $L^-=Q_{\rho,\theta}(z_0)\cap \{u\leq l\}$ and $L^-(t)=B_\rho(x_0)\cap \{u(\cdot,t)\leq l\}$.
  \end{lemma}
  \begin{proof}
  To simplify the notation, we set $B=B_\rho(x_0)$ and $Q=Q_{\rho,\theta}(z_0)$.
 In the weak formulation \eqref{weaksolution} we choose the testing function
  \begin{equation*}\begin{split}\varphi_-=\left[\int_u^l\left(1+\frac{l-s}{d}\right)^{-1-\lambda}\,\mathrm {d}s\right]_+
  \varphi^k=d\phi_+\left(\frac{l-u}{d}\right)\varphi^k.\end{split}\end{equation*}
The use of $\varphi_-$ as a testing function is justified, modulus a time mollification technique (see \cite[Lemma 2.7]{SS}).
This gives
\begin{equation*}\begin{split}
   \iint_Q\varphi_-\partial_tu \,\mathrm {d}x\mathrm {d}t
   + \iint_Q\big\langle A(x,t,u,Du), D\varphi_-\big\rangle
\,\mathrm {d}x\mathrm {d}t=\iint_Qf \varphi_-\,\mathrm {d}x\mathrm {d}t.
    \end{split}\end{equation*}
 We now proceed formally for the term involving the time derivative. From \eqref{PhiG}, we see that
 for any $t\in (t_0-\theta,t_0)$
   \begin{equation*}\begin{split}
   \int_{t_0-\theta}^t&\int_B\varphi_-\partial_tu \,\mathrm {d}x\mathrm {d}t
 \\ & =-\int_{t_0-\theta}^{t}\int_B\frac{\partial}{\partial t} \left[\left(\int_u^l\,\mathrm {d}w
   \int_w^l\left(1+\frac{l-s}{d}\right)^{-1-\lambda}\,\mathrm {d}s\right)_+\right]\varphi^k\,\mathrm {d}x\mathrm {d}t
  \\ &\leq -\gamma d^2\int_{L^-(t)}G\left(\frac{l-u}{d}\right)\varphi^k\,\mathrm {d}x
  \\&\quad+\gamma d^2
  \iint_{L^-}\frac{l-u}{d}\varphi^{k-1}\left|\partial_t\varphi\right|\,\mathrm {d}x\mathrm {d}t
  +\gamma d^2\int_{L^-(t_0-\theta)}G\left(\frac{l-u}{d}\right)\varphi^k\,\mathrm {d}x.
   \end{split}\end{equation*}
   Next, we consider the second term on the left-hand side. To this end, we decompose
   \begin{equation*}\begin{split}
   \iint_Q&\big\langle A(x,t,u,Du), D\varphi_-\big\rangle
\,\mathrm {d}x\mathrm {d}t
\\= &k\iint_Q A(x,t,u,Du)\varphi^{k-1} D\varphi
\left(\int_u^l\left(1+\frac{l-s}{d}\right)^{-1-\lambda}\,\mathrm {d}s\right)_+
\,\mathrm {d}x\mathrm {d}t
\\&-\iint_{L^-} A(x,t,u,Du)\cdot Du
\left(1+\frac{l-u}{d}\right)^{-1-\lambda}
\varphi^k
\,\mathrm {d}x\mathrm {d}t
=:T_1+T_2.
\end{split}\end{equation*}
We first consider the estimate for $T_2$. From the first structure condition \eqref{A}, we deduce that
 \begin{equation*}\begin{split}
 T_2&\leq-C_0\iint_{L^-}u^{m-1}|Du|^p\left(1+\frac{l-u}{d}\right)^{-1-\lambda}\varphi^k\,\mathrm {d}x\mathrm {d}t\\&=
  -C_0d^p\iint_{L^-}u^{m-1}|D\psi_-|^p\varphi^k\,\mathrm {d}x\mathrm {d}t.
 \end{split}\end{equation*}
 To estimate $T_1$, we use the second structure condition \eqref{A} to deduce that
  \begin{equation*}\begin{split}
 T_1&\leq \iint_Q (|Du|^{p-1}u^{m-1}+u^{\frac{m-1}{p}}g)\varphi^{k-1} |D\varphi|
\left(\int_u^l\left(1+\frac{l-s}{d}\right)^{-1-\lambda}
\,\mathrm {d}s\right)_+
\,\mathrm {d}x\mathrm {d}t
\\&=:T_3+T_4,
 \end{split}\end{equation*}
 with the obvious meanings of $T_3$ and $T_4$. Next, we consider the estimate for $T_3$.
We use \eqref{phi} and Young's inequality to conclude that for any fixed
 $\epsilon>0$ there holds
  \begin{equation*}\begin{split}T_3\leq &\gamma d\iint_{L^-}|Du|^{p-1}u^{m-1}\varphi^{k-1} |D\varphi|
  \left(\frac{l-u}{d}\right)\left(1+\frac{l-u}{d}\right)^{-1}
  \,\mathrm {d}x\mathrm {d}t
%  \\ \leq &\epsilon\iint_{L^-}u^{m-1}|Du|^p
%\varphi^k\left(1+\frac{l-u}{d}\right)^{-(1-\lambda)}\left(\frac{l-u}{d}\right)^{-2\lambda}\,\mathrm {d}x\mathrm {d}t
%\\ &+c(\epsilon) d^p\iint_{L^-}\varphi^{k-p}|D\varphi|^pu^{m-1}\left(1+\frac{l-u}{d}\right)^{(1-\lambda)(p-1)}
%\,\mathrm {d}x\mathrm {d}t
  \\ \leq &\epsilon d^p\iint_{L^-}u^{m-1}|D\psi_-|^p\varphi^k\,\mathrm {d}x\mathrm {d}t
\\ &+\gamma d^p\iint_{L^-}\varphi^{k-p}|D\varphi|^pu^{m-1}\left(1+\frac{l-u}{d}\right)
^{\lambda(p-1)-1}\left(\frac{l-u}{d}\right)^{p}
\,\mathrm {d}x\mathrm {d}t
 \\ \leq &\epsilon d^p\iint_{L^-}u^{m-1}|D\psi_-|^p\varphi^k\,\mathrm {d}x\mathrm {d}t
\\ &+\gamma(\lambda,\epsilon) d^p\iint_{L^-}\varphi^{k-p}|D\varphi|^pu^{m-1}\left(\frac{l-u}{d}\right)^{(1+\lambda)(p-1)}
\,\mathrm {d}x\mathrm {d}t,
\end{split}\end{equation*}
since $\lambda<(p-1)^{-1}$.
To estimate $T_4$, we apply \eqref{phi}, $\lambda<(p-1)^{-1}$ and Young's inequality to conclude that
 \begin{equation*}\begin{split}T_4\leq &\gamma d\iint_{L^-}u^{\frac{m-1}{p}}g\varphi^{k-1} |D\varphi|
  \left(\frac{l-u}{d}\right)\left(1+\frac{l-u}{d}\right)^{-1}\,\mathrm {d}x\mathrm {d}t
  \\=&\gamma \iint_{L^-}\left[du^{\frac{m-1}{p}}\varphi^{\frac{k}{p}-1}  \left(\frac{l-u}{d}\right)\left(1+\frac{l-u}{d}\right)^{-\frac{1}{p}
 +\frac{\lambda}{p^\prime}} |D\varphi|\right]
  \\&
  \qquad\times\left[\left(1+\frac{l-u}{d}\right)^{-\frac{1}{p^\prime}-\frac{\lambda}{p^\prime}}g\varphi^{\frac{k}{p^\prime}}\right]
\,\mathrm {d}x\mathrm {d}t
  \\ \leq &\gamma d^p\iint_{L^-}\varphi^{k-p}|D\varphi|^pu^{m-1}\left(\frac{l-u}{d}\right)^{(1+\lambda)(p-1)}
\,\mathrm {d}x\mathrm {d}t
\\&+\gamma\theta\int_Bg^{\frac{p}{p-1}}
  \,\mathrm {d}x,
\end{split}\end{equation*}
where the constant $\gamma$ depends only upon the data.
Taking into account that $\varphi_-\leq \lambda^{-1} d$,
we obtain an estimate for the right-hand side of the equation by
\begin{equation*}\begin{split}
\iint_Qf\varphi_-\,\mathrm {d}x\mathrm {d}t\leq \gamma d\theta\int_B|f|
  \,\mathrm {d}x.
\end{split}\end{equation*}
Combining the above estimates and dividing by
$d^2$, we obtain the desired estimate \eqref{Cacformula1}. We have thus
proved the lemma.
\end{proof}
Furthermore, by substituting
 \begin{equation*}\begin{split}\varphi_+=\left[\int_l^u\left(1+\frac{s-l}{d}\right)^{-1-\lambda}\,\mathrm {d}s\right]_+
  \varphi^k=d\phi_+\left(\frac{u-l}{d}\right)\varphi^k\end{split}\end{equation*}
  into the weak formulation \eqref{weaksolution}, we get the following lemma, the proof of which we omit.
    \begin{lemma}\label{Cac2}
     Let $a>0$, $d>0$, $l>0$ and let
  $u$ be a nonnegative weak solution to the doubly degenerate parabolic equation \eqref{parabolic}
 in the sense of Definition \ref{weak solution}.
  Let $\rho>0$ be such that  $(t_0-a^{1-m}d^{2-p}\rho^p,t_0)\subset (-T,0)$ and $B_\rho(x_0)\subset\Omega$.
 There exists a positive constant $\gamma$ depending only upon the data, such that
 for every piecewise smooth cutoff function $\varphi$ vanishing on $\partial B_\rho(x_0)\times (t_0-a^{1-m}d^{2-p}\rho^p,t_0)$,
there holds
  \begin{equation}\begin{split}\label{Cacformula2}
  \esssup_{t_0-a^{1-m}d^{2-p}\rho^p<t<t_0}
  &\int_{L^+(t)}G\left(\frac{u-l}{d}\right)\varphi^k\,\mathrm {d}x+
  d^{p-2}\iint_{L^+}u^{m-1}|D\psi_+|^p\varphi^k\,\mathrm {d}x\mathrm {d}t
  \\
 \leq &\gamma\int_{L^+(t_0)}G\left(\frac{u-l}{d}\right)\varphi^k\,\mathrm {d}x+
 \gamma \iint_{L^+}\frac{u-l}{d}\left|\partial_t\varphi\right|\,\mathrm {d}x\mathrm {d}t
\\& +\gamma  d^{p-2}\iint_{L^+}u^{m-1}\left(\frac{u-l}{d}\right)^{(1+\lambda)(p-1)}
  \varphi^{k-p}|D\varphi|^p\,\mathrm {d}x\mathrm {d}t
\\&+\gamma \frac{\rho^p}{d^p}a^{1-m}\int_{B_\rho(x_0)}g^{\frac{p}{p-1}}\,\mathrm {d}x
  +\gamma \frac{\rho^p}{d^{p-1}}a^{1-m}\int_{B_\rho(x_0)}|f|\,\mathrm {d}x,
   \end{split}\end{equation}
     where $k>p$, $L^+=[B_\rho(x_0)\times (t_0-a^{1-m}d^{2-p}\rho^p,t_0)]\cap \{u\geq l\}$
     and $L^+(t)=B_\rho(x_0)\cap \{u(\cdot,t)\geq l\}$.
  \end{lemma}
  We now turn to consider the logarithmic estimates for nonnegative weak solutions. To this aim,
  we introduce the logarithmic function
  \begin{equation}\begin{split}\label{ln}\psi^{\pm}(u)&=\ln^+\left(\frac{H_k^{\pm}}{H_k^{\pm}-(u-k)_{\pm}+c}\right),
  \qquad 0<c<H_k^{\pm},
\end{split}\end{equation}
where $H_k^{\pm}$ is a constant chosen such that $H_k^{\pm}\geq \esssup_{Q_{\rho,\theta}}(u-k)_{\pm}$.
We are now in a position to state the following lemma.
\begin{lemma}Let
  $u$ be a nonnegative weak solution to the doubly degenerate parabolic equation \eqref{parabolic}
 in the sense of Definition \ref{weak solution}.
 There exists a  constant $\gamma$ that can be determined a priori only in terms of the data such that
 \begin{equation}\begin{split}\label{lnCac}
\esssup_{t_0-\theta<t<t_0}&\int_{B_\rho(x_0)\times\{t\}}[\psi^{\pm}(u)]^2\zeta^p\,\mathrm {d}x
\\ \leq &\int_{B_\rho(x_0)\times\{t_0-\theta\}}[\psi^{\pm}(u)]^2\zeta^p\,\mathrm {d}x+\gamma
\iint_{Q_{\rho,\theta}(z_0)}u^{m-1}\psi^{\pm}(u)|D\zeta|^p\left[(\psi^{\pm})^\prime(u)\right]^{2-p}
\,\mathrm {d}x\mathrm {d}t
\\&+\gamma\ln\left(\frac{H}{c}\right)\left(\frac{1}{c}\iint_{Q_{\rho,\theta}(z_0)}|f|\,\mathrm {d}x\mathrm {d}t+
\frac{1}{c^2}\iint_{Q_{\rho,\theta}(z_0)}g^{\frac{p}{p-1}}\,\mathrm {d}x\mathrm {d}t\right),
\end{split}\end{equation}
where $\zeta=\zeta(x)$ is independent of $t$ and satisfies \eqref{def zeta}.
\end{lemma}
\begin{proof}
The proof of \eqref{lnCac} follows in a similar manner as the arguments in \cite[Proposition 2.1]{PV} and we sketch the proof.
For simplicity of notation, we write $\psi$ instead of $\psi^{\pm}(u)$ and set $B=B_\rho(x_0)$.
In the weak formulation \eqref{weaksolution} we use the testing function
$\varphi=[\psi^2(u)]^\prime\zeta^p,$
where $\zeta=\zeta(x)\in  C_0^\infty(B_\rho)$.
The following formal
computations can be made rigorous with a standard smoothing procedure
with respect to time.
Then, we find
\begin{equation*}\begin{split}
\int_{t_0-\theta}^{t}\int_{B}\partial_tu[\psi^2(u)]^\prime\zeta^p\,\mathrm {d}x\mathrm {d}t
=\int_{B\times\{t\}}\psi^2(u)\zeta^p\,\mathrm {d}x-
\int_{B\times\{t_0-\theta\}}\psi^2(u)\zeta^p\,\mathrm {d}x.
 \end{split}\end{equation*}
Next, we consider the term involving $ A(x,t,u,Du)$. To this end, we decompose
  \begin{equation*}\begin{split}
   \int_{t_0-\theta}^{t}&\int_{B}\big\langle A(x,t,u,Du), D\varphi\big\rangle
\,\mathrm {d}x\mathrm {d}t
\\&=\int_{t_0-\theta}^{t}\int_{B} A(x,t,u,Du)[\psi^2(u)]^\prime p\zeta^{p-1}\cdot D\zeta\,\mathrm {d}x\mathrm {d}t
\\&+\int_{t_0-\theta}^{t}\int_{B} A(x,t,u,Du)\cdot Du [\psi^2(u)]^{\prime\prime} \zeta^p\,\mathrm {d}x\mathrm {d}t
=:T_1+T_2.
  \end{split}\end{equation*}
  We first estimate the second term $T_2$. Noting that $[\psi^2(u)]^{\prime\prime}
  =2(1+\psi)[\psi^\prime(u)]^2$, we infer from \eqref{A} that
\begin{equation*}\begin{split}
T_2\geq 2C_0\int_{t_0-\theta}^{t}\int_{B}u^{m-1}|Du|^p(1+\psi)[\psi^\prime(u)]^2  \zeta^p\,\mathrm {d}x\mathrm {d}t.
\end{split}\end{equation*}
To estimate $T_1$, we use \eqref{A} to deduce that
\begin{equation*}\begin{split}
T_1&\leq 2C_1\int_{t_0-\theta}^{t}\int_{B}(u^{m-1}|Du|^{p-1}+gu^{\frac{m-1}{p}})
\psi\psi^\prime\zeta^{p-1}|D\zeta|\,\mathrm {d}x\mathrm {d}t
\\&=:T_3+T_4,
\end{split}\end{equation*}
with the obvious meanings of $T_3$ and $T_4$. By Young's inequality, we conclude that
\begin{equation*}\begin{split}
T_3\leq&\epsilon \int_{t_0-\theta}^{t}\int_{B}u^{m-1}|Du|^p\zeta^p[\psi^\prime(u)]^2\psi\,\mathrm {d}x\mathrm {d}t
\\&+c(\epsilon)\int_{t_0-\theta}^{t}\int_{B}u^{m-1}\psi|D\zeta|^p[\psi^\prime(u)]^{2-p}\,\mathrm {d}x\mathrm {d}t,
\end{split}\end{equation*}
since $p-\frac{2p}{p^\prime}=2-p$. Next, we turn our attention to the estimate of $T_4$. By Young's inequality, we see that
\begin{equation*}\begin{split}
T_4=&2C_1\int_{t_0-\theta}^{t}\int_{B}\psi\left[\zeta^{p-1}\psi^\prime(u)^{\frac{2}{p^\prime}}g\right]
\left[|D\zeta|\psi^\prime(u)^{1-\frac{2}{p^\prime}}u^{\frac{m-1}{p}}\right]\,\mathrm {d}x\mathrm {d}t
\\ \leq&\gamma \int_{t_0-\theta}^{t}\int_{B}g^{p^\prime}\zeta^p[\psi^\prime(u)]^2\psi\,\mathrm {d}x\mathrm {d}t
\\&+\gamma \int_{t_0-\theta}^{t}\int_{B}u^{m-1}\psi|D\zeta|^p[\psi^\prime(u)]^{2-p}\,\mathrm {d}x\mathrm {d}t.
\end{split}\end{equation*}
Taking into account that $\psi\leq\ln (H/c)$ and $\psi^\prime\leq 1/c$, we conclude that
\begin{equation*}\begin{split}
\int_{t_0-\theta}^{t}\int_{B}g^{p^\prime}\zeta^p[\psi^\prime(u)]^2\psi\,\mathrm {d}x\mathrm {d}t
\leq \ln\left(\frac{H}{c}\right)
\frac{1}{c^2}\iint_{Q_{\rho,\theta}(z_0)}g^{\frac{p}{p-1}}\,\mathrm {d}x\mathrm {d}t
\end{split}\end{equation*}
and
\begin{equation*}\begin{split}
\int_{t_0-\theta}^{t}\int_{B}f\varphi\,\mathrm {d}x\mathrm {d}t
\leq \ln\left(\frac{H}{c}\right)
\frac{1}{c}\iint_{Q_{\rho,\theta}(z_0)}|f|\,\mathrm {d}x\mathrm {d}t.
\end{split}\end{equation*}
Collecting above estimates and taking the supremum over $t\in(t_0-\theta,t_0)$,
we obtain the desired logarithmic estimate \eqref{lnCac}.
\end{proof}
Next, we state in Lemma \ref{caclemma} a Caccioppoli estimate, which will be used in the proof of Lemma \ref{DeGiorgi3}.
\begin{lemma}\label{caclemma}
Let
  $u$ be a nonnegative weak solution to the doubly degenerate parabolic equation \eqref{parabolic}
 in the sense of Definition \ref{weak solution}.
 There exists a positive constant
$\gamma$ depending only upon the data, such that for every piecewise smooth cutoff function
$\varphi$ and satisfying \eqref{def zeta}, there holds
\begin{equation}\begin{split}\label{Caccioppoli}
\esssup_{t_0-\theta<t<t_0}&\int_{B_\rho(x_0)\times\{t\}}(u-k)_{\pm}^2\varphi^p \,\mathrm {d}x
+\iint_{Q_{\rho,\theta}(z_0)}u^{m-1}|D(u-k)_{\pm}\varphi|^p \,\mathrm {d}x\mathrm {d}t\\
\leq &\int_{B_\rho(x_0)\times\{t_0\}}(u-k)_{\pm}^2\varphi^p \,\mathrm {d}x+\gamma\iint_{Q_{\rho,\theta}(z_0)
} u^{m-1}(u-k)_{\pm}^p|D\varphi|^p
\,\mathrm {d}x\mathrm {d}t\\&+\gamma
\iint_{Q_{\rho,\theta}(z_0)} (u-k)_{\pm}^2|\partial_t\varphi|\,\mathrm {d}x\mathrm {d}t+\gamma
\iint_{Q_{\rho,\theta}(z_0)} |f|(u-k)_{\pm}\,\mathrm {d}x\mathrm {d}t\\&+\gamma
\iint_{Q_{\rho,\theta}(z_0)} g^{\frac{p}{p-1}}\,\mathrm {d}x\mathrm {d}t.
\end{split}\end{equation}
\end{lemma}
This is a standard result that can be proved by choosing the testing function $\pm(u-k)_{\pm}\varphi^p$
into \eqref{weaksolution} and the proof will be omitted.

We now turn our attention to the proof of Theorem \ref{main}.
The continuity of the weak solution $u$ at a point $\mathfrak z_0$
will be a consequence of the following assertion. There exists a family of nested and shrinking
cylinders $Q_{\rho_n,\theta_n}(\mathfrak z_0)$,
with vertex at $\mathfrak z_0$,
such that the essential oscillation of $u$ in $Q_{\rho_n,\theta_n}(\mathfrak z_0)$ converges to
zero as the cylinders shrink to
the point $\mathfrak z_0$.

There is no loss of generality in assuming $\mathfrak z_0=(0,0)$.
Let $\epsilon_0>0$ and $\hat R>0$ be such that $Q_{\hat R,\hat R^{p-\epsilon_0}}(0)\subset\Omega_T$. For any $0<R<\hat R$,
we fix a cylinder $Q_R=B_R\times (-R^{p-\epsilon_0},0)$ and set
\begin{equation*}\begin{split}
\mu_+=\esssup_{Q_R}u,\qquad \mu_-=\essinf_{Q_R}u\qquad\text{and}\qquad\essosc_{Q_R}u=\mu_+-\mu_-.
\end{split}\end{equation*}
Henceforth,
let $\omega>0$ be a parameter such that $\frac{1}{2}\omega\leq\mu_+-\mu_-\leq\omega$.
Here  we assume that $\mu_+>0$, since otherwise, if $\mu_+=0$, then $\omega=0$ and Theorem \ref{main}
holds trivially.
Let $A>1$ be a constant which will be determined later.
We introduce the enlarged cylinder
\begin{equation}\label{hat Q}\hat Q=B_R\times\left(-(\mu_+)^{1-m}\left(\frac{\omega}{A}\right)^{2-p}R^p,0\right).\end{equation}
In the case that $\hat Q$ is not contained in $\widehat Q_R=B_R\times (-\frac{1}{2}R^{p-\epsilon_0},0)$ for all $0<R<\hat R$, we
conclude that
\begin{equation}\begin{split}\label{cylindercontain}
-(\mu_+)^{1-m}\left(\frac{\omega}{A}\right)^{2-p}\leq -\frac{1}{2}R^{-\epsilon_0}.
\end{split}\end{equation}
Since $u$ is nonnegative, we have $\omega\leq2\mu_+$ and the inequality \eqref{cylindercontain} implies that
\begin{equation}\begin{split}\label{cylindernotcontain}
\essosc_{Q_R}u\leq\omega \leq 2^{\frac{m}{m+p-3}}A^{\frac{p-2}{m+p-3}}R^{\frac{\epsilon_0}{m+p-3}}
\end{split}\end{equation}
holds for all $0<R<\hat R$.
In this case the weak solution is H\"older continuous and Theorem \ref{main} holds immediately.
Next,
we concentrate on the case $\hat Q\subset \widehat Q_R$ for a fixed $R\in(0,\hat R)$.
For any fixed time level $-(A^{p-2}-1)(\mu_+)^{1-m}\omega^{2-p}R^p\leq t\leq0$ we introduce
the intrinsic cylinder of the type
\begin{equation*}\begin{split}
Q_r^-(t)=B_r\times(t-(\mu_+)^{1-m}\omega^{2-p}r^p,t),
\end{split}\end{equation*}
where $0<r\leq R$.
Motivated by the work of DiBenedetto \cite[chapter III]{Di93}, we consider two complementary cases.
For a fixed constant $\nu_0>0$, we see that either
  \begin{itemize}
 \item[$\bullet$]
 \textbf{The first alternative}. There exists $-(A^{p-2}-1)(\mu_+)^{1-m}\omega^{2-p}R^p\leq \bar t\leq0$ such that
\begin{equation}\label{1st}\left|\left\{(x,t)\in Q_{\frac{3}{4}R}^-(\bar t)
:u<\mu_-+\frac{\omega}{4}\right\}\right|\leq \nu_0|Q_{\frac{3}{4}R}^-(\bar t)|\end{equation}
\end{itemize}
or this does not hold. Then, we have
  \begin{itemize}
 \item[$\bullet$]
 \textbf{The second alternative}. For any $-(A^{p-2}-1)(\mu_+)^{1-m}\omega^{2-p}R^p\leq \bar t\leq0$, there holds
\begin{equation}\label{2nd}\left|\left\{(x,t)\in Q_{\frac{3}{4}R}
^-(\bar t):u>\mu_+-\frac{\omega}{4}\right\}\right|\leq (1-\nu_0)|Q_{\frac{3}{4}R}^-(\bar t)|.\end{equation}
\end{itemize}
The constant $\nu_0>0$ will be determined in the course of the proof of Lemma \ref{lemmaDeGiorgi1} in Sect.3,
while the value of $A$ will be
fixed during the proof of Proposition \ref{2nd proposition} in Sect. 4.
  \section{The first alternative}
  The aim of this section is to establish a decay estimate of the essential oscillation for the first alternative.
  A key ingredient in the proof of the decay estimate is
 a De Giorgi type lemma.
 Before we prove this result,
 we first recall the following definition of the Lebesgue point.
\begin{definition}\label{def lebesgue}\cite{EG} Let $f$ be a locally integrable function in $\Omega_T$. A point $(x,t)$ for which
\begin{equation*}\lim_{r\to0}\frac{1}{2r^p|B_r|}\ \int_{t-r^p}
^{t+r^p}\int_{B_r(x)}|f(y,s)-f(x,t)| \,\mathrm {d}y \,\mathrm {d}s=0\end{equation*}
holds is called a Lebesgue point of $f$.
\end{definition}
We now state and prove a couple of lemmas. The first lemma can be deduced from
\cite[section 1.3]{S} and \cite[Theorem 1.33]{EG}.
\begin{lemma}\label{Lebesguepoint} If $f\in L_{\loc}(\Omega_T)$, then almost every point in $\Omega_T$ is a Lebesgue point
of $f$.
\end{lemma}
In applications,
the cylinders used here are not centered
on the Lebesgue point, that is in fact the vertex.
Moreover, we need the following lemma, which states that the Lebesgue point for $f$ is also a Lebesgue point for $f_+$.
\begin{lemma}\label{Lebesguepoint+}
Let $f\in L_{\loc}(\Omega_T)$ and $z_1=(x_1,t_1)\in\Omega_T$ be a Lebesgue point for $f$. Then, we have
\begin{equation}
\label{Lebesguepointf+}f^+(z_1)=\lim_{r\to0}\frac{1}{|Q_{r,r^p}(z_1)|}\ \iint_{Q_{r,r^p}(z_1)}f^+(y,s)
\,\mathrm {d}y \,\mathrm {d}s,\end{equation}
where $Q_{r,r^p}(z_1)=B_r(x_1)\times(t_1-r^p,t_1)$.
\end{lemma}
\begin{proof}We first observe that
$|f^+(y,s)-f^+(x,t)|\leq |f(y,s)-f(x,t)|$.
Consequently, we conclude that
\begin{equation*}\begin{split}
\lim_{r\to0}&\frac{1}{|Q_{r,r^p}(z_1)|}\ \iint_{Q_{r,r^p}(z_1)}|f^+(y,s)-f^+(x_1,t_1)| \,\mathrm {d}y \,\mathrm {d}s
\\&\leq
2\lim_{r\to0}
\frac{1}{2r^p|B_r|}\ \int_{t-r^p}
^{t+r^p}\int_{B_r(x)}|f^+(y,s)-f^+(x_1,t_1)| \,\mathrm {d}y \,\mathrm {d}s=0,
\end{split}\end{equation*}
which proves \eqref{Lebesguepointf+}. This completes the proof.
\end{proof}
For technical reasons, we need the following result concerning the scaling property for function $G$ defined in \eqref{G}.
\begin{lemma}\label{Gk}
Let $\epsilon_1>0$ and $k\geq1$ be the fixed constants.
If $v\geq \epsilon_1$, then
\begin{equation}\label{Gkv}G(kv)\leq c(k,\epsilon_1)G(v).\end{equation}
\end{lemma}
\begin{proof}
If $v\leq1$ and $kv\leq1$, then $G(kv)=k^2v^2=k^2G(v)$.
If $v\geq 1$ and $kv\geq1$, then $G(kv)=kv=kG(v)$. Finally,
if $\epsilon_1\leq v\leq1$ and $kv>1$, then $G(kv)=\frac{k}{v}v^2=\frac{k}{v}G(v)\leq\frac{k}{\epsilon_1}G(v)$.
\end{proof}
With the help of these lemmas, we can now prove the following De Giorgi type lemma.
We adopt the Kilpel\"ainen-Mal\'y technique from \cite{KM,LS}. For the treatment of the doubly nonlinear problem,
the Kilpel\"ainen-Mal\'y technique should be suitably modified to handle the cutoff functions.
 \begin{lemma}\label{lemmaDeGiorgi1}
 Let $u$ be a bounded nonnegative weak solution to \eqref{parabolic}-\eqref{A} in $\Omega_T$.
 There exist constants $\nu_0\in(0,1)$ and $B>1$, depending only on the data, such that if
\begin{equation}\label{1st assumption}\left|\left\{(x,t)\in Q_{\frac{3}{4}R}^-(\bar t):u<\mu_-+\frac{\omega}{4}\right\}\right
|\leq \nu_0| Q_{\frac{3}{4}R}^-(\bar t)|,\end{equation}
then either
\begin{equation}
\label{DeGiorgi1}u(x,t)>\mu_-+\frac{\omega}{2^5}\qquad\text{for}\ \ \text{a.e.}\ \ (x,t)\in Q_{\frac{1}{4}R}^-(\bar t)\end{equation}
or
\begin{equation}
\label{omega1}\omega\leq B\left((\mu_+)^{\frac{1-m}{p}}F_1(2R)+(\mu_+)^{\frac{1-m}{p-1}}F_2(2R)\right).\end{equation}
 \end{lemma}
 \begin{proof}
 Let $B>1$ to be determined in the course of the proof.
 We first assume that \eqref{omega1} is violated, that is,
 \begin{equation}
\label{omega1violated}\omega>\frac{3}{4}B\left(\frac{1}{3B}\omega+(\mu_+)^{\frac{1-m}{p}}F_1(2R))+(\mu_+)^{\frac{1-m}{p-1}}F_2(2R)
\right).\end{equation}
If we can prove \eqref{DeGiorgi1} then the lemma follows immediately.
 Fix $(x_1,t_1)\in Q_{\frac{1}{4}R}^-(\bar t)$ and assume that $(x_1,t_1)$ is a Lebesgue point of $u$.
Let $r_j=4^{-j}C^{-1}R$
 and $B_j=B_{r_j}(x_1)$ where $C>4$ is to be determined.
For a sequence $\{l_j\}_{j=0}^\infty$ and a fixed $l>0$, we set
 \begin{equation*}Q_j(l)=B_j\times (t_1-(\mu_+)^{1-m}(l_j-l)^{2-p}r_j^p,t_1).\end{equation*}
Moreover, we define $\varphi_j(l)=\phi_j(x)\theta_{j,l}(t)$, where
$\phi_j\in C_0^\infty(B_j)$, $\phi_j=1$ on $B_{j+1}$, $|D\phi_j|\leq r_j^{-1}$
 and $\theta_{j,l}(t)$ is a Lipschitz function
satisfies
 \begin{equation*}
 \theta_{j,l}(t)=1\qquad\text{in}\qquad t\geq t_1-\frac{4}{9}(\mu_+)^{1-m}(l_j-l)^{2-p}r_j^p,
 \end{equation*}
  \begin{equation*}
 \theta_{j,l}(t)=0\qquad\text{in}\qquad t\leq t_1-\frac{5}{9}(\mu_+)^{1-m}(l_j-l)^{2-p}r_j^p
 \end{equation*}
and
  \begin{equation*}
 \theta_{j,l}(t)=\frac{t-t_1-\frac{5}{9}(\mu_+)^{1-m}(l_j-l)^{2-p}r_j^p}{\frac{1}{9}(\mu_+)^{1-m}(l_j-l)^{2-p}r_j^p}
 \end{equation*}
in
$t_1-\frac{5}{9}(\mu_+)^{1-m}(l_j-l)^{2-p}r_j^p\leq t\leq t_1-\frac{4}{9}(\mu_+)^{1-m}(l_j-l)^{2-p}r_j^p$.
From the definition of $\varphi_j(l)$, we see that $\varphi_j(l)=0$ on $\partial_PQ_j(l)$.
 Next, for  $j=-1,0,1,2,\cdots$, we define the sequence $\{\alpha_j\}$ by
  \begin{equation}\begin{split}\label{alpha}\alpha_j=&\frac{4^{-j-100
  }}{3B}\omega+ \frac{3}{4}\int_0^{r_j}\left(r^{p-n}(\mu_+)^{1-m}\int_{B_r(x_1)}
  g(y)^{\frac{p}{p-1}} \,\mathrm {d}y
  \right)^{\frac{1}{p}}\frac{\mathrm {d}r}{r}\\&+
  \frac{3}{4}\int_0^{r_j}\left(r^{p-n}(\mu_+)^{1-m}\int_{B_r(x_1)}|f(y)| \,\mathrm {d}y
  \right)^{\frac{1}{p-1}}\frac{\mathrm {d}r}{r}.
  \end{split}\end{equation}
  From the definition of $\alpha_j$, we see that $\alpha_j\to 0$ as $j\to\infty$ and there holds $B\alpha_{j-1}\leq\omega$,
  \begin{equation}\begin{split}\label{alpha1}
  \alpha_{j-1}-\alpha_j\geq &\frac{4^{-j-100}}{3B}\omega+\gamma\left(r_j^{p-n}(\mu_+)^{1-m}\int_{B_j} g(y)^{\frac{p}{p-1}}
  \,\mathrm {d}y\right)^{\frac{1}{p}}\\&+\gamma\left(r_j^{p-n}(\mu_+)^{1-m}\int_{B_j}|f(y)|
  \,\mathrm {d}y\right)^{\frac{1}{p-1}}
  \end{split}\end{equation}
  and
   \begin{equation}\begin{split}\label{alpha2}
  \alpha_{j-1}-\alpha_j\leq & \frac{4^{-j-100}}{3B}\omega+\gamma\left(r_{j-1}^{p-n}
  (\mu_+)^{1-m}\int_{B_{j-1}} g(y)^{\frac{p}{p-1}}
  \,\mathrm {d}y\right)^{\frac{1}{p}}\\&+\gamma\left(r_{j-1}^{p-n}(\mu_+)^{1-m}\int_{B_{j-1}}|f(y)|
  \,\mathrm {d}y\right)^{\frac{1}{p-1}}
  \end{split}\end{equation}
  for all $j=0,1,2,\cdots$, where the constant $\gamma$ depends only upon the data.
 Moreover, we define a quantity $A_j(l)$ by
 \begin{equation}\begin{split}\label{A_j}
 A_j(l)=&\frac{(\mu_+)^{m-1}(l_j-l)^{p-2}}{r_j^{n+p}}\iint_{L_j(l)}\left(\frac{l_j-u}{l_j-l}\right)^{(1+\lambda)(p-1)}\varphi_j(l)^{k-p}
 \,\mathrm {d}x\mathrm {d}t
 \\&+\esssup_t\frac{1}{r_j^n}\int_{B_j\times\{t\}}G\left(\frac{l_j-u}{l_j-l}\right)\varphi_j(l)^{k}
 \,\mathrm {d}x,
 \end{split}\end{equation}
 where $k>p$, $G$ is defined in \eqref{G} and $L_j(l)=Q_j(l)\cap \{u\leq l_j\}\cap \Omega_T$.
 The proof of \eqref{DeGiorgi1} will be divided into several steps.

Step 1: \emph{We claim that $A_j(l)$ is continuous in $l<l_j$.}
 To prove this assertion,
 it suffices to show that the function
 \begin{equation*}\begin{split}
 B_j(l):=
 \esssup_t\theta_{j,l}(t)^k\frac{1}{r_j^n}\int_{B_j}G\left(\frac{l_j-u}{l_j-l}\right)\phi_j^{k}
 \,\mathrm {d}x
 \end{split}\end{equation*}
 is continuous. For a fixed $l<l_j$, we take $|l^\prime-l|<\delta<\frac{1}{2}(l_j-l)$. Observe that $G$ is Lipschitz continuous,
 we have
  \begin{equation*}\begin{split}
  \big|G\left(\frac{l_j-u}{l_j-l}\right)-G\left(\frac{l_j-u}{l_j-l^\prime}\right)\big|\leq \gamma
  (l_j-u)_+\frac{|l-l^\prime|}{(l_j-l)(l_j-l^\prime)}
  \leq \gamma l_j\frac{|l-l^\prime|}{(l_j-l)^2},
   \end{split}\end{equation*}
   since $l_j-l^\prime=l_j-l+l-l^\prime>\frac{1}{2}(l_j-l)$. It  follows that
   \begin{equation*}\begin{split}
   |B_j(l)-B_j(l^\prime)|\leq &\gamma\frac{|l-l^\prime|}{(l_j-l)^2}l_j
 + \esssup_t\frac{1}{r_j^n}\int_{B_j}G\left(\frac{l_j-u}{l_j-l}\right)\phi_j^{k}
 \,\mathrm {d}x \left(\esssup_t|\theta_{j,l}(t)^k-\theta_{j,l^\prime}(t)^k|\right)
   \end{split}\end{equation*}
   and this implies that the function $B_j(l)$ is continuous for $l<l_j$.

 Step 2: \emph{Determine the values of $l_0$ and $l_1$.} Initially, we set $l_0=\mu_-+\frac{1}{4}\omega$ and $\bar l=\frac{1}{2}l_0+\frac{1}{2}\mu_-+\frac{1}{16}
 B\alpha_0+\frac{1}{32}\omega$.
 Recalling that $B\alpha_0< \omega$, we deduce
 \begin{equation}\begin{split}\label{l0}
 l_0-\bar l=\frac{1}{8}\omega-\frac{1}{16}B\alpha_0-\frac{1}{32}\omega\geq \frac{1}{32}\omega
 \quad\text{and}\quad  l_0-\bar l\leq \frac{3}{32}\omega.
 \end{split}\end{equation}
 Then, we have
 $(\mu_+)^{1-m}(l_0-\bar l)^{2-p}r_0^p\leq  C^{-p}2^{5(p-2)}(\mu_+)^{1-m}\omega^{2-p}R^p<\frac{1}{100^p}(\mu_+)^{1-m}\omega^{2-p}R^p$,
 provided that we choose $C=2^{\frac{5(p-2)}{p}+100}$.
 This guarantees the inclusion of $Q_0(\bar l)\subset Q_{\frac{3}{4}R}
  ^-(\bar t)$. In  view of $l_0-u\leq \frac{1}{4}\omega$ in $L_0(\bar l)$, we infer from \eqref{l0} and \eqref{1st assumption} that
  the inequality
   \begin{equation}\begin{split}\label{initial estimate}
   &\frac{(\mu_+)^{m-1}(l_0-\bar  l)^{p-2}}{r_0^{n+p}}\iint_{L_0(\bar l)}\left(\frac{l_0-u}{l_0-\bar l}
   \right)^{(1+\lambda)(p-1)}\varphi_0(\bar l)^{k-p}
 \,\mathrm {d}x\mathrm {d}t
 \\&\leq\frac{(\mu_+)^{m-1}(l_0-\bar l)^{p-2-(1+\lambda)(p-1)}}{r_0^{n+p}}\left(\frac{\omega}{4}\right)^{(1+\lambda)(p-1)}|L_0(\bar l)|
 \\&\leq \gamma_0 C^{n+p}\frac{(\mu_+)^{m-1}\omega^{p-2}}{R^{n+p}}|L_0(\bar l)|
= \gamma_0 C^{n+p}\frac{\left|Q_0(\bar l)\cap \left\{u\leq l_0\right\}\right|}{|Q_{\frac{3}{4}R}^-(\bar t)|}
  \\&\leq \gamma_0 C^{n+p}\frac{\left|Q_{\frac{3}{4}R}^-(\bar t)\cap \left
  \{u\leq \mu_-+\frac{\omega}{4}\right\}\right|}{|Q_{\frac{3}{4}R}^-(\bar t)|}
   \leq \gamma_0 C^{n+p}\nu_0
    \end{split}\end{equation}
    holds for a constant $\gamma_0$ depending only upon the data.
    Moreover, we apply Lemma \ref{Cac1} with $(l,d,\theta)$ replaced by  $(l_0,l_0-\bar l,(\mu_+)^{1-m}(l_0-\bar l)^{2-p}r_0^p)$ to obtain
     \begin{equation*}\begin{split}
    \esssup_t& \frac{1}{r_0^n}\int_{B_0\times\{t\}}G\left(\frac{l_0-u}{l_0-\bar l}\right)\varphi_0(\bar l)^{k}
 \,\mathrm {d}x \\ \leq &\gamma\frac{(\mu_+)^{m-1}(l_0-\bar l)^{p-2}}{r_0^{p+n}}\iint_{L_0(\bar l)}\left(\frac{l_0-u}{l_0-\bar l}
 \right)^{(1+\lambda)(p-1)}
 \varphi_0(\bar l)^{k-p}\,\mathrm {d}x\mathrm {d}t
 \\&+\gamma\frac{1}{r_0^n} \iint_{L_0(\bar l)}\frac{l_0-u}{l_0-\bar l}|\partial_t\varphi_0(\bar l)|\,\mathrm {d}x\mathrm {d}t
   \\&+\gamma \frac{r_0^{p-n}}{(l_0-\bar l)^p}(\mu_+)^{1-m}\int_{B_0}g^{\frac{p}{p-1}}\,\mathrm {d}x
  +\gamma \frac{r_0^{p-n}}{(l_0-\bar l)^{p-1}}(\mu_+)^{1-m}\int_{B_0}|f|\,\mathrm {d}x,
 \end{split}\end{equation*}
 since $u\leq \mu_+$ in $L_0(\bar l)$, $|D\varphi_0(\bar l)|\leq r_0^{-1}$, $\varphi_0(\bar l)=0$ on $\partial_PQ_0(\bar l)$
 and the first term on the right-hand side of \eqref{Cacformula1} vanishes.
In view of $l_0-u\leq \frac{1}{4}\omega$ in $L_0(\bar l)$
 and $|\partial_t\varphi_0(\bar l)|\leq 9(l_0-\bar l)^{p-2}(\mu_+)^{m-1}r_0^{-p}$, we infer from \eqref{l0} and \eqref{initial estimate}
 that
\begin{equation*}\begin{split}
\frac{1}{r_0^n}& \iint_{L_0(\bar l)}\frac{l_0-u}{l_0-\bar l}|\partial_t\varphi_0(\bar l)|\,\mathrm {d}x\mathrm {d}t
\\&\leq \gamma\frac{(\mu_+)^{m-1}(l_0-\bar  l)^{p-3}}{r_0^{n+p}}\iint_{L_0(\bar l)}(l_0-u)
 \,\mathrm {d}x\mathrm {d}t
 \\&\leq \gamma C^{n+p}\frac{(\mu_+)^{m-1}\omega^{p-2}}{R^{n+p}}|L_0(\bar l)|
   \leq \gamma C^{n+p}\nu_0,
\end{split}\end{equation*}
where the constant $\gamma$ depends only upon the data. Furthermore, we infer from \eqref{l0} and \eqref{omega1violated} that
there exists a constant $\gamma$ depending only upon the data, such that
\begin{equation*}\begin{split}
\frac{r_0^{p-n}}{(l_0-\bar l)^p}&(\mu_+)^{1-m}\int_{B_0}g^{\frac{p}{p-1}}\,\mathrm {d}x
  +\frac{r_0^{p-n}}{(l_0-\bar l)^{p-1}}(\mu_+)^{1-m}\int_{B_0}|f|\,\mathrm {d}x
  \\&\leq \gamma\frac{r_0^{p-n}}{\omega^p}(\mu_+)^{1-m}\int_{B_0}g^{\frac{p}{p-1}}\,\mathrm {d}x
  +\gamma\frac{r_0^{p-n}}{\omega^{p-1}}(\mu_+)^{1-m}\int_{B_0}|f|\,\mathrm {d}x
  \\&\leq \gamma C^{n-p}(B^{-p}+B^{1-p}).
  \end{split}\end{equation*}
  Collecting above estimates, we conclude with
 \begin{equation*}\begin{split}
   & \esssup_t \frac{1}{r_0^n}\int_{B_0\times\{t\}}G\left(\frac{l_0-u}{l_0-\bar l}\right)\varphi_0(\bar l)^{k}
 \,\mathrm {d}x
 \leq \gamma_1C^{n+p}\nu_0+\gamma_1C^{n-p}(B^{-p}+B^{1-p}),
  \end{split}\end{equation*}
  where $\gamma_1=\gamma_1(\text{data})$.
  Consequently, we arrive at
   \begin{equation*}\begin{split}
  A_0(\bar l)\leq (\gamma_0+\gamma_1)C^{n+p}\nu_0+\gamma_1C^{n-p}(B^{-p}+B^{1-p}).
    \end{split}\end{equation*}
    At this stage, we  fix a number $\chi\in(0,1)$ which will be determined in the course of the proof.
    Then, we choose $\nu_0=\nu_0(\text{data},\chi)<1$ and $B=B(\text{data},\chi)>1$ be such that
     \begin{equation}\begin{split}\label{nu0}
  (\gamma_0+\gamma_1)C^{n+p}\nu_0=\frac{\chi}{4}\qquad\text{and}\qquad\gamma_1C^{n-p}(B^{-p}+B^{1-p})<\frac{\chi}{4}.
    \end{split}\end{equation}
    This implies that $A_0(\bar l)\leq\frac{1}{2}\chi$. Our task now is to determine the value of $l_1$.
    We first consider the case
    \begin{equation*}\iint_{B_{r_0}(x_1)\times\{t_1-r_0^p<t<t_1\}}(l_0-u)_+^{(1+\lambda)(p-1)}\,\mathrm {d}x\,\mathrm {d}t=0.\end{equation*}
According to Lemma \ref{Lebesguepoint+}, we conclude that $u(x_1,t_1)\geq l_0=\mu_-+\frac{1}{4}\omega$, which proves
    the desired estimate \eqref{DeGiorgi1}.
    In the case
     \begin{equation*}\iint_{B_{r_0}(x_1)\times\{t_1-r_0^p<t<t_1\}}(l_0-u)_+^{(1+\lambda)(p-1)}\,\mathrm {d}x\,\mathrm {d}t>0,\end{equation*}
     we see that $A_0(l)\to+\infty$
    as $l\to l_0$.
    Noting that $A_0(l)$ is continuous and increasing, then there exists a number $\tilde l\in (\bar l, l_0)$ such that $A_0(\tilde l)=\chi$.
    From \eqref{l0} and  \eqref{omega1violated}, we infer that for $B>8$ there holds
    \begin{equation}\label{l0barl}l_0-\bar l\geq \frac{1}{32}\omega\geq \frac{1}{4
    B}\omega>\frac{1}{4}(\alpha_{-1}-\alpha_0),\end{equation}
    since $B\alpha_{-1}<\omega$.
    At this point, we set
    \begin{equation}\label{l1}
	l_1=\begin{cases}
\tilde l,&\quad \text{if}\quad \tilde l<l_0-\frac{1}{4}(\alpha_{-1}-\alpha_0),\\
	l_0-\frac{1}{4}(\alpha_{-1}-\alpha_0),&\quad \text{if}\quad \tilde l\geq l_0-\frac{1}{4}(\alpha_{-1}-\alpha_0).
	\end{cases}
\end{equation}
Moreover, we define $Q_0=Q_0(l_1)$ and $d_0=l_0-l_1$. In view of \eqref{l0barl}, we see that the definition of $l_1$ is
justified. Since $B\alpha_0<\omega$, we have $l_1\geq  \bar l>\mu_-+\frac{1}{8}B\alpha_0+\frac{1}{16}\omega$.

   Step 3: \emph{Determine the sequence $\{l_j\}_{j=0}^{+\infty}$.} Assume that we have chosen two
   sequences $l_1,\cdots,l_j$ and $d_0,\cdots,d_{j-1}$ such that for $i=1,\cdots,j$, there holds
    \begin{equation}\label{li}\frac{1}{2}\mu_-+\frac{1}{32}\omega+\frac{1}{2}l_{i-1}+\frac{1}{16}B\alpha_{i-1}<l_i\leq
    l_{i-1}-\frac{1}{4}(\alpha_{i-2}-\alpha_{i-1}),
    \end{equation}
    \begin{equation}\label{Aj-1}
    A_{i-1}(l_i)\leq \chi,
     \end{equation}
      \begin{equation}\label{lj}
      l_i>\mu_-+\frac{1}{8}B\alpha_{i-1}+\frac{1}{16}\omega.
       \end{equation}
       Then, we set $Q_i=Q_i(l_{i+1})$, for $i=1,2,\cdots,j-1$ and claim that
       \begin{equation}\label{Aj}
       A_j(\bar l)\leq \frac{1}{2}\chi,\qquad\text{where}\qquad \bar l=\frac{1}{2}l_j+\frac{1}{16}B\alpha_j+\frac{1}{32}\omega+
       \frac{1}{2}\mu_-.
        \end{equation}
        To prove \eqref{Aj}, we first assert that the inclusion $Q_i\subset \hat Q$ holds for $i=0,1,\cdots,j-1$, where
        $\hat Q$  is defined in \eqref{hat Q}. In view of \eqref{alpha1}, we
        see that $\alpha_{i-1}-\alpha_i\geq\frac{1}{3B}4^{-i-100}\omega$. From \eqref{li}, we can verify the inequality
  \begin{equation*}\begin{split}
  (\mu_+)^{1-m}(l_{i}-l_{i+1})^{2-p}r_i^p&\leq  (\mu_+)^{1-m}4^{p-2}(\alpha_{i-1}-\alpha_i)^{2-p}r_i^p
 \\& \leq 4^{-2i}4^{101(p-2)}(\mu_+)^{1-m}(3B)^{p-2}\omega^{2-p}R^p\leq (\mu_+)^{1-m}\left(\frac{\omega}{A}\right)^{2-p}R^p,
  \end{split}\end{equation*}
  provided that we choose
  \begin{equation}\begin{split}\label{first condition for A}
  A>4^{102}B.
   \end{split}\end{equation}
  This implies the inclusion $Q_i\subset \hat Q$ for $i=0,1,\cdots,j-1$ and hence,
$\mu_-\leq u\leq \mu_+$ on $Q_i$ for $i=0,1,\cdots,j-1$.
 We now turn our attention to the proof of \eqref{Aj}. To start with,
   we introduce a cutoff function $\tilde u=\max\left\{u,\frac{1}{128} \omega\right\}$
   and decompose
   $L_j(\bar l)=L^\prime_j(\bar l)\cup L^{\prime\prime}_j(\bar l)$, where
   \begin{equation}\begin{split}\label{Ldecomposition}
   L^\prime_j(\bar l)=L_j(\bar l)\cap \left\{\frac{l_j-\tilde  u}{l_j-\bar l}\leq\epsilon_1\right\}\qquad\text{and}\qquad
   L^{\prime\prime}_j(\bar l)=L_j(\bar l)\setminus L^\prime_j(\bar l).
   \end{split}\end{equation}
   Moreover, we observe from \eqref{li} that $\{u\leq l_j\}=\{\tilde u\leq l_j\}$ and $l_j-\frac{1}{128}\omega> \frac{1}{2}l_j$.
   This implies  that
    \begin{equation}\begin{split}\label{cut}&\iint_{L_j(\bar l)}\left(\frac{l_j-u}{l_j-\bar l}
    \right)^{(1+\lambda)(p-1)}\varphi_j(\bar l)^{k-p}
 \,\mathrm {d}x\mathrm {d}t
 \\&\leq\iint_{L_j(\bar l)\cap\left\{u>\frac{1}{128}\omega\right\}}\left(\frac{l_j-\tilde u}{l_j-\bar l}
 \right)^{(1+\lambda)(p-1)}\varphi_j(\bar l)^{k-p}
 \\&\quad+\iint_{L_j(\bar l)\cap\left\{u\leq\frac{1}{128}\omega\right\}}
 \left(\frac{l_j}{l_j-\bar l}\right)^{(1+\lambda)(p-1)}\varphi_j(\bar l)^{k-p}
 \,\mathrm {d}x\mathrm {d}t
 \\&\leq\iint_{L_j(\bar l)\cap\left\{u>\frac{1}{128}\omega\right\}}\left(\frac{l_j-\tilde u}{l_j-\bar l}
 \right)^{(1+\lambda)(p-1)}\varphi_j(\bar l)^{k-p}
 \\&\quad+2^{(1+\lambda)(p-1)}\iint_{L_j(\bar l)\cap\left\{u\leq\frac{1}{128}\omega\right\}}
 \left(\frac{l_j-\frac{1}{128}\omega}{l_j-\bar l}\right)^{(1+\lambda)(p-1)}\varphi_j(\bar l)^{k-p}
 \,\mathrm {d}x\mathrm {d}t
 \\&\leq 2^{(1+\lambda)(p-1)}\iint_{L_j(\bar l)}\left(\frac{l_j-\tilde u}{l_j-\bar l}\right)^{(1+\lambda)(p-1)}\varphi_j(\bar l)^{k-p}
 \,\mathrm {d}x\mathrm {d}t
  \end{split}\end{equation}
  and we can rewrite $A_j(\bar l)$ by
    \begin{equation*}\begin{split}
 A_j(\bar l)\leq&\gamma\frac{(\mu_+)^{m-1}(l_j-\bar l)^{p-2}}{r_j^{n+p}}\iint_{L_j(\bar l)}
 \left(\frac{l_j-\tilde u}{l_j-\bar l}\right)^{(1+\lambda)(p-1)}\varphi_j(\bar l)^{k-p}
 \,\mathrm {d}x\mathrm {d}t
 \\&+\esssup_t\frac{1}{r_j^n}\int_{B_j\times\{t\}}G\left(\frac{l_j-u}{l_j-\bar l}\right)\varphi_j(\bar l)^{k}
 \,\mathrm {d}x,
 \end{split}\end{equation*}
 where the constant $\gamma$ depends only upon the data.
 Furthermore, we infer from \eqref{li} and  \eqref{lj} that
 \begin{equation}\begin{split}\label{upper bound for l}
 l_j-\bar l&
       =\frac{1}{2}l_j-\frac{1}{16}B\alpha_j-\frac{1}{32}\omega-
       \frac{1}{2}\mu_-
   \\&
        \geq
        \frac{1}{4}(l_{j-1}-l_j)+\frac{1}{4}l_j
      +\frac{1}{64}\omega+\dfrac{1}{32}B\alpha_{j-1} -\frac{1}{16}B\alpha_j-\frac{1}{32}\omega-
       \frac{1}{4}\mu_-
       \\&\geq
       \frac{1}{4}(l_{j-1}-l_j)+\frac{1}{32}\left(B\alpha_{j-1}+\frac{1}{2}\omega\right)
+\dfrac{1}{32}B\alpha_{j-1} -\frac{1}{16}B\alpha_j-\frac{1}{64}\omega
\\&=\frac{1}{4}(l_{j-1}-l_j)+\frac{1}{16}B(\alpha_{j-1}-\alpha_j).
 \end{split}\end{equation}
 For simplicity of notation, we write $\varphi_i=\varphi_i(l_{i+1})$
 and observe from \eqref{upper bound for l} that
 \begin{equation*}\begin{split}
 (\mu_+)^{1-m}(l_j-\bar l)^{2-p}r_j^p\leq (\mu_+)^{1-m}\frac{r_{j-1}^p}{4^p}\left(\frac{l_{j-1}-l_j}{4}\right)^{2-p}=\frac{1}{16}
 (\mu_+)^{1-m}(l_{j-1}-l_j)
 ^{2-p}r_{j-1}^p,
  \end{split}\end{equation*}
  which yields $Q_j(\bar l)\subset Q_{j-1}\subset \hat Q$ and  $\varphi_{j-1}(x,t)=1$ for $(x,t)\in Q_j(\bar l)$.
  Since $u\leq l_j$ on $L_j(\bar l)$, we use \eqref{Aj-1} to deduce
  \begin{equation}\begin{split}\label{0st estimate}
 &\frac{(\mu_+)^{m-1}(l_j-\bar l)^{p-2}}{r_j^{n+p}}|L_j(\bar l)|
 \leq \frac{1}{r_j^n}\esssup_t\int_{L_j(t)}
  \varphi_{j-1}^k(\cdot,t)
 \,\mathrm {d}x
  \\&\leq \frac{4^n}{r_{j-1}^n}\esssup_t\int_{B_{j-1}}G\left(\frac{l_{j-1}-u}{l_{j-1}-l_j}\right)
  \varphi_{j-1}^k
 \,\mathrm {d}x\leq 4^n\chi,
 \end{split}\end{equation}
 where $L_j(t)=\{x\in B_j:u(\cdot,t)\leq l_j\}$.
 In view of \eqref{0st estimate}, we conclude that
 \begin{equation}\begin{split}\label{1st estimate}
 &\frac{(\mu_+)^{m-1}(l_j-\bar l)^{p-2}}{r_j^{n+p}}\iint_{L_j^\prime(\bar l)}
 \left(\frac{l_j-\tilde u}{l_j-\bar l}\right)^{(1+\lambda)(p-1)}\varphi_j(\bar l)^{k-p}
 \,\mathrm {d}x\mathrm {d}t
 \\&\leq \frac{(\mu_+)^{m-1}(l_j-\bar l)^{p-2}}{r_j^{n+p}}\epsilon_1^{(1+\lambda)(p-1)}|L_j(\bar l)|
\leq 4^n\epsilon_1^{(1+\lambda)(p-1)}\chi.
 \end{split}\end{equation}
 Moreover, for any fixed  $\epsilon_2<1$, we apply the Young's inequality to obtain
  \begin{equation*}\begin{split}
 &\frac{(\mu_+)^{m-1}(l_j-\bar l)^{p-2}}{r_j^{n+p}}\iint_{L_j^{\prime\prime}(\bar l)}
 \left(\frac{l_j-\tilde u}{l_j-\bar l}\right)^{(1+\lambda)(p-1)}\varphi_j(\bar l)^{k-p}
 \,\mathrm {d}x\mathrm {d}t
 \\&\leq\epsilon_2\frac{(\mu_+)^{m-1}(l_j-\bar l)^{p-2}}{r_j^{n+p}}|L_j(\bar l)|
 \\&\quad+\gamma(\epsilon_2)\frac{(\mu_+)^{m-1}(l_j-\bar l)^{p-2}}{r_j^{n+p}}
 \iint_{L_j^{\prime\prime}(\bar l)}\left(\frac{l_j-\tilde u}{l_j-\bar l}\right)^{p\frac{n+h}{nh}}\varphi_j(\bar l)^{(k-p)q}
 \,\mathrm {d}x\mathrm {d}t
 \\&=:T_1+T_2,
  \end{split}\end{equation*}
  with the obvious meanings of $T_1$ and $T_2$. Here, we set
\begin{equation}\label{hq}h=\frac{p}{p-1-\lambda}>1\qquad\text{and}\qquad q=p\frac{n+h}{nh(1+\lambda)(p-1)}=\frac{p-1-
\lambda+\frac{p}{n}}{p-1+\lambda p-\lambda}>1,\end{equation}
since $\lambda<\min\left\{p-1,\frac{1}{n}\right\}$. According to \eqref{0st estimate}, we see that $T_1\leq 4^n\epsilon_2\chi$.
To estimate $T_2$, we set
   \begin{equation}\begin{split}\label{tildepsi}\tilde\psi_j(x,t)=\frac{1}{l_j-\bar l}\left[\int_{\tilde u}^{l_j}
   \left(1+\frac{l_j-s}{l_j-\bar l}\right)^{-\frac{1}{p}-\frac{\lambda}{p}}\,\mathrm {d}s\right]_+\end{split}\end{equation}
  and apply Lemma \ref{lemmainequalitypsi-} with $(l,d)$ replaced by $(l_j,l_j-\bar l)$ to conclude with
   \begin{equation}\begin{split}\label{T21}
   T_2\leq \gamma\frac{(\mu_+)^{m-1}(l_j-\bar l)^{p-2}}{r_j^{n+p}}\iint_{L_j^{\prime\prime}(\bar l)}
   \tilde\psi_j^{p\frac{n+h}{n}}\varphi_j(\bar l)^{(k-p)q} \,\mathrm {d}x\mathrm {d}t.
   \end{split}\end{equation}
   Let $v=\tilde\psi_j\varphi_j^{k_1}$, where $k_1=\frac{(k-p)nq}{p(n+h)}$. Recalling that $p<n$,
   we use H\"older's inequality with $r=\frac{n}{n-p}$ and $r^\prime=\frac{n}{p}$, Sobolev's inequality and Lemma \ref{lemmainequalitypsi-}
   with $(l,d)$ replaced by $(l_j,l_j-\bar l)$ to deduce
   \begin{equation}\begin{split}\label{T22}
\iint_{L_j^{\prime\prime}(\bar l)}&
   \tilde\psi_j^{p\frac{n+h}{n}}\varphi_j(\bar l)^{(k-p)q} \,\mathrm {d}x\mathrm {d}t=
   \iint_{L_j^{\prime\prime}(\bar l)}
v^{p+\frac{ph}{n}} \,\mathrm {d}x\mathrm {d}t
\\&\leq \int_{t_1- (\mu_+)^{1-m}(l_j-\bar l)^{2-p}r_j^p}^{t_1}\left(\int_{B_j}
v^{\frac{np}{n-p}} \,\mathrm {d}x\right)^{\frac{n-p}{n}}
\left(\int_{L_j^{\prime\prime}(t)}
v^{h} \,\mathrm {d}x\right)^{\frac{p}{n}}
\mathrm {d}t
\\&\leq  \gamma\esssup_t\left(\int_{L_j^{\prime\prime}(t)}
\frac{l_j-\tilde u}{l_j-\bar l}\varphi_j(\bar l)^{k_1h} \,\mathrm {d}x\right)^{\frac{p}{n}}\iint_{Q_j(\bar l)}
|Dv|^p \,\mathrm {d}x\mathrm {d}t,
   \end{split}\end{equation}
   where
   $$L_j^{\prime\prime}(t)=\{x\in B_j:u(\cdot,t)\leq l_j\}\cap \left\{x\in B_j:\frac{l_j-\tilde u(\cdot,t)}{l_j-\bar l}>
   \epsilon_1\right\}.$$
   According to Lemma \ref{Gk}, we see that
   \begin{equation}\begin{split}\label{T23}
&   \int_{L_j^{\prime\prime}(t)}
\frac{l_j-\tilde u}{l_j-\bar l}\varphi_j(\bar l)^{k_1h} \,\mathrm {d}x
\leq \int_{L_j^{\prime\prime}(t)}\epsilon_1
G\left(\frac{1}{\epsilon_1}\frac{l_j-\tilde u}{l_j-\bar l} \right)\varphi_j(\bar l)^{k_1h}\,\mathrm {d}x
\\&\leq c(\epsilon_1)
\int_{L_j^{\prime\prime}(t)}
G\left(\frac{l_j-\tilde u}{l_j-\bar l} \right)\varphi_j(\bar l)^{k_1h}\,\mathrm {d}x
\leq c(\epsilon_1)
\int_{B_j}
G\left(\frac{l_j-u}{l_j-\bar l} \right)\varphi_j(\bar l)^{k_1h}\,\mathrm {d}x.
    \end{split}\end{equation}
    At this point, we apply Lemma \ref{Cac1}
    with $(l,d,\theta)$ replaced by  $(l_j,l_j-\bar l,(\mu_+)^{1-m}(l_j-\bar l)^{2-p}r_j^p)$
    to conclude that
     \begin{equation*}\begin{split}
  \esssup_t&
  \frac{1}{r_j^n}
  \int_{B_j}
G\left(\frac{l_j-u}{l_j-\bar l} \right)\varphi_j(\bar l)^{k_1h}\,\mathrm {d}x
  \\
 \leq &
 \gamma  \frac{(l_j-\bar l)^{p-2}}{r_j^{p+n}}\iint_{L_j(\bar l)}u^{m-1}\left(\frac{l_j-u}{l_j-\bar l}\right)^{(1+\lambda)(p-1)}
  \varphi_j(\bar l)^{k_1h-p}\,\mathrm {d}x\mathrm {d}t
  \\&+\gamma
    \frac{1}{r_j^n}\iint_{L_j(\bar l)}\frac{l_j-u}{l_j-\bar l}|\partial_t\varphi_j(\bar l)|\,\mathrm {d}x\mathrm {d}t
   \\&+\gamma \frac{r_j^{p-n}}{(l_j-\bar l)^p}(\mu_+)^{1-m}\int_{B_j}g^{\frac{p}{p-1}}\,\mathrm {d}x
  +\gamma \frac{r_j^{p-n}}{(l_j-\bar l)^{p-1}}(\mu_+)^{1-m}\int_{B_j}|f|\,\mathrm {d}x
  \\=&:T_3+T_4+T_5+T_6,
   \end{split}\end{equation*}
   with the obvious meanings of $T_3$-$T_7$.
To estimate $T_3$,
we apply \eqref{upper bound for l} and \eqref{Aj-1} to conclude that
    \begin{equation}\begin{split}\label{T3}
    T_3\leq &\gamma\frac{(l_j-\bar l)^{p-2-(1+\lambda)(p-1)}(\mu_+)^{m-1}}{r_j^{p+n}}\iint_{L_j(\bar l)}
    (l_j-u)^{(1+\lambda)(p-1)}
  \varphi_j(\bar l)^{k_1h-p}\,\mathrm {d}x\mathrm {d}t
  \\ \leq &\gamma\frac{(l_{j-1}-l_j)^{p-2}(\mu_+)^{m-1}}{r_{j-1}^{p+n}}\iint_{L_{j-1}
  ( l_j)}\left(\frac{l_{j-1}-u}{l_{j-1}-l_j}\right)^{(1+\lambda)(p-1)}
  \varphi_{j-1}^{k-p}\,\mathrm {d}x\mathrm {d}t
  \\ \leq & \gamma A_{j-1}(l_j)\leq \gamma_1\chi,
  \end{split}\end{equation}
  where the constant  $\gamma_1$ depends only upon the data.
Noting that $(1+\lambda)(p-1)>1$ and
  $|\partial_t\varphi_j(\bar l)|\leq 9(l_j-\bar l)^{p-2}(\mu_+)^{m-1}r_j^{-p}$, we infer from \eqref{0st estimate}
  and \eqref{Aj-1} that the inequality
     \begin{equation*}\begin{split}
    T_4\leq &\gamma
    \frac{(l_j-\bar l)^{p-2}(\mu_+)^{m-1}}{r_j^{n+p}}\iint_{L_j(\bar l)}\frac{l_j-u}
    {l_j-\bar l}\,\mathrm {d}x\mathrm {d}t
    \\ \leq &\gamma\frac{(l_j-\bar l)^{p-2}(\mu_+)^{m-1}}{r_j^{n+p}}|L_j(\bar l)|
  +\gamma
    \frac{(l_j-\bar l)^{p-2}(\mu_+)^{m-1}}{r_j^{n+p}}\iint_{L_j(\bar l)}\left(\frac{l_j-u}
    {l_j-\bar l}\right)^{(1+\lambda)(p-1)}\,\mathrm {d}x\mathrm {d}t
  \\ \leq & \gamma\chi+A_{j-1}(l_j)\leq \gamma_2\chi,
    \end{split}\end{equation*}
holds for the constant  $\gamma_2$ depending only upon the data. Finally, we deduce from
    \eqref{alpha1} and \eqref{upper bound for l} that $T_5+T_6\leq \gamma (B^{-p}+B^{-(p-1)})$. Combining the above estimates,
    we infer that the estimate
    \begin{equation}\begin{split}\label{important G}
  \esssup_t&
  \frac{1}{r_j^n}
  \int_{B_j}
G\left(\frac{l_j-u}{l_j-\bar l} \right)\varphi_j(\bar l)^{k_1h}\,\mathrm {d}x\leq \gamma\chi+\gamma (B^{-p}+B^{-(p-1)})
    \end{split}\end{equation}
    holds for a constant $\gamma$ depending only upon the data.
   We now turn our attention to the estimate of $T_2$.
   With the help of \eqref{T21}-\eqref{important G}, we can rewrite the upper bound for $T_2$ by
    \begin{equation*}\begin{split}
    T_2\leq & \gamma \frac{(\mu_+)^{m-1}(l_j-\bar l)^{p-2}}{r_j^n}(\chi+B^{-p}+B^{-(p-1)})^{\frac{p}{n}}
    \\&\times\left[\iint_{Q_j(\bar l)}
\varphi_j(\bar l)^{k_1p}|D\tilde \psi_j|^p \,\mathrm {d}x\mathrm {d}t
+\iint_{Q_j(\bar l)}
\varphi_j(\bar l)^{(k_1-1)p}\tilde \psi_j^p|D\varphi_j|^p \,\mathrm {d}x\mathrm {d}t\right]
\\=&:\gamma(\chi+B^{-p}+B^{-(p-1)})^{\frac{p}{n}}(T_7+T_8),
    \end{split}\end{equation*}
    with the obvious meanings of $T_7$ and $T_8$. We first consider the  estimate for $T_7$. In the case $\mu_-\leq\frac{1}{2}\mu_+$.
    We have $\mu_+\leq\mu_-+\omega\leq \frac{1}{2}\mu_++\omega$, which implies that $\mu_+\leq2\omega$.
    Therefore, we deduce
    \begin{equation*}\tilde u\geq\frac{1}{128}\omega\geq\frac{1}{256}\mu_+.\end{equation*}
    In the case $\mu_->\frac{1}{2}\mu_+$.
    We get $\tilde u\geq u\geq\mu_->\frac{1}{2}\mu_+$ on $Q_j(\bar l)$. In both cases, we see that
    \begin{equation}\label{um-1}\tilde u^{m-1}\geq \left(\frac{1}{256}\mu_+\right)^{m-1}
    \qquad\text{on}\qquad Q_j(\bar l).\end{equation}
    In view of \eqref{um-1}, we conclude that
    \begin{equation*}\begin{split}
    T_7&=\frac{(\mu_+)^{m-1}(l_j-\bar l)^{p-2}}{r_j^n}
    \iint_{Q_j(\bar l)}
\varphi_j(\bar l)^{k_1p}|D\tilde \psi_j|^p \,\mathrm {d}x\mathrm {d}t
\\&\leq \gamma \frac{(l_j-\bar l)^{p-2}}{r_j^n}
    \iint_{Q_j(\bar l)}\tilde u^{m-1}
|D\tilde \psi_j|^p \varphi_j(\bar l)^{k_1p}\,\mathrm {d}x\mathrm {d}t.
    \end{split}\end{equation*}
    Taking into account that $\{\tilde u<l_j\}=\{u<l_j\}$ and $D\tilde u=0$ on $\left\{u\leq\frac{1}{128}\omega\right\}$,
    we conclude from \eqref{tildepsi} that
    \begin{equation*}\begin{split}
    \tilde u^{\frac{m-1}{p}}
|D\tilde \psi_j|&=\tilde u^{\frac{m-1}{p}}|D\tilde u|\frac{1}{l_j-\bar l}
   \left(1+\frac{l_j-\tilde u}{l_j-\bar l}\right)^{-\frac{1}{p}-\frac{\lambda}{p}}
  \chi_{\left\{\tilde u<l_j\right\}}
  \\&=u^{\frac{m-1}{p}}|Du|\frac{1}{l_j-\bar l}
   \left(1+\frac{l_j-u}{l_j-\bar l}\right)^{-\frac{1}{p}-\frac{\lambda}{p}}\chi_{\left\{\frac{1}{128} \omega<u<l_j\right\}}
  \\&\leq u^{\frac{m-1}{p}}|Du|\frac{1}{l_j-\bar l}
   \left(1+\frac{l_j-u}{l_j-\bar l}\right)^{-\frac{1}{p}-\frac{\lambda}{p}}
\chi_{\left\{u<l_j\right\}}=u^{\frac{m-1}{p}}|D\psi_j|,
     \end{split}\end{equation*}
     where
        \begin{equation*}\begin{split}\psi_j(x,t)=\frac{1}{l_j-\bar l}\left[\int_{u}^{l_j}
   \left(1+\frac{l_j-s}{l_j-\bar l}\right)^{-\frac{1}{p}-\frac{\lambda}{p}}\,\mathrm {d}s\right]_+.\end{split}\end{equation*}
  At this point, we apply Lemma \ref{Cac1}
   with $(l,d,\theta)$ replaced by  $(l_j,l_j-\bar l,(\mu_+)^{1-m}(l_j-\bar l)^{2-p}r_j^p)$.
 Taking into account the estimates for $T_3$-$T_6$, we conclude that
  \begin{equation*}\begin{split}
  T_7\leq &\gamma \frac{(l_j-\bar l)^{p-2}}{r_j^n}
    \iint_{Q_j(\bar l)}u^{m-1}
|D\psi_j|^p \varphi_j(\bar l)^{k_1p}\,\mathrm {d}x\mathrm {d}t
  \\
 \leq &\gamma  \frac{(l_j-\bar l)^{p-2}}{r_j^p}\iint_{L_j(\bar l)}u^{m-1}\left(\frac{l_j-u}{l_j-\bar l}\right)^{(1+\lambda)(p-1)}
 \varphi_j(\bar l)^{(k_1-1)p}\,\mathrm {d}x\mathrm {d}t
  \\&+\gamma \iint_{L_j(\bar l)}\frac{l_j-u}{l_j-\bar l}|\partial_t\varphi_j(\bar l)|\,\mathrm {d}x\mathrm {d}t
+\gamma \frac{r_j^p}{(l_j-\bar l)^p}(\mu_+)^{1-m}\int_{B_j}g^{\frac{p}{p-1}}\,\mathrm {d}x
 \\& +\gamma \frac{r_j^p}{(l_j-\bar l)^{p-1}}(\mu_+)^{1-m}\int_{B_j}|f|\,\mathrm {d}x
 \\ \leq&\gamma (\chi+B^{-p}+B^{-(p-1)}).
  \end{split}\end{equation*}
We now turn our attention to the estimate of $T_8$. To this end, we use H\"older inequality to deduce
an upper bound for $\tilde\psi_j$
   \begin{equation}\begin{split}\label{psitilde}
   \tilde\psi_j(x,t)\leq\frac{1}{l_j-\bar l}\left[\int_{\tilde u}^{l_j}
   \left(1+\frac{l_j-s}{l_j-\bar l}\right)^{-1-\lambda}\,\mathrm {d}s\right]^{\frac{1}{p}}(l_j-\tilde u)_+^{\frac{1}{p^\prime}}
   \leq \frac{(l_j-\tilde u)_+^{\frac{1}{p^\prime}}}{(l_j-\bar l)^{\frac{1}{p^\prime}}}.
   \end{split}\end{equation}
   Since $|D\varphi_j|\leq  r_j^{-1}$, we use Young's inequality, \eqref{0st estimate} and \eqref{T3} to conclude that
    \begin{equation*}\begin{split}
  T_8&\leq \frac{(\mu_+)^{m-1}(l_j-\bar l)^{p-2}}{r_j^n}\iint_{Q_j(\bar l)}
\varphi_j(\bar l)^{(k_1-1)p}\tilde \psi_j^p|D\varphi_j|^p \,\mathrm {d}x\mathrm {d}t
\\&\leq  \frac{(\mu_+)^{m-1}(l_j-\bar l)^{p-2}}{r_j^{n+p}}\iint_{L_j(\bar l)}
\left(\frac{l_j-\tilde u}{l_j-\bar l}\right)^{p-1}
\varphi_j(\bar l)^{(k_1-1)p} \,\mathrm {d}x\mathrm {d}t
\\&\leq \frac{(\mu_+)^{m-1}(l_j-\bar l)^{p-2}}{r_j^{n+p}}|L_j(\bar l)|
\\&\quad+
\frac{(\mu_+)^{m-1}(l_j-\bar l)^{p-2}}{r_j^{n+p}}\iint_{L_j(\bar l)}
\left(\frac{l_j-u}{l_j-\bar l}\right)^{(1+\lambda)(p-1)}
\varphi_j(\bar l)^{(k_1-1)p} \,\mathrm {d}x\mathrm {d}t
\\&\leq \gamma \chi
  \end{split}\end{equation*}
  and hence we arrive at
 $T_2\leq \gamma (\chi+B^{-p}+B^{-(p-1)})^{1+\frac{p}{n}}$ for a constant $\gamma$ depending only upon the data.
Finally, we arrive at
   \begin{equation}\begin{split}\label{Lprimeprime}
  &\frac{(\mu_+)^{m-1}(l_j-\bar l)^{p-2}}{r_j^{n+p}}\iint_{L_j^{\prime\prime}(\bar l)}
  \left(\frac{l_j-\tilde u}{l_j-\bar l}\right)^{(1+\lambda)(p-1)}\varphi_j(\bar l)^{k-p}
 \,\mathrm {d}x\mathrm {d}t
 \\&\leq 4^n\epsilon_2\chi+\gamma(\epsilon_2)(\chi+B^{-p}+B^{-(p-1)})^{1+\frac{p}{n}}.
 \end{split}\end{equation}
This also implies that
  \begin{equation}\begin{split}\label{AA1}
  &\frac{(\mu_+)^{m-1}(l_j-\bar l)^{p-2}}{r_j^{n+p}}\iint_{L_j(\bar l)}\left(\frac{l_j-\tilde u}{l_j-\bar l}
  \right)^{(1+\lambda)(p-1)}\varphi_j(\bar l)^{k-p}
 \,\mathrm {d}x\mathrm {d}t
 \\&\leq 4^n\epsilon_1^{(1+\lambda)(p-1)}\chi+
 4^n\epsilon_2\chi+\gamma(\epsilon_2)(\chi+B^{-p}+B^{-(p-1)})^{1+\frac{p}{n}}.
 \end{split}\end{equation}
 Our next goal is to refine the estimate \eqref{important G}.
 To this aim, we apply Lemma \ref{Cac1} with $(l,d,\theta)$ replaced by  $(l_j,l_j-\bar l,(\mu_+)^{1-m}(l_j-\bar l)^{2-p}r_j^p)$
 to obtain
 \begin{equation}\begin{split}\label{estimateforG}
  \esssup_t&
  \frac{1}{r_j^n}
  \int_{B_j}
G\left(\frac{l_j-u}{l_j-\bar l} \right)\varphi_j(\bar l)^k\,\mathrm {d}x
  \\
 \leq &
 \gamma  \frac{(l_j-\bar l)^{p-2}}{r_j^{p+n}}\iint_{L_j(\bar l)}u^{m-1}\left(\frac{l_j-u}{l_j-\bar l}\right)^{(1+\lambda)(p-1)}
  \varphi_j(\bar l)^{k-p}\,\mathrm {d}x\mathrm {d}t
  \\&+\gamma
    \frac{1}{r_j^n}\iint_{L_j(\bar l)}\frac{l_j-u}{l_j-\bar l}\varphi_j(\bar l)^{k-1}
    |\partial_t\varphi_j(\bar l)|\,\mathrm {d}x\mathrm {d}t
   \\&+\gamma \frac{r_j^{p-n}}{(l_j-\bar l)^p}(\mu_+)^{1-m}\int_{B_j}g^{\frac{p}{p-1}}\,\mathrm {d}x
  +\gamma \frac{r_j^{p-n}}{(l_j-\bar l)^{p-1}}(\mu_+)^{1-m}\int_{B_j}|f|\,\mathrm {d}x
  \\=&:S_1+S_2+S_3+S_4,
   \end{split}\end{equation}
   with the obvious meanings of $S_1$-$S_4$. To estimate $S_1$, we apply \eqref{AA1} and \eqref{cut} to deduce
   \begin{equation*}\begin{split}
   S_1&\leq 2^{(1+\lambda)(p-1)}
   \gamma\frac{(l_j-\bar l)^{p-2}(\mu_+)^{m-1}}{r_j^{p+n}}\iint_{L_j(\bar l)}
   \left(\frac{l_j-\tilde u}{l_j-\bar l}\right)^{(1+\lambda)(p-1)}
  \varphi_j(\bar l)^{k-p}\,\mathrm {d}x\mathrm {d}t
  \\&\leq 2^{(1+\lambda)(p-1)}
   \gamma\left[4^n\epsilon_1^{(1+\lambda)(p-1)}\chi+
 4^n\epsilon_2\chi+\gamma(\epsilon_2)(\chi+B^{-p}+B^{-(p-1)})^{1+\frac{p}{n}}\right].
   \end{split}\end{equation*}
   Next, we consider the estimate for $S_2$. Similar to \eqref{cut}, we infer from $l_j-\frac{1}{128} \omega> \frac{1}{2}l_j$ that
   \begin{equation*}\begin{split}
   S_2=&\gamma
    \frac{1}{r_j^n}
    \iint_{L_j(\bar l)\cap \left\{u>\frac{1}{128}\omega\right\}}
    \frac{l_j-\tilde   u}{l_j-\bar l}\varphi_j(\bar l)^{k-1}|\partial_t\varphi_j(\bar l)|\,\mathrm {d}x\mathrm {d}t+
   \\&+\gamma
    \frac{1}{r_j^n} \iint_{L_j(\bar l)\cap \left\{u\leq\frac{1}{128}\omega\right\}}
    \frac{l_j}{l_j-\bar l}\varphi_j(\bar l)^{k-1}|\partial_t\varphi_j(\bar l)|\,\mathrm {d}x\mathrm {d}t
    \\ \leq& 2\gamma
    \frac{1}{r_j^n}
    \iint_{L_j(\bar l)}
    \frac{l_j-\tilde   u}{l_j-\bar l}\varphi_j(\bar l)^{k-1}|\partial_t\varphi_j(\bar l)|\,\mathrm {d}x\mathrm {d}t.
    \end{split}\end{equation*}
Furthermore, we decompose
   $L_j(\bar l)=L^\prime_j(\bar l)\cup L^{\prime\prime}_j(\bar l)$, where $L^\prime_j(\bar l)$ and $L^{\prime\prime}_j(\bar l)$
   satisfy \eqref{Ldecomposition}. Since $|\partial_t\varphi_j(\bar l)|\leq 9(l_j-\bar l)^{p-2}(\mu_+)^{m-1}r_j^{-p}$, we
   use \eqref{0st estimate}
   and \eqref{Lprimeprime} to conclude that
    \begin{equation*}\begin{split}
   S_2\leq & \gamma
    \frac{(l_j-\bar l)^{p-2}(\mu_+)^{m-1}}{r_j^{n+p}}\iint_{L_j(\bar l)}\frac{l_j-\tilde u}{l_j-\bar l}\varphi_j(\bar l)^{k-1}
    \,\mathrm {d}x\mathrm {d}t
   \\
   \leq &\gamma\epsilon_1\frac{(\mu_+)^{m-1}(l_j-\bar l)^{p-2}}{r_j^{n+p}}|L_j^\prime(\bar l)|
   \\&+\gamma
   \frac{(l_j-\bar l)^{p-2}(\mu_+)^{m-1}}{r_j^{p+n}}\iint_{L_j^{\prime\prime}(\bar l)}
   \left(\frac{l_j-\tilde u}{l_j-\bar l}\right)^{(1+\lambda)(p-1)}
  \varphi_j(\bar l)^{k-p}\,\mathrm {d}x\mathrm {d}t
  \\ \leq & 4^n\epsilon_1\chi+\gamma\left[4^n\epsilon_2\chi+\gamma(\epsilon_2)(\chi+B^{-p}+B^{-(p-1)})^{1+\frac{p}{n}}\right].
     \end{split}\end{equation*}
     Similar to the estimates of $T_5$ and $T_6$, we see that $S_3+S_4\leq \gamma (B^{-p}+B^{-(p-1)})$.
     Plugging the estimates for $S_1$-$S_4$
     into \eqref{estimateforG} and taking into account \eqref{AA1}, we arrive at
     \begin{equation}\begin{split}\label{Ajbar}
     A_j(\bar l)\leq &\gamma (B^{-p}+B^{-(p-1)})+2^{(1+\lambda)(p-1)}
   \gamma4^n(\epsilon_1+
 \epsilon_2)\chi
 \\&+\gamma(\epsilon_1,\epsilon_2)\left(\chi^{1+\frac{p}{n}}+(B^{-p}+B^{-(p-1)})^{1+\frac{p}{n}}\right).
      \end{split}\end{equation}
At this point, we first choose $\epsilon_1$ and $\epsilon_2$ be such that
 \begin{equation}\begin{split}\label{epsilon}
 \epsilon_1=\epsilon_2=\frac{1}{2^{3+2n+(1+\lambda)(p-1)}\gamma}.
 \end{split}\end{equation}
 Next, we determine the value of $\chi$ by
  \begin{equation}\begin{split}\label{chi}
  \chi=\frac{1}{100^{\frac{n}{p}}\gamma(\epsilon_1,\epsilon_2)^{\frac{n}{p}}}.
  \end{split}\end{equation}
 Finally, with the choices of $\epsilon_1$, $\epsilon_2$ and $\chi$, we set $B$ so large that
  \begin{equation*}\begin{split}
  B^{-p}+B^{1-p}<\min\left\{\frac{1}{100\gamma}\chi,\ \left(\frac{1}{100\gamma(\epsilon_1,
  \epsilon_2)}\chi\right)^\frac{1}{1+\frac{p}{n}}\right\}.
   \end{split}\end{equation*}
      With the choices of $\epsilon_1$, $\epsilon_2$, $\chi$ and $B$,
      we get $A_j(\bar l)\leq \frac{1}{2}\chi$,
      which completes the proof of
      \eqref{Aj}.
      We remark that the choice of $\chi$ in \eqref{chi} determines the value of $\nu_0$ via \eqref{nu0}, i.e.,
     \begin{equation}\begin{split}\label{nu0fixed}\nu_0=\frac{\chi}{4(\gamma_0+\gamma_1)C^{n+p}},\end{split}\end{equation}
     where $\gamma_0$ and $\gamma_1$ are the constants in \eqref{nu0}.
Our task now is to determine the value of $l_{j+1}$.
     To this end, we first consider the case
    \begin{equation*}\iint_{B_{r_j}(x_1)\times\left\{t_1-r_j^p<t<t_1\right\}}(l_j-u)_+^{(1+\lambda)(p-1)}
    \,\mathrm {d}x\,\mathrm {d}t=0.\end{equation*}
    Then, we apply Lemma \ref{Lebesguepoint+} and \eqref{lj} to conclude that
   \begin{equation*} u(x_1,t_1)\geq l_j\geq  \mu_-+\frac{1}{8}B\alpha_{i-1}+\frac{1}{16}\omega
   > \mu_-+\frac{1}{16}\omega,
   \end{equation*}
   which proves \eqref{DeGiorgi1}.
   Next, we consider the case
     \begin{equation*}\iint_{B_{r_j}(x_1)\times\left\{t_1-r_j^p<t<t_1\right\}}(l_j-u)_+^{(1+\lambda)(p-1)}
     \,\mathrm {d}x\,\mathrm {d}t>0.\end{equation*}
     We see that $A_j(l)\to+\infty$
    as $l\to l_j$.
    Noting that $A_j(l)$ is continuous and increasing, then there exists a number $\tilde l\in (\bar l, l_j)$ such that $A_j(\tilde l)=\chi$.
    At this point, we choose $l_{j+1}$ via
    \begin{equation}\label{def lj+1}
	l_{j+1}=\begin{cases}
\tilde l,&\quad \text{if}\quad \tilde l<l_j-\frac{1}{4}(\alpha_{j-1}-\alpha_j),\\
	l_j-\frac{1}{4}(\alpha_{j-1}-\alpha_j),&\quad \text{if}\quad \tilde l\geq l_j-\frac{1}{4}(\alpha_{j-1}-\alpha_j).
	\end{cases}
\end{equation}
According to \eqref{upper bound for l},
we see that $\bar l<l_j-\frac{1}{4}(\alpha_{j-1}-\alpha_j)$ and
the definition of $l_{j+1}$ is justified.

Step 4: \emph{Proof of the inequality \eqref{DeGiorgi1}.}
To start with, we set
$Q_j=Q_j(l_{j+1})$, $L_j=L_j(l_{j+1})$ and $d_j=l_j-l_{j+1}$. In view of \eqref{def lj+1}, we observe that \eqref{li} and \eqref{Aj-1}
hold with $i=j+1$. Moreover, from \eqref{lj}$_j$, we find that
\begin{equation*}\begin{split}
l_{j+1}>&\bar l=\frac{1}{2}l_j+\frac{1}{16}B\alpha_j+\frac{1}{32}\omega+
       \frac{1}{2}\mu_-
       \\&>\frac{1}{2}\left(\mu_-+\frac{1}{8}B\alpha_{j-1}+\frac{1}{16}\omega\right)+\frac{1}{16}B\alpha_j+\frac{1}{32}\omega+
       \frac{1}{2}\mu_-
       \\&=\mu_-+\frac{1}{16}B(\alpha_j+\alpha_{j-1})+\frac{1}{16}\omega\geq \mu_-+\frac{1}{8}B\alpha_j+\frac{1}{16}\omega,
\end{split}\end{equation*}
since $\alpha_{j-1}\geq\alpha_j$. This implies that \eqref{lj}$_{j+1}$ holds. Repeating the arguments above, we continue to define
$l_{j+2}$ and therefore we can construct a sequence of numbers $\{l_i\}_{i=0}^\infty$ satisfying \eqref{li}-\eqref{lj}.
Since the sequence $\{l_i\}_{i=0}^\infty$ is decreasing, we infer from \eqref{lj} that the limitation of $l_i$ exists.
This also implies that $d_i\to0$ as $i\to\infty$. Set $$\hat l=\lim_{i\to\infty}l_i$$ and we claim that $\hat l=u(x_1,t_1)$.
Noting that $\mu_+>0$ and \eqref{Aj-1} holds for any $i=1,2,\cdots$, we conclude that
  \begin{equation*}\begin{split}\frac{1}{r_i^{n+p}}&\iint_{B_{r_i}(x_1)\times\left\{t_1-r_i^p<t<t_1\right\}}
  (\hat l-u)_+^{(1+\lambda)(p-1)}\,\mathrm {d}x\,\mathrm {d}t\leq
  \frac{4^{n+p}}{r_{i-1}^{n+p}}\iint_{L_{i-1}}\left(l_{i-1}-u\right)^{(1+\lambda)(p-1)}\varphi_{i-1}^{k-p}
 \,\mathrm {d}x\mathrm {d}t
 \\&\leq 4^{n+p}(\mu_+)^{1-m}A_{i-1}(l_i)d_{i-1}^{(1+\lambda)(p-1)-(p-2)}\leq 4^{n+p}(\mu_+)^{1-m}\chi d_{i-1}^{(1+\lambda)(p-1)-(p-2)}
 \to 0
  \end{split}\end{equation*}
  as $i\to\infty$. According to Lemma \ref{Lebesguepoint+}, we have $\hat l=u(x_1,t_1)$. Next, we show that for any $j\geq1$ there holds
  \begin{equation}\begin{split}\label{djdj-1}
  d_j\leq &\frac{1}{4}d_{j-1}+\gamma \frac{4^{-j-100}}{B}\omega+
  \gamma\left(r_{j-1}^{p-n}(\mu_+)^{1-m}\int_{B_{j-1}} g(y)^{\frac{p}{p-1}}
  \,\mathrm {d}y\right)^{\frac{1}{p}}\\&+\gamma\left(r_{j-1}^{p-n}(\mu_+)^{1-m}\int_{B_{j-1}}|f(y)|
  \,\mathrm {d}y\right)^{\frac{1}{p-1}}.
  \end{split}\end{equation}
  To start with, for any fixed $j\geq 1$, we first assume that
  \begin{equation}\begin{split}\label{djdj-1proof}
  d_j>\frac{1}{4}d_{j-1}\qquad\text{and}\qquad d_j>\frac{1}{4}(\alpha_{j-1}-\alpha_j),
   \end{split}\end{equation}
   since otherwise \eqref{djdj-1} holds immediately. From $d_j>\frac{1}{4}(\alpha_{j-1}-\alpha_j)$, we infer from \eqref{def lj+1} that
   $A_j(l_{j+1})=A_j(\tilde l)=\chi$. In view of $d_j>\frac{1}{4}d_{j-1}$, we find that
    \begin{equation*}\begin{split}
 (\mu_+)^{1-m}d_j^{2-p}r_j^p\leq (\mu_+)^{1-m}\frac{r_{j-1}^p}{4^p}\left(\frac{d_{j-1}}{4}\right)^{2-p}=\frac{1}{16}
 (\mu_+)^{1-m}d_{j-1}
 ^{2-p}r_{j-1}^p,
  \end{split}\end{equation*}
  which yields $Q_j\subset Q_{j-1}\subset \hat Q$ and  $\varphi_{j-1}(x,t)=1$ for $(x,t)\in Q_j$.
  Taking into account that $d_j>\frac{1}{4}d_{j-1}$, we can
  repeat the arguments from Step 3. Then, we deduce the estimate similar to \eqref{Ajbar},
  \begin{equation*}\begin{split}
  \chi=A_j(l_{j+1})\leq&
  \gamma \frac{r_j^{p-n}}{d_j^p}(\mu_+)^{1-m}\int_{B_j}g^{\frac{p}{p-1}}\,\mathrm {d}x+
  \gamma \frac{r_j^{p-n}}{d_j^{p-1}}(\mu_+)^{1-m}\int_{B_j}|f|\,\mathrm {d}x
  \\
  &+2^{(1+\lambda)(p-1)}
   \gamma4^n(\epsilon_1+
 \epsilon_2)\chi
 +\gamma(\epsilon_1,\epsilon_2)\chi^{1+\frac{p}{n}}\\&+\gamma\left[
 \frac{r_j^{p-n}}{d_j^p}(\mu_+)^{1-m}\int_{B_j}g^{\frac{p}{p-1}}\,\mathrm {d}x+\gamma
 \frac{r_j^{p-n}}{d_j^{p-1}}(\mu_+)^{1-m}\int_{B_j}|f|\,\mathrm {d}x\right]^{1+\frac{p}{n}}.
  \end{split}\end{equation*}
  Recalling that the choices of $\epsilon_1$, $\epsilon_2$ and $\chi$ in \eqref{epsilon}-\eqref{chi} yields that
    \begin{equation*}\begin{split}
  2^{(1+\lambda)(p-1)}
   \gamma4^n(\epsilon_1+
 \epsilon_2)\chi
 +\gamma(\epsilon_1,\epsilon_2)\chi^{1+\frac{p}{n}}\leq\frac{1}{2}\chi.
 \end{split}\end{equation*}
 Consequently, we conclude that either
  \begin{equation*}\begin{split}
  d_j\leq \gamma\left(r_j^{p-n}(\mu_+)^{1-m}\int_{B_j}g^{\frac{p}{p-1}}\,\mathrm {d}x\right)^{\frac{1}{p}}
  \qquad\text{or}\qquad
  d_j\leq
  \gamma \left(r_j^{p-n}(\mu_+)^{1-m}\int_{B_j}|f|\,\mathrm {d}x\right)^{\frac{1}{p-1}},
   \end{split}\end{equation*}
   which proves the inequality \eqref{djdj-1}. Let $J>1$ be a fixed integer. We sum up the
   inequality \eqref{djdj-1} for $j=1,\cdots,J-1$
   and there holds
   \begin{equation*}\begin{split}
   l_1-l_J\leq& \frac{1}{3}d_0+\gamma \frac{1}{B}
   \sum_{j=1}^{J-1} 4^{-j-100}\omega+\gamma\sum_{j=1}^{J-1}\left(r_{j-1}^{p-n}(\mu_+)^{1-m}\int_{B_{j-1}}|f(y)|
  \,\mathrm {d}y\right)^{\frac{1}{p-1}}
\\&+ \gamma\sum_{j=1}^{J-1}\left(r_{j-1}^{p-n}(\mu_+)^{1-m}\int_{B_{j-1}} g(y)^{\frac{p}{p-1}}
  \,\mathrm {d}y\right)^{\frac{1}{p}},
   \end{split}\end{equation*}
   where the constant $\gamma$ depends only upon the data.
   Since $d_0=l_0-l_1$ and $l_0=\mu_-+\frac{1}{4}\omega$, we use \eqref{omega1violated} to obtain
  \begin{equation}\begin{split}\label{uupperbound}
  \mu_-+\frac{1}{4}\omega\leq
 & \frac{4}{3}d_0+l_J+\gamma\frac{1}{B}\omega
  +\gamma(\mu_+)^{\frac{1-m}{p-1}}\int_0^{4C^{-1}R}\left(\frac{1}{r^{n-p}}\int_{B_r(x_1)}f(y)
\,\mathrm {d}y\right)^{\frac{1}{p-1}}\frac{1}{r}\,\mathrm {d}r
  \\&+\gamma(\mu_+)^{\frac{1-m}{p}}\int_0^{4C^{-1}R}\left(\frac{1}{r^{n-p}}\int_{B_r(x_1)}g(y)^{\frac{p}{p-1}}
\,\mathrm {d}y\right)^{\frac{1}{p}}\frac{1}{r}\,\mathrm {d}r
\\ \leq & \frac{4}{3}d_0+l_J+\gamma\frac{1}{B}\omega
   \end{split}\end{equation}
   for a constant $\gamma$ depending only upon the data.
   Recalling that $d_0\leq l_0-\bar  l$ where $\bar l=\frac{1}{2}l_0+\frac{1}{2}\mu_-+\frac{1}
   {16}B\alpha_0+\frac{1}{32}\omega$. According to \eqref{l0}, we have $d_0\leq\frac{1}{8}\omega$.
  Letting $J\to\infty$ in \eqref{uupperbound}, we infer that
   \begin{equation*}\begin{split}
   u(x_1,t_1)>\mu_-+\frac{1}{24}\omega,
    \end{split}\end{equation*}
    provided that we choose $B>24\gamma$.
    Since $(x_1,t_1)$ is a fixed Lebesgue point of $u$, we conclude that
    the inequality \eqref{DeGiorgi1} holds for
    almost everywhere point in $Q_{\frac{1}{4}R}^-(\bar t)$, provided that
    $\nu_0$ satisfies \eqref{nu0fixed}.
    The proof of the lemma is now complete.
\end{proof}
We now define a time level
$\hat t=\bar t-(\mu_+)^{1-m}\omega^{2-p}\left(\frac{1}{4}R\right)^p$
 and it follows from Lemma \ref{lemmaDeGiorgi1} that
 \begin{equation}\begin{split}\label{hat t lower}
 u(x,\hat  t)>\mu_-+\frac{\omega}{2^5}\qquad\text{for}\qquad x\in B_{\frac{R}{4}}.
  \end{split}\end{equation}
 Furthermore, there exists a constant $A_1>0$ such that $4^{-\frac{p}{p-2}}<A_1<A$ and
  \begin{equation*}\begin{split}
\hat t=-(\mu_+)^{1-m}A_1^{p-2}\omega^{2-p}R^p.
 \end{split}\end{equation*}
Next, we establish the following result regarding the time propagation of positivity.
\begin{lemma}\label{timeexpandlemma}
Let $u$ be a bounded nonnegative weak solution to \eqref{parabolic}-\eqref{A} in $\Omega_T$.
Given $\nu_*\in(0,1)$, there exists a constant $s_*=s_*(\text{data},\nu_*)>5$, such that either
\begin{equation}
\label{omega2}\omega\leq 2^{\frac{2}{p}s_*}(\mu_+)^{\frac{1-m}{p}}F_1(2R)+2^{\frac{1}{p-1}s_*}(\mu_+)^{\frac{1-m}{p-1}}F_2(2R)
\end{equation}
or
\begin{equation}\begin{split}\label{time expand}
\left|\left\{x\in B_{\frac{R}{8}}:u(x,t)<\mu_-+\frac{\omega}{2^{s_*}}\right\}\right|\leq \nu_*|B_{\frac{R}{8}}|
\end{split}\end{equation}
holds for any $t\in (\hat t,0)$.
 \end{lemma}
\begin{proof}
We first recall that
\begin{equation}\begin{split}\label{SS0}\psi^-&=\ln^+\left(\frac{H_k^-}{H_k^--(u-k)_-+c}\right)
=\begin{cases}
	\ln\left(\dfrac{H_k^-}{H_k^-+u-k+c}\right),&\quad k-H_k^-\leq u<k-c,\\
	0,&\quad u\geq k-c,
	\end{cases}\end{split}\end{equation}
where
\begin{equation*} H_k^-\geq \esssup_{B_{\frac{R}{4}}\times[\hat t,0]}(u-k)_-\qquad\text{and}\qquad 0<c<H_k^-.\end{equation*}
Next, we let
$k=\mu_-+\frac{1}{2^5}\omega$, $c=\frac{1}{2^{5+l}}\omega$ where $l\geq1$ will be determined later.
Moreover, we choose $H_k^-=\frac{1}{32}\omega$, which is admissible since $(u-k)_-\leq \frac{1}{32}\omega$. In view of
\eqref{SS0}, we see that
\begin{equation}\begin{split}\label{psipsiprime}
\psi^-(u)\leq l\ln2\qquad\text{and}\qquad |(\psi^-)^\prime(u)|^{2-p}\leq \omega^{p-2}.
\end{split}\end{equation}
We now apply the logarithmic estimate \eqref{lnCac} over the cylinder $B_{\frac{R}{4}}\times[\hat t,0]$. This yields that
\begin{equation*}\begin{split}
\esssup_{\hat t<t<0}&\int_{B_{\frac{R}{4}}\times\{t\}}[\psi^-(u)]^2\zeta^p\,\mathrm {d}x
\\ \leq &\int_{B_{\frac{R}{4}}\times\{\hat t\}}[\psi^-(u)]^2\zeta^p\,\mathrm {d}x+\gamma
\iint_{B_{\frac{R}{4}}\times[\hat t,0]}u^{m-1}\psi^-(u)|D\zeta|^p\left[(\psi^-)^\prime(u)\right]^{2-p}
\,\mathrm {d}x\mathrm {d}t
\\&+\gamma\ln\left(\frac{H}{c}\right)\left(\frac{1}{c}\iint_{B_{\frac{R}{4}}\times[\hat t,0]}|f|\,\mathrm {d}x\mathrm {d}t+
\frac{1}{c^2}\iint_{B_{\frac{R}{4}}\times[\hat t,0]}g^{\frac{p}{p-1}}\,\mathrm {d}x\mathrm {d}t\right)
\\ =&:I_1+I_2+I_3.
\end{split}\end{equation*}
We take the cutoff function $\zeta=\zeta(x)$, which satisfies $0\leq\zeta\leq1$ in $B_{\frac{R}{4}}$, $\zeta\equiv 1$
in $B_{\frac{R}{8}}$ and $|D\zeta|\leq 4R^{-1}$.
Observe from \eqref{hat t lower} that $\psi^-(x,\hat t)= 0$ for all $x\in B_{\frac{R}{4}}$,
we get $I_1=0$. To estimate $I_2$, we note that
$-\hat t\leq A^{p-2}(\mu_+)^{1-m}\omega^{2-p}R^p$. In view of \eqref{psipsiprime}, we conclude that
\begin{equation*}\begin{split}
I_2\leq 4l(\ln2)(-\hat t)|B_{\frac{R}{4}}|(\mu_+)^{m-1}R^{-p}\omega^{p-2}\leq \gamma lA^{p-2}|B_{\frac{R}{8}}|.
\end{split}\end{equation*}
Finally, we consider the estimate for  $I_3$. Using the inequality $-\hat t\leq A^{p-2}(\mu_+)^{1-m}\omega^{2-p}R^p$ again, we infer
that
\begin{equation*}\begin{split}
I_3\leq &l(\ln2)\frac{2^{5+l}A^{p-2}}{\omega^{p-1}}(\mu_+)^{1-m}R^p\int_{B_{R}}|f|\,\mathrm {d}x
+ l(\ln2)\frac{2^{10+2l}A^{p-2}}{\omega^{p}}(\mu_+)^{1-m}R^p\int_{B_{R}}g^{\frac{p}{p-1}}\,\mathrm {d}x
\\ \leq &\gamma (lA^{p-2})\frac{2^l}{\omega^{p-1}}\left((\mu_+)^{1-m}R^{p-n}\int_{B_{R}}|f|\,\mathrm {d}x\right)|B_{\frac{R}{8}}|
\\&+\gamma (lA^{p-2})\frac{2^{2l}}{\omega^{p}}\left((\mu_+)^{1-m}R^{p-n}\int_{B_{R}}g^{\frac{p}{p-1}}\,\mathrm {d}x\right)|B_{\frac{R}{8}}|.
\end{split}\end{equation*}
At this point, we assume that
\begin{equation}\label{omega2assumption}
\omega> 2^{\frac{2}{p}l}(\mu_+)^{\frac{1-m}{p}}F_1(2R)+2^{\frac{1}{p-1}l}(\mu_+)^{\frac{1-m}{p-1}}F_2(2R)
\end{equation}
and this yields $I_3\leq \gamma lA^{p-2}|B_{\frac{R}{8}}|$. Therefore, we arrive at
\begin{equation}\begin{split}\label{psi-upper}
\int_{B_{\frac{R}{4}}\times\{t\}}[\psi^-(u)]^2\zeta^p\,\mathrm {d}x\leq \gamma lA^{p-2}|B_{\frac{R}{8}}|
\end{split}\end{equation}
for any $t\in(\hat t,0)$. On the other hand,
the left-hand side of \eqref{psi-upper} is estimated below by integrating
over the smaller set
$$\left\{x\in B_{\frac{R}{8}}:u<\mu_-+\frac{\omega}{2^{l+5}}\right\}\subset B_{\frac{R}{4}}.$$
On such a set, $\zeta\equiv 1$ and
\begin{equation*}\begin{split}\psi^-&=\ln^+\left(\frac{\frac{1}{2^5}\omega}{\frac{1}{2^5}\omega
-(u-k)_-+\frac{1}{2^{l+5}}\omega}\right)
\geq \ln2^{l-1}>(l-1)\ln2.\end{split}\end{equation*}
This yields
\begin{equation}\begin{split}\label{SS4}\int_{B_{\frac{R}{8}}\times\{t\}}\left(\psi^-\right)^2\zeta^p dx\geq
(l-1)^2(\ln2)^2
\left|\left\{x\in B_{\frac{R}{8}}:u<\mu_-+\frac{\omega}{2^{l+5}} \right\}\right|,\end{split}\end{equation}
for all $t\in(\hat t,0)$. Combining \eqref{psi-upper} and \eqref{SS4}, we deduce
\begin{equation*}
\left|\left\{x\in B_{\frac{R}{8}}:u<\mu_-+\frac{\omega}{2^{l+5}} \right\}\right|\leq \gamma\frac{l}{(l-1)^2}A^{p-2}
|B_{\frac{R}{8}}|.
\end{equation*}
For $\nu_*\in(0,1)$, we set $l$ be a fixed integer such that $l>1+\gamma\nu_*^{-1}A^{p-2}$ and choose $s_*=l+5$.
This proves the inequality \eqref{time expand} under the assumption \eqref{omega2assumption}. On the other hand, if \eqref{omega2assumption}
is violated, then we obtain \eqref{omega2} for such a choice of $s_*$. The proof of the lemma is now complete.
\end{proof}
According to the proof of Lemma \ref{timeexpandlemma}, we conclude that for any fixed $\nu_*\in(0,1)$, we can choose
\begin{equation}\label{s*}s_*=s_*(\nu_*)=2 \nu_*^{-1}A^{p-2}\end{equation}
and Lemma \ref{timeexpandlemma} follows for such a choice of $s_*$.
Moreover, for $R>0$, we set
\begin{equation*}\widetilde Q_1=B_\frac{R}{8}\times (\hat t,0)\qquad \widetilde Q_2=B_\frac{R}{16}\times \left(\frac{\hat t}{2},0\right).\end{equation*}
We now establish a variant De Giorgi type lemma
which concerns the initial data.
The proof is similar to that of Lemma \ref{lemmaDeGiorgi1}, but
the arguments need some nontrivial modifications.
 \begin{lemma}\label{lemmaDeGiorgi2}
Let $\xi=2^{-s_*}$ and $0<\xi<2^{-5}$.
Let $u$ be a bounded nonnegative weak solution to \eqref{parabolic}-\eqref{A} in $\Omega_T$.
Then, there exist $\nu_1\in(0,1)$ and $\tilde B>1$ which
can be quantitatively determined a priori only in terms of the data, such that if
\begin{equation*}\left|\left\{(x,t)\in \widetilde Q_1:u<\mu_-+\xi \omega\right\}
\right|\leq \nu_1|\widetilde Q_1|,\end{equation*}
then either
\begin{equation}
\label{DeGiorgi2}u(x,t)>\mu_-+\frac{1}{2}\xi\omega\qquad\text{for}\ \ \text{a.e.}\ \ (x,t)\in \widetilde Q_2\end{equation}
or
\begin{equation}
\label{omega3}\xi\omega\leq \tilde B\left(
(\mu_-+\xi\omega)^{\frac{1-m}{p}}F_1(2R)+(\mu_-+\xi\omega)^{\frac{1-m}{p-1}}F_2(2R)\right).\end{equation}
 \end{lemma}
 \begin{proof}
For simplicity of notation, we write $\tilde \mu=\mu_-+\xi\omega$ and observe that $\tilde\mu\leq\mu_+$.
 We first assume that \eqref{omega3} is violated, that is,
 \begin{equation}
\label{omega1violated1}\xi\omega>\tilde B
\left(\tilde\mu^{\frac{1-m}{p}}F_1(2R)+\tilde\mu^{\frac{1-m}{p-1}}F_2(2R)
\right),\end{equation}
since otherwise there is nothing to
prove.
Our problem reduces to prove \eqref{DeGiorgi2}.
 Fix $(x_1,t_1)\in \widetilde Q_{\frac{R}{16}}$ and assume that $(x_1,t_1)$ is a Lebesgue point of $u$ in the sense of
 Definition \ref{def lebesgue}.
 Set $r_j=4^{-j-4}R$
 and $B_j=B_{r_j}(x_1)$. Next, we set
 $Q_j=B_j\times (\hat t,t_1)$
and $\phi_j\in C_0^\infty(B_j)$ where $\phi_j=1$ on $B_{j+1}$ and $|D\phi_j|\leq r_j^{-1}$.
For a sequence $\{l_j\}_{j=0}^\infty$ and a fixed $l>0$, we set
$$Q_j(l)=B_j\times (t_1-\tilde\mu^{1-m}(l_j-l)^{2-p}r_j^p,t_1).$$
Furthermore, we define $\varphi_j(l)=\phi_j(x)\theta_{j,l}(t)$, where $\phi_j\in C_0^\infty(B_j)$,
$\phi_j=1$ on $B_{j+1}$, $|D\phi_j|\leq r_j^{-1}$
 and $\theta_{j,l}(t)$ is a Lipschitz function
satisfies
 \begin{equation*}
 \theta_{j,l}(t)=1\qquad\text{in}\qquad t\geq t_1-\frac{4}{9}\tilde\mu^{1-m}(l_j-l)^{2-p}r_j^p,
 \end{equation*}
  \begin{equation*}
 \theta_{j,l}(t)=0\qquad\text{in}\qquad t\leq t_1-\frac{5}{9}\tilde\mu^{1-m}(l_j-l)^{2-p}r_j^p
 \end{equation*}
and
  \begin{equation*}
 \theta_{j,l}(t)=\frac{t-t_1-\frac{5}{9}\tilde\mu^{1-m}(l_j-l)^{2-p}r_j^p}
 {\frac{1}{9}\tilde\mu^{1-m}(l_j-l)^{2-p}r_j^p}
 \end{equation*}
in
$t_1-\frac{5}{9}\tilde\mu^{1-m}(l_j-l)^{2-p}r_j^p\leq t\leq t_1-\frac{4}{9}\tilde\mu^{1-m}(l_j-l)^{2-p}r_j^p.$
 Next, for  $j=-1,0,1,2,\cdots$, we introduce the sequence $\{\alpha_j\}$ by
  \begin{equation}\begin{split}\label{alphaxi}\alpha_j=&\int_0^{r_j}\left(r^{p-n}\tilde\mu
  ^{1-m}\int_{B_r(x_1)}
  g(y)^{\frac{p}{p-1}} \,\mathrm {d}y
  \right)^{\frac{1}{p}}\frac{\mathrm {d}r}{r}\\&+\int_0^{r_j}\left(r^{p-n}\tilde\mu^{1-m}\int_{B_r(x_1)}|f(y)| \,\mathrm {d}y
  \right)^{\frac{1}{p-1}}\frac{\mathrm {d}r}{r}.
  \end{split}\end{equation}
In view of \eqref{alphaxi}, we see that $\alpha_j\to 0$ as $j\to\infty$ and there exists a constant $\gamma$ depending only
upon the data such that
\begin{equation}\begin{split}\label{alpha1xi}
  \alpha_{j-1}-\alpha_j\geq &\gamma\left(r_j^{p-n}\tilde\mu^{1-m}\int_{B_j} g(y)^{\frac{p}{p-1}}
  \,\mathrm {d}y\right)^{\frac{1}{p}}+\gamma\left(r_j^{p-n}\tilde\mu^{1-m}\int_{B_j}|f(y)|
  \,\mathrm {d}y\right)^{\frac{1}{p-1}}
  \end{split}\end{equation}
  and
   \begin{equation}\begin{split}\label{alpha2xi}
  \alpha_{j-1}-\alpha_j\leq &\gamma\left(r_{j-1}^{p-n}\tilde\mu^{1-m}\int_{B_{j-1}} g(y)^{\frac{p}{p-1}}
  \,\mathrm {d}y\right)^{\frac{1}{p}}+\gamma\left(r_{j-1}^{p-n}\tilde\mu^{1-m}\int_{B_{j-1}}|f(y)|
  \,\mathrm {d}y\right)^{\frac{1}{p-1}}.
  \end{split}\end{equation}
  We introduce a cutoff function $\tilde u=\max\left\{u,\frac{1}{128}\xi\omega\right\}$.
   Moreover,
 for a sequence $\{l_j\}$ and a fixed $l>0$
 we define a quantity $A_j(l)$ as follows.
   \begin{itemize}
 \item[$\bullet$]In the case $Q_j(l)\subseteq \widetilde Q_1$, we define $A_j(l)$ by
 \begin{equation}\begin{split}\label{Aj1st}
 A_j(l)=&\frac{(l_j-l)^{p-2}}{r_j^{n+p}}\iint_{L_j(l)}\tilde u^{m-1}\left(\frac{l_j-u}{l_j-l}\right)^{(1+\lambda)(p-1)}\varphi_j(l)^{k-p}
 \,\mathrm {d}x\mathrm {d}t
 \\&+\esssup_{t}\frac{1}{r_j^n}\int_{B_j\times\{t\}}G\left(\frac{l_j-u}{l_j-l}\right)\varphi_j(l)^{k}
 \,\mathrm {d}x.
 \end{split}\end{equation}
 \end{itemize}
  \begin{itemize}
 \item[$\bullet$]In the case $Q_j(l)\nsubseteq\widetilde Q_1$, i.e., $t_1-\tilde\mu^{1-m}(l_j-l)^{2-p}r_j^p<\hat t$,
 we define $A_j(l)$ by
 \begin{equation}\begin{split}\label{Aj2nd}
 A_j(l)=&\frac{(l_j-l)^{p-2}}{r_j^{n+p}}\iint_{L_j}\tilde u^{m-1}\left(\frac{l_j-u}{l_j-l}\right)^{(1+\lambda)(p-1)}\phi_j^{k-p}
 \,\mathrm {d}x\mathrm {d}t
 \\&+\esssup_{\hat t<t<t_1}\frac{1}{r_j^n}\int_{B_j\times\{t\}}G\left(\frac{l_j-u}{l_j-l}\right)\phi_j^{k}
 \,\mathrm {d}x.
 \end{split}\end{equation}
 \end{itemize}
 Here, the function $G$ is defined in \eqref{G}, $L_j(l)=Q_j(l)\cap \{u\leq l_j\}$ and
 $L_j=Q_j\cap \{u\leq l_j\}$. From the proof of Lemma \ref{lemmaDeGiorgi1},
 we conclude that $A_j(l)$ is continuous for $l<l_j$.
The proof of the lemma will be divided into several steps.

Step 1: \emph{Determine the values of $l_0$ and $l_1$.}
 To start with, we set $l_0=\mu_-+\xi\omega$ and $\bar l=\frac{1}{2}l_0+\frac{1}{2}\mu_-+\frac{1}{4}\tilde B\alpha_0+\frac{1}{8}\xi\omega$.
 Observe that $\tilde B\alpha_0< \xi\omega$ and this implies
 \begin{equation}\begin{split}\label{xil0}
 l_0-\bar l=\frac{1}{2}\mu_-+\frac{1}{2}\xi\omega-\frac{1}{2}\mu_--\frac{1}{4}\tilde B\alpha_0-\frac{1}{8}\xi\omega\geq \frac{1}{8}\xi\omega
 \qquad\text{and}\qquad l_0-\bar l\leq \frac{3}{8}\xi\omega.
 \end{split}\end{equation}
 Let $A_p>1$ be a constant such that for any $X>A_p$ there holds $X\leq 2^{(p-2)X}$. At this point, we choose
  \begin{equation}\begin{split}\label{second condition for A}
  A>\max\left\{\left(\frac{100p}{p-2}\right)^{\frac{1}{p-2}},\ A_p^\frac{1}{p-2}\right\}.
  \end{split}\end{equation}
We infer from \eqref{s*} and \eqref{second condition for A} that
 \begin{equation}\begin{split}\label{xi2-p}
 \xi^{p-2}=2^{-s_*(p-2)}=2^{-2(p-2) \nu_*^{-1}A^{p-2}}\leq 2^{-2(p-2)A^{p-2}}\leq 2^{-100p}A^{2-p},
  \end{split}\end{equation}
  since $\nu_*<1$.
We now claim that $Q_0(\bar l)\nsubseteq\widetilde Q_1$. Taking into account that $\tilde\mu\leq\mu_+$, we use \eqref{xil0}
and \eqref{xi2-p} to obtain
\begin{equation*}\begin{split}
\tilde\mu^{1-m}(l_0-\bar l)^{2-p}r_0^p&\geq 4^{-4p}(\mu_+)^{1-m}\left(\xi\omega\right)^{2-p}R^p
\geq t_1+(\mu_+)^{1-m}A^{p-2}\omega^{2-p}R^p,
\end{split}\end{equation*}
which proves $Q_0(\bar l)\nsubseteq\widetilde Q_1$.
Our task now is to deduce an upper bound for $A_0(\bar l)$.
Noting that $\tilde u\leq \mu_+$ and $l_0-u\leq \xi\omega$ in $L_0$, we infer from \eqref{xil0} and \eqref{xi2-p} that
there exists a constant $\gamma_1$ depending only upon the data such that
   \begin{equation*}\begin{split}
  & \frac{(l_0-\bar  l)^{p-2}}{r_0^{n+p}}\iint_{L_0}\tilde u^{m-1}
   \left(\frac{l_0-u}{l_0-\bar l}\right)^{(1+\lambda)(p-1)}\phi_0^{k-p}
 \,\mathrm {d}x\mathrm {d}t
 \\&\leq\frac{(\mu_+)^{m-1}(l_0-\bar l)^{p-2-(1+\lambda)(p-1)}}{r_0^{n+p}}\left(\xi\omega\right)^{(1+\lambda)(p-1)}|L_0|
 \\&\leq \gamma_1 \frac{(\mu_+)^{m-1}(\xi\omega)^{p-2}}{R^{n+p}}|L_0|
\leq \gamma_1 \frac{(\mu_+)^{m-1}A^{2-p}\xi\omega^{p-2}}{R^{n+p}}R^n(-\hat t)
\frac{\left|Q_0\cap \{u\leq l_0\}\right|}{|\widetilde Q_1|}
  \\&\leq \gamma_1\frac{\left|\left\{(x,t)\in \widetilde Q_1:u<\mu_-+\xi \omega\right\}\right|}{|\widetilde Q_1|}
   \leq \gamma_1 \nu_1,
    \end{split}\end{equation*}
   since $-\hat t\leq A^{p-2}(\mu_+)^{1-m}\omega^{2-p}R^p$.
    In view of \eqref{hat t lower}, we see that $u(x,\hat t)<\mu_-+\xi\omega=l_0$ for all $x\in B_{\frac{R}{4}}$.
    We now apply Lemma \ref{Cac1} with $(l,d,\theta)$ replaced by $(l_0,l_0-
    \bar l,t_1-\hat t)$, \eqref{omega1violated1} and \eqref{xil0} to deduce
     \begin{equation*}\begin{split}
    &\esssup_{\hat t<t<t_1}\frac{1}{r_0^n}\int_{B_0\times\{t\}}G\left(\frac{l_0-u}{l_0-\bar l}\right)\phi_0^{k}
 \,\mathrm {d}x \\ \leq &\gamma  \frac{(l_0-\bar l)^{p-2}}{r_0^{p+n}}\iint_{L_0}u^{m-1}
 \left(\frac{l_0-u}{l_0-\bar l}\right)^{(1+\lambda)(p-1)}
 \phi_0^{k-p}\,\mathrm {d}x\mathrm {d}t
   \\&+\gamma\frac{r_0^{p-n}}{(\xi\omega)^p}(\mu_+)^{1-m}\int_{B_0}g^{\frac{p}{p-1}}\,\mathrm {d}x
  +\gamma \frac{r_0^{p-n}}{(\xi\omega)^{p-1}}(\mu_+)^{1-m}\int_{B_0}|f|\,\mathrm {d}x
  \\ \leq &\gamma_1  \nu_1+\gamma\frac{(\mu_+)^{1-m}}{\tilde B^p\tilde\mu^{1-m}}+
  \gamma\frac{(\mu_+)^{1-m}}{\tilde B^{p-1}\tilde\mu^{1-m}},
 \end{split}\end{equation*}
 since $\phi_0=\phi_0(x)$ is time independent and the first two terms on the right-hand side of \eqref{Cacformula1} vanishes.
 Since $\tilde\mu\leq\mu_+$,
we conclude that there exists a constant $\gamma_1$ depending only upon the data such that
 \begin{equation*}\begin{split}
   & \esssup_{\hat t<t<t_1} \frac{1}{r_0^n}\int_{B_0\times\{t\}}G\left(\frac{l_0-u}{l_0-\bar l}\right)\phi_0^k
 \,\mathrm {d}x
 \leq \gamma_1\nu_1+\gamma_1(\tilde B^{-p}+\tilde B^{1-p}).
  \end{split}\end{equation*}
  Consequently, we conclude that
$A_0(\bar l)\leq 2\gamma_1\nu_1+\gamma_1(\tilde B^{-p}+\tilde B^{1-p}).$
    At this stage, we  fix a number $\chi\in(0,1)$ which will be determined later. Now, we choose $\nu_1<1$ and $\tilde B>1$ be such that
     \begin{equation}\begin{split}\label{nu1}
  2\gamma_1\nu_1=\frac{\chi}{4}\qquad\text{and}\qquad\gamma_1(\tilde B^{-p}+\tilde B^{1-p})<\frac{\chi}{4}.
    \end{split}\end{equation}
    This yields that $A_0(\bar l)\leq\frac{1}{2}\chi$.
    The next thing to do is to determine the value of $l_1$.
    We first consider the case
    \begin{equation*}\iint_{B_{r_0}(x_1)\times\{t_1-r_0^p<t<t_1\}}\tilde u^{m-1}
    (l_0-u)_+^{(1+\lambda)(p-1)}\,\mathrm {d}x\,\mathrm {d}t=0.\end{equation*}
    Then, we apply Lemma \ref{Lebesguepoint+} to conclude that $u(x_1,t_1)\geq l_0=\mu_-+\xi\omega$, which proves \eqref{DeGiorgi2}.
    Next, we consider the case
     \begin{equation*}\iint_{B_{r_0}(x_1)\times\{t_1-r_0^p<t<t_1\}}\tilde u^{m-1}
     (l_0-u)_+^{(1+\lambda)(p-1)}\,\mathrm {d}x\,\mathrm {d}t>0.\end{equation*}
     We find that $A_0(l)\to+\infty$
    as $l\to l_0$.
    Noting that $A_0(l)$ defined in \eqref{Aj2nd}
    is continuous and increasing, then there exists a number $\tilde l\in (\bar l, l_0)$ such that $A_0(\tilde l)=\chi$.
    In view of \eqref{xil0} and  \eqref{omega1violated1}, we infer that for $B>8$ there holds
    \begin{equation}\label{l0barxi}l_0-\bar l\geq \frac{1}{8}\xi\omega>\frac{1}{4}(\alpha_{-1}-\alpha_0).\end{equation}
    At this point, we define
    \begin{equation}\label{l1xi}
	l_1=\begin{cases}
\tilde l,&\quad \text{if}\quad \tilde l<l_0-\frac{1}{4}(\alpha_{-1}-\alpha_0),\\
	l_0-\frac{1}{4}(\alpha_{-1}-\alpha_0),&\quad \text{if}\quad \tilde l\geq l_0-\frac{1}{4}(\alpha_{-1}-\alpha_0),
	\end{cases}
\end{equation}
and $d_0=l_0-l_1$.
Since $\tilde B\alpha_0<\xi\omega$, we have $l_1\geq  \bar l>\mu_-+\frac{1}{2}\tilde B\alpha_0+\frac{1}{4}\xi\omega$.

Step 2: \emph{Determine the sequence $\{l_j\}_{j=0}^\infty$.} Suppose that we have chosen two
sequences $l_1,\cdots,l_j$ and $d_0,\cdots,d_{j-1}$ such that for $i=1,\cdots,j$, there holds
    \begin{equation}\label{lixi}\frac{1}{2}\mu_-+\frac{1}{8}\xi\omega+\frac{1}{2}l_{i-1}+\frac{1}{4}\tilde B\alpha_{i-1}<l_i\leq
    l_{i-1}-\frac{1}{4}(\alpha_{i-2}-\alpha_{i-1}),
    \end{equation}
    \begin{equation}\label{Aj-1xi}
    A_{i-1}(l_i)\leq \chi,
     \end{equation}
      \begin{equation}\label{ljxi}
      l_i>\mu_-+\frac{1}{2}\tilde B\alpha_{i-1}+\frac{1}{4}\xi\omega.
       \end{equation}
       Then, we claim that
       \begin{equation}\label{Ajxi}
       A_j(\bar l)\leq \frac{1}{2}\chi,\qquad\text{where}\qquad \bar l=\frac{1}{2}l_j+\frac{1}{4}\tilde B\alpha_j+\frac{1}{8}\xi\omega+
       \frac{1}{2}\mu_-.
        \end{equation}
        To prove \eqref{Ajxi}, we first observe from \eqref{lixi} and  \eqref{ljxi} that $l_j-\bar l$ can be bounded below by
 \begin{equation}\begin{split}\label{upper bound for lxi}
 l_j-\bar l&
       =\frac{1}{2}l_j-\frac{1}{4}\tilde B\alpha_j-\frac{1}{8}\xi\omega-
       \frac{1}{2}\mu_-
   \\&
        \geq
        \frac{1}{4}(l_{j-1}-l_j)+\frac{1}{4}l_j
      +\frac{1}{16}\xi\omega+\dfrac{1}{8}\tilde B\alpha_{j-1} -\frac{1}{4}\tilde B\alpha_j-\frac{1}{8}\xi\omega-
       \frac{1}{4}\mu_-
       \\& \geq
        \frac{1}{4}(l_{j-1}-l_j)+\frac{1}{4}\left(\mu_-+\frac{1}{2}\tilde B\alpha_{j-1}+\frac{1}{4}\xi\omega\right)
      \\&\quad+\frac{1}{16}\xi\omega+\dfrac{1}{8}\tilde B\alpha_{j-1} -\frac{1}{4}\tilde B\alpha_j-\frac{1}{8}\xi\omega-
       \frac{1}{4}\mu_-
\\&=\frac{1}{4}(l_{j-1}-l_j)+\frac{1}{4}\tilde B(\alpha_{j-1}-\alpha_j).
 \end{split}\end{equation}
 To proceed further, we distinguish two cases. In the case $Q_j(\bar l)\nsubseteq\widetilde Q_1$.
We infer from \eqref{upper bound for lxi} that
  \begin{equation}\begin{split}\label{inclusion}
  \tilde\mu^{1-m}(l_{j-1}-l_j)^{2-p}r_{j-1}^p\geq 16\tilde\mu^{1-m}(l_j-\bar l)^{2-p}r_j^p>t_1-\hat t,
  \end{split}\end{equation}
  since
 $t_1-\tilde\mu^{1-m}(l_j-\bar l)^{2-p}r_j^p<\hat t$. This yields that $Q_{j-1}(l_j)\nsubseteq\widetilde Q_1$.
 We now decompose
   $L_j=L^\prime_j\cup L^{\prime\prime}_j$, where
   \begin{equation}\begin{split}\label{Ldecompositionxi}
   L^\prime_j=L_j\cap \left\{\frac{l_j-\tilde  u}{l_j-\bar l}\leq\epsilon_1\right\}\qquad\text{and}\qquad
   L^{\prime\prime}_j=L_j\setminus L^\prime_j.
   \end{split}\end{equation}
   Furthermore, we deduce from \eqref{ljxi} that $\{u\leq l_j\}=\{\tilde u\leq l_j\}$ and
    $l_j-\frac{1}{128}\xi\omega> \frac{1}{2}l_j$,
   where $\tilde u=\max\left\{u,\frac{1}{128}\xi\omega\right\}$.
   It follows that
    \begin{equation}\begin{split}\label{cutxi}
    &\iint_{L_j}\tilde u^{m-1}\left(\frac{l_j-u}{l_j-\bar l}
    \right)^{(1+\lambda)(p-1)}\phi_j^{k-p}
 \,\mathrm {d}x\mathrm {d}t
 \\&\leq\iint_{L_j\cap\left\{u>\frac{1}{128}\xi
 \omega\right\}}\tilde u^{m-1}\left(\frac{l_j-\tilde u}{l_j-\bar l}\right)
 ^{(1+\lambda)(p-1)}\phi_j^{k-p}
 \\&\quad+2^{(1+\lambda)(p-1)}\iint_{L_j\cap\left\{u\leq\frac{1}{128}\xi\omega\right\}}\tilde u^{m-1}
 \left(\frac{l_j-\frac{1}{128}\xi\omega}{l_j-\bar l}\right)^{(1+\lambda)(p-1)}\phi_j^{k-p}
 \,\mathrm {d}x\mathrm {d}t
 \\&\leq 2^{(1+\lambda)(p-1)}\iint_{L_j}\tilde u^{m-1}\left(\frac{l_j-\tilde u}
 {l_j-\bar l}\right)^{(1+\lambda)(p-1)}\phi_j^{k-p}
 \,\mathrm {d}x\mathrm {d}t
  \end{split}\end{equation}
  and hence
    \begin{equation*}\begin{split}
 A_j(\bar l)\leq&\gamma\frac{(l_j-\bar l)^{p-2}}{r_j^{n+p}}\iint_{L_j}\tilde u^{m-1}\left(\frac{l_j-\tilde u}
 {l_j-\bar l}\right)^{(1+\lambda)(p-1)}\phi_j^{k-p}
 \,\mathrm {d}x\mathrm {d}t
 \\&+\esssup_{\hat t<t<t_1}\frac{1}{r_j^n}\int_{B_j\times\{t\}}G\left(\frac{l_j-u}{l_j-\bar l}\right)\phi_j^{k}
 \,\mathrm {d}x,
 \end{split}\end{equation*}
 where the constant $\gamma$ depends only upon $p$ and $\lambda$.
    Our task now is to establish an estimate for $A_j(\bar  l)$.
Noting that $u\leq l_j$ on $L_j$, we have
     \begin{equation}\begin{split}\label{tildeu}
   \tilde u\leq\mu_-+\xi\omega=\tilde\mu\qquad\text{and}\qquad
   \frac{l_{j-1}-u}
 {l_{j-1}-l_j}\geq 1\qquad\text{on}\qquad L_j.
    \end{split}\end{equation}
Taking into account that $\phi_{j-1}(x)=1$ for $(x,t)\in Q_j$, we use \eqref{inclusion}, \eqref{Aj-1xi} and \eqref{tildeu} to deduce
  \begin{equation}\begin{split}\label{0st estimatexi}
 & \frac{(l_j-\bar l)^{p-2}}{r_j^{n+p}}\iint_{L_j}\tilde u^{m-1}
 \,\mathrm {d}x\mathrm {d}t\leq
 \frac{(l_j-\bar l)^{p-2}\tilde\mu^{m-1}}{r_j^{n+p}}|L_j|
 \\&\leq \gamma \frac{(l_j-\bar l)^{p-2}\tilde\mu^{m-1}(t_1-\hat t)}{r_j^{n+p}}\esssup_{\hat t<t<t_1}\int_{B_{j-1}}
\phi_{j-1}^k
 \,\mathrm {d}x
 \\&\leq \gamma \frac{1}{r_{j-1}^{n}}\esssup_{\hat t<t<t_1}\int_{B_{j-1}}
 G\left(\frac{l_{j-1}-u}
 {l_{j-1}-l_j}\right)\phi_{j-1}^k
 \,\mathrm {d}x
 \leq\gamma A_{j-1}(l_j)\leq \gamma \chi,
  \end{split}\end{equation}
  since $Q_{j-1}(l_j)\nsubseteq\widetilde Q_1$ and $A_{j-1}(l_j)$ is defined via \eqref{Aj2nd} with $l=l_j$.
 According to \eqref{0st estimatexi}, we have
 \begin{equation}\begin{split}\label{1st estimatexi}
 &\frac{(l_j-\bar l)^{p-2}}{r_j^{n+p}}\iint_{L_j^\prime}\tilde u^{m-1}\left(\frac{l_j-\tilde u}{l_j-\bar l}
 \right)^{(1+\lambda)(p-1)}\phi_j^{k-p}
 \,\mathrm {d}x\mathrm {d}t
 \\&\leq \frac{\tilde\mu^{m-1}(l_j-\bar l)^{p-2}}{r_j^{n+p}}\epsilon_1^{(1+\lambda)(p-1)}|L_j|
\leq \gamma\epsilon_1^{(1+\lambda)(p-1)}\chi.
 \end{split}\end{equation}
 On the other hand, for any fixed  $\epsilon_2<1$, we apply Young's inequality to conclude that
  \begin{equation*}\begin{split}
 &\frac{(l_j-\bar l)^{p-2}}{r_j^{n+p}}\iint_{L_j^{\prime\prime}}\tilde u^{m-1}\left(\frac{l_j-\tilde u}
 {l_j-\bar l}\right)^{(1+\lambda)(p-1)}\phi_j^{k-p}
 \,\mathrm {d}x\mathrm {d}t
 \\&\leq\epsilon_2\frac{\tilde \mu^{m-1}(l_j-\bar l)^{p-2}}{r_j^{n+p}}|L_j|
 \\&\quad+\gamma(\epsilon_2)\frac{\tilde\mu^{m-1}(l_j-\bar l)^{p-2}}{r_j^{n+p}}\iint_{L_j^{\prime\prime}}
 \left(\frac{l_j-\tilde u}{l_j-\bar l}\right)^{p\frac{n+h}{nh}}\phi_j^{(k-p)q}
 \,\mathrm {d}x\mathrm {d}t
 \\&=:T_1+T_2,
  \end{split}\end{equation*}
  with the obvious meanings of $T_1$ and $T_2$. Here, parameters $h$ and $q$ are chosen via \eqref{hq}.
In view of \eqref{0st estimatexi}, we deduce that $T_1\leq \gamma\epsilon_2\chi$.
To estimate $T_2$, we set
   \begin{equation}\begin{split}\label{tildepsixi}\tilde\psi_j(x,t)=\frac{1}{l_j-\bar l}\left[\int_{\tilde u}^{l_j}
   \left(1+\frac{l_j-s}{l_j-\bar l}\right)^{-\frac{1}{p}-\frac{\lambda}{p}}\,\mathrm {d}s\right]_+\end{split}\end{equation}
  and use Lemma \ref{lemmainequalitypsi-} with $(l,d)$ replaced by $(l_j,l_j-\bar l)$ to deduce that
   \begin{equation}\begin{split}\label{T21xi}
   T_2\leq \gamma\frac{\tilde\mu^{m-1}(l_j-\bar l)^{p-2}}{r_j^{n+p}}\iint_{L_j^{\prime\prime}}
   \tilde\psi_j^{p\frac{n+h}{n}}\phi_j^{(k-p)q} \,\mathrm {d}x\mathrm {d}t.
   \end{split}\end{equation}
   To proceed further, we
   set $v=\tilde\psi_j\phi_j^{k_1}$, where $k_1=\frac{(k-p)nq}{p(n+h)}$. Using the similar arguments as in the proof of Step 3
   of Lemma \ref{lemmaDeGiorgi1}, we can deduce the following estimate
   \begin{equation}\begin{split}\label{T22xi}
\iint_{L_j^{\prime\prime}}&
   \tilde\psi_j^{p\frac{n+h}{n}}\phi_j^{(k-p)q} \,\mathrm {d}x\mathrm {d}t=
   \iint_{L_j^{\prime\prime}}
v^{p+\frac{ph}{n}} \,\mathrm {d}x\mathrm {d}t
\\&\leq \int_{\hat t}^{t_1}\left(\int_{B_j}
v^{\frac{np}{n-p}} \,\mathrm {d}x\right)^{\frac{n-p}{n}}
\left(\int_{L_j^{\prime\prime}(t)}
v^{h} \,\mathrm {d}x\right)^{\frac{p}{n}}
\mathrm {d}t
\\&\leq  \gamma\esssup_{\hat t<t<t_1}\left(\int_{L_j^{\prime\prime}(t)}
\frac{l_j-\tilde u}{l_j-\bar l}\phi_j^{k_1h} \,\mathrm {d}x\right)^{\frac{p}{n}}\iint_{Q_j}
|Dv|^p \,\mathrm {d}x\mathrm {d}t,
   \end{split}\end{equation}
 where
  \begin{equation}
  \label{Ljjjj}
  L_j^{\prime\prime}(t)=\{x\in B_j:u(\cdot,t)\leq l_j\}\cap \left\{x\in B_j:\frac{l_j-\tilde u(\cdot,t)}{l_j-\bar l}>
   \epsilon_1\right\}.\end{equation}
   In view of \eqref{Gkv}, we conclude with
   \begin{equation}\begin{split}\label{T23xi}
&   \int_{L_j^{\prime\prime}(t)}
\frac{l_j-\tilde u}{l_j-\bar l}\phi_j^{k_1h} \,\mathrm {d}x
\\&\leq c(\epsilon_1)
\int_{L_j^{\prime\prime}(t)}
G\left(\frac{l_j-\tilde u}{l_j-\bar l} \right)\phi_j^{k_1h}\,\mathrm {d}x
\leq c(\epsilon_1)
\int_{B_j}
G\left(\frac{l_j-u}{l_j-\bar l} \right)\phi_j^{k_1h}\,\mathrm {d}x.
    \end{split}\end{equation}
    At this stage, we use Lemma \ref{Cac1} with $(l,d,\theta)$ replaced by $(l_j,l_j-\bar l, t_1-\hat t)$ to deduce that
     \begin{equation*}\begin{split}
  &\esssup_{\hat t<t<t_1}
  \frac{1}{r_j^n}
  \int_{B_j}
G\left(\frac{l_j-u}{l_j-\bar l} \right)\phi_j^{k_1h}\,\mathrm {d}x
  \\
 \leq &
 \gamma  \frac{(l_j-\bar l)^{p-2}}{r_j^{p+n}}\iint_{L_j}u^{m-1}\left(\frac{l_j-u}{l_j-\bar l}\right)^{(1+\lambda)(p-1)}
  \phi_j^{k_1h-p}\,\mathrm {d}x\mathrm {d}t
   \\&+\gamma\frac{t_1-\hat t}{(l_j-\bar l)^2r_j^n}\int_{B_j}g^{\frac{p}{p-1}}\,\mathrm {d}x
  +\gamma \frac{t_1-\hat t}{(l_j-\bar l)r_j^n}\int_{B_j}|f|\,\mathrm {d}x
  \\=&:T_3+T_4+T_5,
   \end{split}\end{equation*}
   with the obvious meanings of $T_3$-$T_5$.
   We first consider the estimate for $T_3$.
We apply \eqref{0st estimatexi}, \eqref{upper bound for lxi} and \eqref{Aj-1xi} to find that
    \begin{equation}\begin{split}\label{T3xi}
    T_3\leq &\gamma\frac{(l_j-\bar l)^{p-2-(1+\lambda)(p-1)}}{r_j^{p+n}}\iint_{L_j}\tilde u^{m-1}
    (l_j-u)^{(1+\lambda)(p-1)}
  \phi_j^{k_1h-p}\,\mathrm {d}x\mathrm {d}t
  \\ \leq &\gamma\frac{(l_{j-1}-l_j)^{p-2}}{r_{j-1}^{p+n}}\iint_{L_{j-1}}\tilde u^{m-1}
  \left(\frac{l_{j-1}-u}{l_{j-1}-l_j}\right)^{(1+\lambda)(p-1)}
  \phi_{j-1}^{k-p}\,\mathrm {d}x\mathrm {d}t
  \\ \leq & \gamma A_{j-1}(l_j)\leq \gamma_1\chi,
  \end{split}\end{equation}
  where the constant  $\gamma_1$ depends only upon the data.
Finally, we infer from
    \eqref{alpha1xi}, \eqref{upper bound for lxi} and \eqref{inclusion} that
    \begin{equation*}\begin{split}T_4+T_5\leq
    \gamma\frac{\tilde\mu^{1-m}}{(l_j-\bar l)^p}r_j^{p-n}\int_{B_j}g^{\frac{p}{p-1}}\,\mathrm {d}x
  +\gamma \frac{\tilde\mu^{1-m}}{(l_j-\bar l)^{p-1}}r_j^{p-n}\int_{B_j}|f|\,\mathrm {d}x\leq
    \gamma (\tilde B^{-p}+\tilde B^{-(p-1)}).\end{split}\end{equation*}
    Consequently,
    we infer that
    \begin{equation}\begin{split}\label{important Gxi}
  \esssup_{\hat t<t<t_1}&
  \frac{1}{r_j^n}
  \int_{B_j}
G\left(\frac{l_j-u}{l_j-\bar l} \right)\phi_j^{k_1h}\,\mathrm {d}x\leq \gamma\chi+\gamma (\tilde B^{-p}+\tilde B^{-(p-1)}).
    \end{split}\end{equation}
We now turn our attention to the estimate of $T_2$.
    Combining the estimates \eqref{T21xi}-\eqref{important Gxi}, we can rewrite the upper bound for $T_2$ by
    \begin{equation*}\begin{split}
    T_2\leq & \gamma \frac{\tilde\mu^{m-1}(l_j-\bar l)^{p-2}}{r_j^n}(\chi+\tilde B^{-p}+\tilde B^{-(p-1)})^{\frac{p}{n}}
    \\&\times\left[\iint_{Q_j}
\phi_j^{k_1p}|D\tilde \psi_j|^p \,\mathrm {d}x\mathrm {d}t
+\iint_{Q_j}
\phi_j^{(k_1-1)p}\tilde \psi_j^p|D\phi_j|^p \,\mathrm {d}x\mathrm {d}t\right]
\\=&:\gamma(\chi+\tilde B^{-p}+\tilde B^{-(p-1)})^{\frac{p}{n}}(T_6+T_7),
    \end{split}\end{equation*}
    with the obvious meanings of $T_6$ and $T_7$. We first consider the  estimate for $T_6$.
Noting that
 \begin{equation}\begin{split}\label{tildemutildeu}\tilde\mu^{m-1}\leq  \gamma_m(\mu_-^{m-1}+(\xi\omega)^{m-1})\leq \gamma\tilde u^{m-1}\quad\text{on}\quad
 \hat Q=B_R\times\left(-(\mu_+)^{1-m}A^{p-2}\omega^{2-p}R^p,0\right),\end{split}\end{equation}
we conclude that
    \begin{equation*}\begin{split}
    T_6&=\frac{\tilde\mu^{m-1}(l_j-\bar l)^{p-2}}{r_j^n}
    \iint_{Q_j}
\phi_j^{k_1p}|D\tilde \psi_j|^p \,\mathrm {d}x\mathrm {d}t
\leq \gamma \frac{(l_j-\bar l)^{p-2}}{r_j^n}
    \iint_{Q_j}\tilde u^{m-1}
|D\tilde \psi_j|^p \phi_j^{k_1p}\,\mathrm {d}x\mathrm {d}t.
    \end{split}\end{equation*}
    Taking into account that $\{\tilde u<l_j\}=\{u<l_j\}$, we infer from \eqref{tildepsixi} that
    \begin{equation}\begin{split}\label{reduction}
    \tilde u^{\frac{m-1}{p}}
|D\tilde \psi_j|&=u^{\frac{m-1}{p}}|Du|\frac{1}{l_j-\bar l}
   \left(1+\frac{l_j-u}{l_j-\bar l}\right)^{-\frac{1}{p}-\frac{\lambda}{p}}\chi_{\left\{\frac{1}{128}\xi\omega<u<l_j\right\}}
  \\&\leq u^{\frac{m-1}{p}}|Du|\frac{1}{l_j-\bar l}
   \left(1+\frac{l_j-u}{l_j-\bar l}\right)^{-\frac{1}{p}-\frac{\lambda}{p}}
\chi_{\left\{u<l_j\right\}}=u^{\frac{m-1}{p}}|D\psi_j|,
     \end{split}\end{equation}
     where
        \begin{equation*}\begin{split}\psi_j(x,t)=\frac{1}{l_j-\bar l}\left[\int_{u}^{l_j}
   \left(1+\frac{l_j-s}{l_j-\bar l}\right)^{-\frac{1}{p}-\frac{\lambda}{p}}\,\mathrm {d}s\right]_+.\end{split}\end{equation*}
  At this stage, we use Lemma \ref{Cac1}
  with $(l,d,\theta)$ replaced by $(l_j,l_j-\bar l, t_1-\hat t)$
  and taking into account the estimates for $T_3$-$T_5$. We conclude that
  \begin{equation*}\begin{split}
  T_6\leq &\gamma \frac{(l_j-\bar l)^{p-2}}{r_j^n}
    \iint_{Q_j}u^{m-1}
|D\psi_j|^p \phi_j^{k_1p}\,\mathrm {d}x\mathrm {d}t
  \\
 \leq &\gamma  \frac{(l_j-\bar l)^{p-2}}{r_j^p}\iint_{L_j}u^{m-1}\left(\frac{l_j-u}{l_j-\bar l}\right)^{(1+\lambda)(p-1)}
 \phi_j^{(k_1-1)p}\,\mathrm {d}x\mathrm {d}t
  \\&
+\gamma\frac{t_1-\hat t}{(l_j-\bar l)^2r_j^n}\int_{B_j}g^{\frac{p}{p-1}}\,\mathrm {d}x
  +\gamma \frac{t_1-\hat t}{(l_j-\bar l)r_j^n}\int_{B_j}|f|\,\mathrm {d}x
  \\ \leq &
  \gamma\frac{(l_{j-1}-l_j)^{p-2}}{r_{j-1}^{p+n}}\iint_{L_{j-1}}\tilde u^{m-1}
  \left(\frac{l_{j-1}-u}{l_{j-1}-l_j}\right)^{(1+\lambda)(p-1)}
  \phi_{j-1}^{k-p}\,\mathrm {d}x\mathrm {d}t
    \\&+\gamma(\tilde B^{-p}+\tilde B^{-(p-1)})
 \\ \leq&\gamma(\chi+\tilde B^{-p}+\tilde B^{-(p-1)}).
  \end{split}\end{equation*}
  Next, we consider the estimate for $T_7$.
  Similar to the estimate \eqref{psitilde}, we have
   \begin{equation}\begin{split}\label{tildepsiupper}\tilde\psi_j(x,t)
   \leq \frac{(l_j-\tilde u)_+^{\frac{1}{p^\prime}}}{(l_j-\bar l)^{\frac{1}{p^\prime}}}.
   \end{split}\end{equation}
Consequently, we apply \eqref{tildemutildeu}, \eqref{tildepsiupper} and \eqref{T3xi} to conclude that
    \begin{equation*}\begin{split}
  T_7&\leq  \frac{\tilde\mu^{m-1}(l_j-\bar l)^{p-2}}{r_j^{n+p}}\iint_{L_j}
\left(\frac{l_j-\tilde u}{l_j-\bar l}\right)^{p-1}
\phi_j^{(k_1-1)p} \,\mathrm {d}x\mathrm {d}t
\\&\leq \frac{\tilde\mu^{m-1}(l_j-\bar l)^{p-2}}{r_j^{n+p}}|L_j|
\\&\quad+
\frac{(l_j-\bar l)^{p-2}}{r_j^{n+p}}\iint_{L_j}\tilde u^{m-1}
\left(\frac{l_j-u}{l_j-\bar l}\right)^{(1+\lambda)(p-1)}
\phi_j^{(k_1-1)p} \,\mathrm {d}x\mathrm {d}t
\\&\leq \gamma \chi.
  \end{split}\end{equation*}
Then, we arrive at
$T_2\leq \gamma(\chi+\tilde B^{-p}+\tilde B^{-(p-1)})^{1+\frac{p}{n}}$
  and hence
   \begin{equation}\begin{split}\label{Lprimeprimexi}
  &\frac{(l_j-\bar l)^{p-2}}{r_j^{n+p}}\iint_{L_j^{\prime\prime}}\tilde u^{m-1}\left(\frac{l_j-\tilde u}
  {l_j-\bar l}\right)^{(1+\lambda)(p-1)}\phi_j^{k-p}
 \,\mathrm {d}x\mathrm {d}t
 \\&\leq \gamma\epsilon_2\chi+\gamma(\epsilon_2)(\chi+\tilde B^{-p}+\tilde B^{-(p-1)})^{1+\frac{p}{n}}.
 \end{split}\end{equation}
 This yields that
  \begin{equation}\begin{split}\label{AA1xi}
  &\frac{(l_j-\bar l)^{p-2}}{r_j^{n+p}}\iint_{L_j}\tilde u^{m-1}\left(\frac{l_j-\tilde u}{l_j-\bar l}
  \right)^{(1+\lambda)(p-1)}\phi_j^{k-p}
 \,\mathrm {d}x\mathrm {d}t
 \\&\leq \gamma\epsilon_1^{(1+\lambda)(p-1)}\chi+
\gamma\epsilon_2\chi+\gamma(\epsilon_2)(\chi+\tilde B^{-p}+\tilde B^{-(p-1)})^{1+\frac{p}{n}}.
 \end{split}\end{equation}
 Furthermore, we apply Lemma \ref{Cac1} to obtain
 \begin{equation}\begin{split}\label{estimateforGxi}
  &\esssup_t
  \frac{1}{r_j^n}
  \int_{B_j}
G\left(\frac{l_j-u}{l_j-\bar l} \right)\phi_j^k\,\mathrm {d}x
  \\
 \leq &
 \gamma  \frac{(l_j-\bar l)^{p-2}}{r_j^{p+n}}\iint_{L_j}u^{m-1}\left(\frac{l_j-u}{l_j-\bar l}\right)^{(1+\lambda)(p-1)}
  \phi_j^{k-p}\,\mathrm {d}x\mathrm {d}t
   \\&+\gamma\frac{t_1-\hat t}{(l_j-\bar l)^2r_j^n}\int_{B_j}g^{\frac{p}{p-1}}\,\mathrm {d}x
  +\gamma \frac{t_1-\hat t}{(l_j-\bar l)r_j^n}\int_{B_j}|f|\,\mathrm {d}x
  \\=&:S_1+S_2+S_3,
   \end{split}\end{equation}
   with the obvious meanings of $S_1$-$S_3$. To estimate $S_1$, we use \eqref{AA1xi} and \eqref{cutxi} to infer that
   \begin{equation*}\begin{split}
   S_1&\leq
   \gamma\frac{(l_j-\bar l)^{p-2}}{r_j^{p+n}}\iint_{L_j}\tilde u^{m-1}
   \left(\frac{l_j-\tilde u}{l_j-\bar l}\right)^{(1+\lambda)(p-1)}
  \phi_j^{k-p}\,\mathrm {d}x\mathrm {d}t
  \\&\leq
   \gamma\left[\epsilon_1^{(1+\lambda)(p-1)}\chi+
 \epsilon_2\chi+\gamma(\epsilon_2)(\chi+\tilde B^{-p}+\tilde B^{-(p-1)})^{1+\frac{p}{n}}\right].
   \end{split}\end{equation*}
     Similar to the estimates of $T_4$ and $T_5$, we obtain $S_2+S_3\leq \gamma (\tilde B^{-p}+\tilde
     B^{-(p-1)})$. Inserting the estimates for $S_1$-$S_3$
     into \eqref{estimateforGxi} and taking \eqref{AA1xi} into consideration, we conclude that there exist constants
     $\gamma_1=\gamma_1(\text{data})$ and $\gamma_2=\gamma_2(\text{data},\epsilon_1,\epsilon_2)$ such that
     \begin{equation}\begin{split}\label{Ajinclusion}
     A_j(\bar l)\leq &\gamma_1 (\tilde B^{-p}+\tilde B^{-(p-1)})+
   \gamma_1(\epsilon_1+
 \epsilon_2)\chi
 +\gamma_2(\chi+\tilde B^{-p}+\tilde B^{-(p-1)})^{1+\frac{p}{n}}.
      \end{split}\end{equation}
%      At this point, we first choose $\epsilon_1$ and $\epsilon_2$ be such that
 %\begin{equation}\begin{split}\label{epsilonxi1}
 %\epsilon_1+\epsilon_2<\frac{1}{8\gamma_1}
 %\end{split}\end{equation}
%and then we determine the value of $\chi$ by
%  \begin{equation}\begin{split}\label{chixi1}
 % \chi<\frac{1}{100^{\frac{n}{p}}\gamma_2^{\frac{n}{p}}}.
 % \end{split}\end{equation}
% Finally, with the choices of $\epsilon_1$, $\epsilon_2$ and $\chi$, we set $B$ so large that
 % \begin{equation*}\begin{split}
 %\tilde B^{-p}+\tilde B^{1-p}<\min\left\{\frac{1}{100\gamma_1}\chi,\left(\frac{1}{100\gamma_2
% }\chi\right)^\frac{1}{1+\frac{p}{n}}\right\}.
 %  \end{split}\end{equation*}
  %    With the choices of $\epsilon_1$, $\epsilon_2$, $\chi$ and $B$,
   %   we obtain $A_j(\bar l)\leq \frac{1}{2}\chi$, which completes the proof of
    %  \eqref{Ajxi} for the case $Q_j(\bar l)\nsubseteq\widetilde Q_1$.
We now turn our attention to consider the case $Q_j(\bar l)\subseteq\widetilde Q_1$. For simplicity, we may take
$\varphi_{j-1}=\varphi_{j-1}(l_j)$.
We decompose
   $L_j(\bar l)=L^\prime_j(\bar l)\cup L^{\prime\prime}_j(\bar l)$, where
   \begin{equation}\begin{split}\label{Ldecompositionxixi}
   L^\prime_j(\bar l)=L_j(\bar l)\cap \left\{\frac{l_j-\tilde  u}{l_j-\bar l}\leq\epsilon_1\right\}\qquad\text{and}\qquad
   L^{\prime\prime}_j(\bar l)=L_j(\bar l)\setminus L^\prime_j(\bar l).
   \end{split}\end{equation}
Similar to \eqref{cutxi}, we have
    \begin{equation}\begin{split}\label{cutxixi}\iint_{L_j(\bar l)}&\tilde u^{m-1}\left(\frac{l_j-u}{l_j-\bar l}\right)^{(1+\lambda)(p-1)}\varphi_j(\bar l)^{k-p}
 \,\mathrm {d}x\mathrm {d}t
 \\&\leq 2^{(1+\lambda)(p-1)}\iint_{L_j(\bar l)}\tilde u^{m-1}\left(\frac{l_j-\tilde u}{l_j-\bar l}\right)^{(1+\lambda)(p-1)}\varphi_j(\bar l)^{k-p}
 \,\mathrm {d}x\mathrm {d}t
  \end{split}\end{equation}
  and hence
    \begin{equation*}\begin{split}
 A_j(\bar l)\leq&\gamma\frac{(l_j-\bar l)^{p-2}}{r_j^{n+p}}\iint_{L_j(\bar l)}\tilde u^{m-1}\left(\frac{l_j-\tilde u}
 {l_j-\bar l}\right)^{(1+\lambda)(p-1)}\varphi_j(\bar l)^{k-p}
 \,\mathrm {d}x\mathrm {d}t
 \\&+\esssup_{t}\frac{1}{r_j^n}\int_{B_j\times\{t\}}G\left(\frac{l_j-u}{l_j-\bar l}\right)\varphi_j(\bar l)^{k}
 \,\mathrm {d}x,
 \end{split}\end{equation*}
 where the constant $\gamma$ depends only upon $p$ and $\lambda$.
 The next thing to do in the proof is to obtain an upper bound for $A_j(\bar l)$. To this end,
we first observe that $u\leq l_j$ on $L_j(\bar l)$
 and hence
  \begin{equation*}\begin{split}\frac{l_{j-1}-u}
 {l_{j-1}-l_j}\geq 1\qquad\text{on}\quad L_j(\bar l).
  \end{split}\end{equation*}
 In the case $Q_{j-1}(l_j)\nsubseteq\widetilde Q_1$, we have $Q_j(\bar l)\subset Q_{j-1}$ and
$\phi_{j-1}(x)=1$ for $(x,t)\in Q_j(\bar l)$. Then, we use \eqref{Aj-1xi} to deduce
  \begin{equation*}\begin{split}
 \frac{(l_j-\bar l)^{p-2}\tilde\mu^{m-1}}{r_j^{n+p}}|L_j(\bar l)|
 &\leq \gamma \frac{1}{r_{j-1}^{n}}\esssup_{\hat t<t<t_1}\int_{B_{j-1}}
 G\left(\frac{l_{j-1}-u}
 {l_{j-1}-l_j}\right)\phi_{j-1}^k
 \,\mathrm {d}x
\\& \leq\gamma A_{j-1}(l_j)\leq \gamma \chi.
  \end{split}\end{equation*}
Moreover, we apply \eqref{Aj-1xi} and \eqref{upper bound for lxi} to conclude that
  \begin{equation*}\begin{split}
  &\frac{(l_j-\bar l)^{p-2}}{r_j^{n+p}}\iint_{L_j(\bar l)}\tilde u^{m-1}\left(\frac{l_j- u}{l_j-\bar l}
 \right)^{(1+\lambda)(p-1)}
 \,\mathrm {d}x\mathrm {d}t
 \\&\leq \gamma\frac{(l_{j-1}-l_j)^{p-2-(1+\lambda)(p-1)}}{r_j^{n+p}}\iint_{L_j(\bar l)}\tilde u^{m-1}(l_{j-1}- u)
 ^{(1+\lambda)(p-1)}
 \,\mathrm {d}x\mathrm {d}t
 \\&\leq \gamma\frac{(l_{j-1}-l_j)^{p-2}}{r_j^{n+p}}\iint_{L_{j-1}}\tilde u^{m-1}\left(\frac{l_{j-1}- u}{l_{j-1}-l_j}
 \right)^{(1+\lambda)(p-1)}\phi_{j-1}^{k-p}
 \,\mathrm {d}x\mathrm {d}t
 \\&\leq \gamma A_{j-1}(l_j)\leq\gamma\chi.
 \end{split}\end{equation*}
  Next, we consider the case $Q_{j-1}(l_j)\subseteq\widetilde Q_1$.
In view of \eqref{upper bound for lxi}, we see that
$\tilde\mu^{1-m}(l_j-\bar l)^{2-p}r_j^p\leq\frac{1}{16}\tilde\mu^{1-m}(l_{j-1}-l_j)^{2-p}r_{j-1}^p.$
This yields that $Q_j(\bar l)\subset Q_{j-1}(l_j)$ and $\varphi_{j-1}(x,t)=1$ for $(x,t)\in Q_j(\bar l)$. It follows that
    \begin{equation*}\begin{split}
 \frac{(l_j-\bar l)^{p-2}\tilde\mu^{m-1}}{r_j^{n+p}}|L_j(\bar l)|
 &\leq \gamma \frac{1}{r_{j-1}^{n}}\esssup_{t}\int_{B_{j-1}}
 G\left(\frac{l_{j-1}-u}
 {l_{j-1}-l_j}\right)\varphi_{j-1}^k
 \,\mathrm {d}x
\\& \leq\gamma A_{j-1}(l_j)\leq \gamma \chi
  \end{split}\end{equation*}
  and
   \begin{equation*}\begin{split}
  \frac{(l_j-\bar l)^{p-2}}{r_j^{n+p}}&\iint_{L_j(\bar l)}\tilde u^{m-1}\left(\frac{l_j- u}{l_j-\bar l}
 \right)^{(1+\lambda)(p-1)}
 \,\mathrm {d}x\mathrm {d}t
 \\&\leq \gamma\frac{(l_{j-1}-l_j)^{p-2}}{r_j^{n+p}}\iint_{L_{j-1}(l_j)}\tilde u^{m-1}\left(\frac{l_{j-1}- u}{l_{j-1}-l_j}
 \right)^{(1+\lambda)(p-1)}\varphi_{j-1}^{k-p}
 \,\mathrm {d}x\mathrm {d}t\\&\leq \gamma A_{j-1}(l_j)\leq\gamma\chi.
 \end{split}\end{equation*}
 Consequently, we conclude that the estimates
   \begin{equation}\begin{split}\label{0st estimatexixi1}
 \frac{(l_j-\bar l)^{p-2}\tilde\mu^{m-1}}{r_j^{n+p}}|L_j(\bar l)|
\leq \gamma \chi
  \end{split}\end{equation}
  and
    \begin{equation}\begin{split}\label{0st estimatexixi2}
  \frac{(l_j-\bar l)^{p-2}}{r_j^{n+p}}&\iint_{L_j(\bar l)}\tilde u^{m-1}\left(\frac{l_j- u}{l_j-\bar l}
 \right)^{(1+\lambda)(p-1)}
 \,\mathrm {d}x\mathrm {d}t
\leq\gamma\chi
 \end{split}\end{equation}
 hold for either $Q_{j-1}(l_j)\nsubseteq\widetilde Q_1$ or $Q_{j-1}(l_j)\subseteq\widetilde Q_1$.
 According to \eqref{0st estimatexixi1} and \eqref{tildeu}, we deduce that
 \begin{equation}\begin{split}\label{1st estimatexixi}
 &\frac{(l_j-\bar l)^{p-2}}{r_j^{n+p}}\iint_{L_j^\prime(\bar l)}\tilde u^{m-1}\left(\frac{l_j-\tilde u}{l_j-\bar l}
 \right)^{(1+\lambda)(p-1)}\varphi_j(\bar l)^{k-p}
 \,\mathrm {d}x\mathrm {d}t
 \\&\leq \frac{\tilde \mu^{m-1}(l_j-\bar l)^{p-2}}{r_j^{n+p}}\epsilon_1^{(1+\lambda)(p-1)}|L_j(\bar l)|
\leq \gamma\epsilon_1^{(1+\lambda)(p-1)}\chi.
 \end{split}\end{equation}
 Moreover, for any fixed  $\epsilon_2<1$, we apply Young's inequality and \eqref{tildeu} to conclude that
  \begin{equation*}\begin{split}
 &\frac{(l_j-\bar l)^{p-2}}{r_j^{n+p}}\iint_{L_j^{\prime\prime}(\bar l)}\tilde u^{m-1}\left(\frac{l_j-\tilde u}
 {l_j-\bar l}\right)^{(1+\lambda)(p-1)}\varphi_j(\bar l)^{k-p}
 \,\mathrm {d}x\mathrm {d}t
 \\&\leq\epsilon_2\frac{\tilde\mu^{m-1}(l_j-\bar l)^{p-2}}{r_j^{n+p}}|L_j(\bar l)|
 \\&\quad+\gamma(\epsilon_2)\frac{\tilde\mu^{m-1}(l_j-\bar l)^{p-2}}{r_j^{n+p}}\iint_{L_j^{\prime\prime}(\bar l)}
 \left(\frac{l_j-\tilde u}{l_j-\bar l}\right)^{p\frac{n+h}{nh}}\varphi_j(\bar l)^{(k-p)q}
 \,\mathrm {d}x\mathrm {d}t
 \\&=:\tilde T_1+\tilde T_2,
  \end{split}\end{equation*}
  with the obvious meanings of $\tilde T_1$ and $\tilde  T_2$. In view of \eqref{0st estimatexixi1}, we obtain
  $\tilde T_1\leq \gamma\epsilon_2\chi$.
Next, we use Lemma \ref{lemmainequalitypsi-} with $(l,d)$ replaced by $(l_j,l_j-\bar l)$ to deduce that
   \begin{equation}\begin{split}\label{T21xixi}
  \tilde T_2\leq \gamma\frac{\tilde\mu^{m-1}(l_j-\bar l)^{p-2}}{r_j^{n+p}}\iint_{L_j^{\prime\prime}(\bar l)}
   \tilde\psi_j^{p\frac{n+h}{n}}\varphi_j(\bar l)^{(k-p)q} \,\mathrm {d}x\mathrm {d}t,
   \end{split}\end{equation}
   where $\tilde\psi_j$ is the function defined in \eqref{tildepsixi}.
   Let $v=\tilde\psi_j\varphi_j^{k_1}$, where $k_1=\frac{(k-p)nq}{p(n+h)}$. We adopt the same procedure
   as in the proof of Step 3
   of Lemma \ref{lemmaDeGiorgi1}. This yields that
   \begin{equation}\begin{split}\label{T22xixi}
&\iint_{L_j^{\prime\prime}(\bar l)}
   \tilde\psi_j^{p\frac{n+h}{n}}\varphi_j(\bar l)^{(k-p)q} \,\mathrm {d}x\mathrm {d}t\\&\leq \int_{t_1- (\mu_+)^{1-m}(l_j-\bar l)^{2-p}r_j^p}^{t_1}\left(\int_{B_j}
v^{\frac{np}{n-p}} \,\mathrm {d}x\right)^{\frac{n-p}{n}}
\left(\int_{L_j^{\prime\prime}(t)}
v^{h} \,\mathrm {d}x\right)^{\frac{p}{n}}
\mathrm {d}t
\\&\leq  \gamma\esssup_t\left(\int_{L_j^{\prime\prime}(t)}
\frac{l_j-\tilde u}{l_j-\bar l}\varphi_j(\bar l)^{k_1h} \,\mathrm {d}x\right)^{\frac{p}{n}}\iint_{Q_j(\bar l)}
|Dv|^p \,\mathrm {d}x\mathrm {d}t,
   \end{split}\end{equation}
   where $L_j^{\prime\prime}(t)$ is the set defined in \eqref{Ljjjj}. According to Lemma \ref{Gk}, we deduce that
   \begin{equation}\begin{split}\label{T23xixi}
&   \int_{L_j^{\prime\prime}(t)}
\frac{l_j-\tilde u}{l_j-\bar l}\varphi_j(\bar l)^{k_1h} \,\mathrm {d}x
\leq c(\epsilon_1)
\int_{B_j}
G\left(\frac{l_j-u}{l_j-\bar l} \right)\varphi_j(\bar l)^{k_1h}\,\mathrm {d}x.
    \end{split}\end{equation}
    At this point, we apply Lemma \ref{Cac1} with $(l,d,\theta)$ replaced by  $(l_j,l_j-\bar l,\tilde \mu^{1-m}(l_j-\bar l)^{2-p}r_j^p)$
    to conclude that
     \begin{equation*}\begin{split}
  &\esssup_t
  \frac{1}{r_j^n}
  \int_{B_j}
G\left(\frac{l_j-u}{l_j-\bar l} \right)\varphi_j(\bar l)^{k_1h}\,\mathrm {d}x
  \\
 \leq &
 \gamma  \frac{(l_j-\bar l)^{p-2}}{r_j^{p+n}}\iint_{L_j(\bar l)}u^{m-1}\left(\frac{l_j-u}{l_j-\bar l}\right)^{(1+\lambda)(p-1)}
  \varphi_j(\bar l)^{k_1h-p}\,\mathrm {d}x\mathrm {d}t
  \\&+\gamma
    \frac{1}{r_j^n}\iint_{L_j(\bar l)}\frac{l_j-u}{l_j-\bar l}|\partial_t\varphi_j(\bar l)|\,\mathrm {d}x\mathrm {d}t
   \\&+\gamma \frac{r_j^{p-n}}{(l_j-\bar l)^p}\tilde\mu^{1-m}\int_{B_j}g^{\frac{p}{p-1}}\,\mathrm {d}x
  +\gamma \frac{r_j^{p-n}}{(l_j-\bar l)^{p-1}}\tilde\mu^{1-m}\int_{B_j}|f|\,\mathrm {d}x
  \\=&:\tilde T_3+\tilde T_4+\tilde T_5+\tilde T_6,
   \end{split}\end{equation*}
   with the obvious meanings of $\tilde T_3$-$\tilde T_6$. In view of \eqref{0st estimatexixi2}, we find that $\tilde T_3\leq\gamma\chi$,
   since $u\leq \tilde u$ on $L_j(\bar l)$.
Noting that $(1+\lambda)(p-1)>1$ and
  $|\partial_t\varphi_j(\bar l)|\leq 9(l_j-\bar l)^{p-2}\tilde\mu^{m-1}r_j^{-p}$, we infer from \eqref{tildemutildeu},
  \eqref{0st estimatexixi1} and \eqref{0st estimatexixi2} that
     \begin{equation*}\begin{split}
    \tilde T_4\leq &\gamma
    \frac{(l_j-\bar l)^{p-2}\tilde\mu^{m-1}}{r_j^{n+p}}\iint_{L_j(\bar l)}\frac{l_j-u}
    {l_j-\bar l}\,\mathrm {d}x\mathrm {d}t
    \\ \leq &\gamma\frac{(l_j-\bar l)^{p-2}\tilde\mu^{m-1}}{r_j^{n+p}}|L_j(\bar l)|
  +\gamma
    \frac{(l_j-\bar l)^{p-2}}{r_j^{n+p}}\iint_{L_j(\bar l)}\tilde u^{m-1}\left(\frac{l_j-u}
    {l_j-\bar l}\right)^{(1+\lambda)(p-1)}\,\mathrm {d}x\mathrm {d}t
  \\ \leq & \gamma\chi+A_{i-1}(l_i)\leq \gamma_2\chi,
    \end{split}\end{equation*}
  where the constant  $\gamma_2$ depends only upon the data. Finally, we infer from
    \eqref{alpha1xi} and \eqref{upper bound for lxi} that $\tilde T_5+\tilde T_6\leq \gamma (\tilde B^{-p}+\tilde B^{-(p-1)})$.
Combining the above estimates,
    we arrive at
    \begin{equation}\begin{split}\label{important Gxixi}
  \esssup_t&
  \frac{1}{r_j^n}
  \int_{B_j}
G\left(\frac{l_j-u}{l_j-\bar l} \right)\varphi_j(\bar l)^{k_1h}\,\mathrm {d}x\leq \gamma\chi+\gamma (\tilde B^{-p}+\tilde B^{-(p-1)}).
    \end{split}\end{equation}
We now turn our attention to the estimate of $\tilde T_2$.
    Combining \eqref{T21}-\eqref{important G}, we can rewrite the upper bound for $\tilde T_2$ by
    \begin{equation*}\begin{split}
    \tilde T_2\leq & \gamma \frac{\tilde\mu^{m-1}(l_j-\bar l)^{p-2}}{r_j^n}(\chi+B^{-p}+B^{-(p-1)})^{\frac{p}{n}}
    \\&\times\left[\iint_{Q_j(\bar l)}
\varphi_j(\bar l)^{k_1p}|D\tilde \psi_j|^p \,\mathrm {d}x\mathrm {d}t
+\iint_{Q_j(\bar l)}
\varphi_j(\bar l)^{(k_1-1)p}\tilde \psi_j^p|D\varphi_j|^p \,\mathrm {d}x\mathrm {d}t\right]
\\=&:\gamma(\chi+B^{-p}+B^{-(p-1)})^{\frac{p}{n}}(\tilde T_7+\tilde T_8),
    \end{split}\end{equation*}
    with the obvious meanings of $\tilde T_7$ and $\tilde T_8$.
    We first consider the  estimate for $\tilde T_7$.
    In view of \eqref{tildemutildeu} and \eqref{reduction}, we see that
    \begin{equation*}\begin{split}
    \tilde T_7&\leq \gamma \frac{(l_j-\bar l)^{p-2}}{r_j^n}
    \iint_{Q_j(\bar l)}\tilde u^{m-1}
|D\tilde \psi_j|^p \varphi_j(\bar l)^{k_1p}\,\mathrm {d}x\mathrm {d}t
\\&\leq \gamma \frac{(l_j-\bar l)^{p-2}}{r_j^n}
    \iint_{Q_j(\bar l)}u^{m-1}
|D\psi_j|^p \varphi_j(\bar l)^{k_1p}\,\mathrm {d}x\mathrm {d}t.
    \end{split}\end{equation*}
To proceed further, we use Lemma \ref{Cac1} and take into account the estimates for $\tilde T_3$-$\tilde T_6$. We proceed to estimate
  $\tilde T_7$ by
  \begin{equation*}\begin{split}
 \tilde T_7
 \leq &\gamma  \frac{(l_j-\bar l)^{p-2}}{r_j^p}\iint_{L_j(\bar l)}u^{m-1}\left(\frac{l_j-u}{l_j-\bar l}\right)^{(1+\lambda)(p-1)}
 \varphi_j(\bar l)^{(k_1-1)p}\,\mathrm {d}x\mathrm {d}t
  \\&+\gamma \iint_{L_j(\bar l)}\frac{l_j-u}{l_j-\bar l}|\partial_t\varphi_j(\bar l)|\,\mathrm {d}x\mathrm {d}t
+\gamma \frac{r_j^p}{(l_j-\bar l)^p}\tilde\mu^{1-m}\int_{B_j}g^{\frac{p}{p-1}}\,\mathrm {d}x
 \\& +\gamma \frac{r_j^p}{(l_j-\bar l)^{p-1}}\tilde\mu^{1-m}\int_{B_j}|f|\,\mathrm {d}x
 \\ \leq&\gamma (\chi+\tilde B^{-p}+\tilde B^{-(p-1)}),
  \end{split}\end{equation*}
  where the constant $\gamma$ depends only upon the data.
  To estimate $\tilde T_8$, we use \eqref{tildepsiupper}, \eqref{tildemutildeu}, \eqref{0st estimatexixi1}, \eqref{0st estimatexixi2}
  and Young's inequality to obtain
    \begin{equation*}\begin{split}
  \tilde T_8&\leq  \frac{\tilde\mu^{m-1}(l_j-\bar l)^{p-2}}{r_j^{n+p}}\iint_{L_j(\bar l)}
\left(\frac{l_j-\tilde u}{l_j-\bar l}\right)^{p-1}
\varphi_j(\bar l)^{(k_1-1)p} \,\mathrm {d}x\mathrm {d}t
\\&\leq \frac{\tilde\mu^{m-1}(l_j-\bar l)^{p-2}}{r_j^{n+p}}|L_j(\bar l)|
\\&\quad+
\frac{(l_j-\bar l)^{p-2}}{r_j^{n+p}}\iint_{L_j(\bar l)}\tilde u^{m-1}
\left(\frac{l_j-u}{l_j-\bar l}\right)^{(1+\lambda)(p-1)}
\varphi_j(\bar l)^{(k_1-1)p} \,\mathrm {d}x\mathrm {d}t
\\&\leq \gamma \chi
  \end{split}\end{equation*}
  and hence we arrive at
  $\tilde T_2\leq \gamma(\chi+\tilde B^{-p}+\tilde B^{-(p-1)})^{1+\frac{p}{n}}$.
  This also yields that
   \begin{equation}\begin{split}\label{Lprimeprimexixi}
  &\frac{\tilde\mu^{m-1}(l_j-\bar l)^{p-2}}{r_j^{n+p}}\iint_{L_j^{\prime\prime}(\bar l)}\left(\frac{l_j-\tilde u}
  {l_j-\bar l}\right)^{(1+\lambda)(p-1)}\varphi_j(\bar l)^{k-p}
 \,\mathrm {d}x\mathrm {d}t
 \\&\leq \gamma\epsilon_2\chi+\gamma(\epsilon_2)(\chi+\tilde B^{-p}+\tilde B^{-(p-1)})^{1+\frac{p}{n}}.
 \end{split}\end{equation}
Moreover, we conclude with
  \begin{equation}\begin{split}\label{AA1xixi}
  &\frac{\tilde\mu^{m-1}(l_j-\bar l)^{p-2}}{r_j^{n+p}}\iint_{L_j(\bar l)}\left(\frac{l_j-\tilde u}{l_j-\bar l}
  \right)^{(1+\lambda)(p-1)}\varphi_j(\bar l)^{k-p}
 \,\mathrm {d}x\mathrm {d}t
 \\&\leq \gamma\epsilon_1^{(1+\lambda)(p-1)}\chi+
 \gamma\epsilon_2\chi+\gamma(\epsilon_2)\eta^{1-m}(\chi+\tilde B^{-p}+\tilde B^{-(p-1)})^{1+\frac{p}{n}}.
 \end{split}\end{equation}
In order to derive an upper bound for $A_j(\bar l)$,
we need to estimate the second term on the right-hand side of \eqref{Aj1st}. To this end,
we apply Lemma \ref{Cac1} with $(l,d,\theta)$ replaced by  $(l_j,l_j-\bar l,\tilde \mu^{1-m}(l_j-\bar l)^{2-p}r_j^p)$ to obtain
 \begin{equation}\begin{split}\label{estimateforGxixi}
  &\esssup_t
  \frac{1}{r_j^n}
  \int_{B_j}
G\left(\frac{l_j-u}{l_j-\bar l} \right)\varphi_j(\bar l)^k\,\mathrm {d}x
  \\
 \leq &
 \gamma  \frac{(l_j-\bar l)^{p-2}}{r_j^{p+n}}\iint_{L_j(\bar l)}u^{m-1}\left(\frac{l_j-u}{l_j-\bar l}\right)^{(1+\lambda)(p-1)}
  \varphi_j(\bar l)^{k-p}\,\mathrm {d}x\mathrm {d}t
  \\&+\gamma
    \frac{1}{r_j^n}\iint_{L_j(\bar l)}\frac{l_j-u}{l_j-\bar l}|\partial_t\varphi_j(\bar l)|\varphi_j(\bar l)^{k-1}\,\mathrm {d}x\mathrm {d}t
   \\&+\gamma \frac{r_j^{p-n}}{(l_j-\bar l)^p}\tilde\mu^{1-m}\int_{B_j}g^{\frac{p}{p-1}}\,\mathrm {d}x
  +\gamma \frac{r_j^{p-n}}{(l_j-\bar l)^{p-1}}\tilde\mu^{1-m}\int_{B_j}|f|\,\mathrm {d}x
  \\=&:\tilde S_1+\tilde S_2+\tilde S_3+\tilde S_4,
   \end{split}\end{equation}
   with the obvious meanings of $\tilde S_1$-$\tilde S_4$. To estimate $\tilde S_1$, we apply \eqref{AA1xixi} and \eqref{cutxixi} to deduce
   \begin{equation*}\begin{split}
   \tilde S_1&\leq 2^{(1+\lambda)(p-1)}
   \gamma\frac{(l_j-\bar l)^{p-2}\tilde\mu^{m-1}}{r_j^{p+n}}\iint_{L_j(\bar l)}\left(\frac{l_j-u}{l_j-\bar l}\right)^{(1+\lambda)(p-1)}
  \varphi_j(\bar l)^{k-p}\,\mathrm {d}x\mathrm {d}t
  \\&\leq 2^{(1+\lambda)(p-1)}
   \gamma\left[\epsilon_1^{(1+\lambda)(p-1)}\chi+
 \epsilon_2\chi+\gamma(\epsilon_2)\eta^{1-m}(\chi+\tilde B^{-p}+\tilde B^{-(p-1)})^{1+\frac{p}{n}}\right].
   \end{split}\end{equation*}
   Next, we consider the estimate for $\tilde S_2$. Similar to \eqref{cutxi}, we note from $l_j-\frac{1}{128}\xi\omega> \frac{1}{2}l_j$ that
   \begin{equation*}\begin{split}
   \tilde S_2=&\gamma
    \frac{1}{r_j^n}
    \iint_{L_j(\bar l)\cap \left\{u>\frac{1}{128}\xi\omega\right\}}
    \frac{l_j-\tilde   u}{l_j-\bar l}|\partial_t\varphi_j(\bar l)|\varphi_j(\bar l)^{k-1}\,\mathrm {d}x\mathrm {d}t
   \\&+\gamma
    \frac{1}{r_j^n} \iint_{L_j(\bar l)\cap \left\{u\leq\frac{1}{128}\xi\omega\right\}}
    \frac{l_j}{l_j-\bar l}|\partial_t\varphi_j(\bar l)|\varphi_j(\bar l)^{k-1}\,\mathrm {d}x\mathrm {d}t
    \\ \leq& 2\gamma
    \frac{1}{r_j^n}
    \iint_{L_j(\bar l)}
    \frac{l_j-\tilde   u}{l_j-\bar l}|\partial_t\varphi_j(\bar l)|\varphi_j(\bar l)^{k-1}\,\mathrm {d}x\mathrm {d}t.
    \end{split}\end{equation*}
Furthermore, we decompose
   $L_j(\bar l)=L^\prime_j(\bar l)\cup L^{\prime\prime}_j(\bar l)$, where $L^\prime_j(\bar l)$ and $L^{\prime\prime}_j(\bar l)$
   satisfy \eqref{Ldecompositionxixi}. In view of $|\partial_t\varphi_j(\bar l)|\leq 9(l_j-\bar l)^{p-2}\tilde\mu^{m-1}r_j^{-p}$, we
   use \eqref{0st estimatexixi1}
   and \eqref{Lprimeprimexixi} to conclude that
    \begin{equation*}\begin{split}
   \tilde S_2\leq & \gamma
    \frac{(l_j-\bar l)^{p-2}\tilde\mu^{m-1}}{r_j^{n+p}}\iint_{L_j(\bar l)}\frac{l_j-\tilde u}{l_j-\bar l}\varphi_j(\bar l)^{k-1}
    \,\mathrm {d}x\mathrm {d}t
   \\
   \leq &\gamma\epsilon_1\frac{\tilde\mu^{m-1}(l_j-\bar l)^{p-2}}{r_j^{n+p}}|L_j^\prime(\bar l)|
   \\&+\gamma
   \frac{(l_j-\bar l)^{p-2}\tilde\mu^{m-1}}{r_j^{p+n}}\iint_{L_j^{\prime\prime}(\bar l)}
   \left(\frac{l_j-\tilde u}{l_j-\bar l}\right)^{(1+\lambda)(p-1)}
  \varphi_j(\bar l)^{k-p}\,\mathrm {d}x\mathrm {d}t
  \\ \leq & \epsilon_1\chi+\gamma\left[\epsilon_2\chi+\gamma(\epsilon_2)(\chi+\tilde B^{-p}+\tilde B^{-(p-1)})^{1+\frac{p}{n}}\right].
     \end{split}\end{equation*}
     Similar to the estimates of $\tilde T_5$ and $\tilde T_6$, we see that $\tilde S_3+\tilde S_4\leq \gamma (\tilde B^{-p}+\tilde B^{-(p-1)})$. Next, we substitute the estimates for $S_1$-$S_4$
     into \eqref{estimateforG}. Taking into account \eqref{AA1},
      we infer that there exist constants
     $\gamma_1^\prime=\gamma_1^\prime(\text{data})$ and $\gamma_2^\prime=\gamma_2^\prime(\text{data},\epsilon_1,\epsilon_2)$
     such that
     \begin{equation*}\begin{split}
     A_j(\bar l)\leq &\gamma_1^\prime (\tilde B^{-p}+\tilde B^{-(p-1)})+
   \gamma_1^\prime(\epsilon_1+
 \epsilon_2)\chi
 +\gamma_2^\prime(\chi+\tilde B^{-p}+\tilde B^{-(p-1)})^{1+\frac{p}{n}}.
      \end{split}\end{equation*}
      At this point, we first choose $\epsilon_1$ and $\epsilon_2$ be such that
 \begin{equation}\begin{split}\label{epsilon1epsilon2}
 \epsilon_1=\epsilon_2=\frac{1}{2}\min\left\{\frac{1}{8\gamma_1},\ \frac{1}{8\gamma_1^\prime}\right\}
 \end{split}\end{equation}
and then we determine the value of $\chi$ by
  \begin{equation}\begin{split}\label{chixi11}
  \chi=\frac{1}{2}\min\left\{\frac{1}{100^{\frac{n}{p}}\gamma_2^{\frac{n}{p}}},\ \frac{1}
  {100^{\frac{n}{p}}(\gamma_2^\prime)^{\frac{n}{p}}}\right\},
  \end{split}\end{equation}
  where $\gamma_1$ and $\gamma_2$ are the constants in the estimate \eqref{Ajinclusion}.
 Finally, with the choices of $\epsilon_1$, $\epsilon_2$ and $\chi$, we set $B$ so large that
  \begin{equation*}\begin{split}
 \tilde B^{-p}+\tilde B^{1-p}<\min\left\{\frac{1}{100\gamma_1}\chi,\ \left(\frac{1}{100\gamma_2
 }\chi\right)^\frac{1}{1+\frac{p}{n}},\ \frac{1}{100\gamma_1^\prime}\chi,\ \left(\frac{1}{100\gamma_2^\prime
 }\chi\right)^\frac{1}{1+\frac{p}{n}}\right\}.
   \end{split}\end{equation*}
      With the choices of $\epsilon_1$, $\epsilon_2$, $\chi$ and $\tilde B$, we get $A_j(\bar l)\leq \frac{1}{2}\chi$ holds
       for either $Q_j(\bar l)\nsubseteq\widetilde Q_1$ or $Q_j(\bar l)\subseteq\widetilde Q_1$, which completes the proof of
      \eqref{Ajxi}.
Our next goal is to determine the value of $l_{j+1}$.
      We first consider the case
    \begin{equation*}\iint_{B_{r_j}(x_1)\times\left\{t_1-r_j^p<t<t_1\right\}}\tilde u^{m-1}(l_j-u)_+^{(1+\lambda)(p-1)}
    \,\mathrm {d}x\,\mathrm {d}t=0.\end{equation*}
According to Lemma \ref{Lebesguepoint+} and \eqref{ljxi}, we conclude that
   \begin{equation*} u(x_1,t_1)\geq l_j\geq  \mu_-+\frac{1}{2}\tilde B\alpha_{i-1}+\frac{1}{4}\xi\omega
   > \mu_-+\frac{1}{4}\xi\omega,
   \end{equation*}
   which proves \eqref{DeGiorgi2}.
We now consider the case
     \begin{equation*}\iint_{B_{r_j}(x_1)\times\left\{t_1-r_j^p<t<t_1\right\}}
    \tilde u^{m-1} (l_j-u)_+^{(1+\lambda)(p-1)}\,\mathrm {d}x\,\mathrm {d}t>0.\end{equation*}
    It follows that $A_j(l)\to+\infty$
    as $l\to l_j$.
Since $A_j(l)$ is continuous and increasing, we infer that there exists a number
 $\tilde l\in (\bar l, l_j)$ such that $A_j(\tilde l)=\chi$.
    At this point, we set
    %If $\tilde l<l_0-\frac{1}{4}(\alpha_{-1}-\alpha_0)$, then we set $l_1=\tilde l$.
    \begin{equation}\label{def lj+1xi}
	l_{j+1}=\begin{cases}
\tilde l,&\quad \text{if}\quad \tilde l<l_j-\frac{1}{4}(\alpha_{j-1}-\alpha_j),\\
	l_j-\frac{1}{4}(\alpha_{j-1}-\alpha_j),&\quad \text{if}\quad \tilde l\geq l_j-\frac{1}{4}(\alpha_{j-1}-\alpha_j).
	\end{cases}
\end{equation}
Recalling from \eqref{upper bound for lxi} that $l_j-\bar l>\frac{1}{4}\tilde B(\alpha_{j-1}-\alpha_j)>\frac{1}{4}(\alpha_{j-1}-\alpha_j)$,
we find that the definition of $l_{j+1}$ is justified.

Step 4: \emph{Proof of the estimate \eqref{DeGiorgi2}.}
For simplicity of notation, we write
$Q_j=Q_j(l_{j+1})$, $L_j=L_j(l_{j+1})$ and $d_j=l_j-l_{j+1}$. According to \eqref{def lj+1xi}, we
find that \eqref{lixi} and \eqref{Aj-1xi}
hold with $i=j+1$. In view of \eqref{ljxi}$_j$, we conclude from $\alpha_{j-1}\geq\alpha_j$ that
\begin{equation*}\begin{split}
l_{j+1}>&\bar l=\frac{1}{2}l_j+\frac{1}{4}\tilde B\alpha_j+\frac{1}{8}\xi\omega+
       \frac{1}{2}\mu_-
       \\&>\frac{1}{2}\left(\mu_-+\frac{1}{2}\tilde B\alpha_{j-1}+\frac{1}{4}\xi\omega\right)+\frac{1}{4}\tilde B\alpha_j+\frac{1}{8}\xi\omega+
       \frac{1}{2}\mu_-
       \\&=\mu_-+\frac{1}{4}\tilde B(\alpha_j+\alpha_{j-1})+\frac{1}{4}\xi\omega\geq \mu_-+\frac{1}{2}\tilde
       B\alpha_j+\frac{1}{4}\xi\omega.
\end{split}\end{equation*}
This yields that \eqref{ljxi}$_{j+1}$ holds. Repeating the arguments as in Step 3, we construct
$l_{j+2}$ and hence we can determine a sequence of numbers $\{l_i\}_{i=0}^\infty$ satisfying \eqref{lixi}-\eqref{ljxi}.
Noting that the sequence $\{l_i\}_{i=0}^\infty$ is decreasing, we infer from \eqref{ljxi} that the limitation of $l_i$ exists.
This also yields that $d_i\to0$ as $i\to\infty$. Define
$$\hat l=\lim_{i\to\infty}l_i$$
and we assert that $\hat l=u(x_1,t_1)$.
Recalling that \eqref{Aj-1xi} holds for any $i=1,2,\cdots$, we deduce that
  \begin{equation*}\begin{split}\frac{1}{r_i^{n+p}}&\iint_{B_{r_i}(x_1)\times\left\{t_1-r_i^p<t<t_1\right\}}
  (\hat l-u)_+^{(1+\lambda)(p-1)}\,\mathrm {d}x\,\mathrm {d}t\leq
  \frac{4^{n+p}}{r_{i-1}^{n+p}}\iint_{L_{i-1}}\left(l_{i-1}-u\right)^{(1+\lambda)(p-1)}\varphi_{i-1}^{k-p}
 \,\mathrm {d}x\mathrm {d}t
 \\&\leq 4^{n+p}(\xi\omega)^{1-m}A_{i-1}(l_i)d_{i-1}^{(1+\lambda)(p-1)-(p-2)}\leq 4^{n+p}(\xi\omega)^{1-m}\chi d_{i-1}^{(1+\lambda)(p-1)-(p-2)}
 \to 0
  \end{split}\end{equation*}
  as $i\to\infty$. By Lemma \ref{Lebesguepoint+}, we have $\hat l=u(x_1,t_1)$. Next, we claim that for any $j\geq1$ there holds
  \begin{equation}\begin{split}\label{djdj-1xi}
  d_j\leq &\frac{1}{4}d_{j-1}+
  \gamma\left(r_{j-1}^{p-n}\tilde\mu^{1-m}\int_{B_{j-1}} g(y)^{\frac{p}{p-1}}
  \,\mathrm {d}y\right)^{\frac{1}{p}}\\&+\gamma\left(r_{j-1}^{p-n}\tilde\mu^{1-m}\int_{B_{j-1}}|f(y)|
  \,\mathrm {d}y\right)^{\frac{1}{p-1}}.
  \end{split}\end{equation}
  To start with, for any fixed $j\geq 1$, we first assume that
  \begin{equation}\begin{split}\label{djdj-1proofxi}
  d_j>\frac{1}{4}d_{j-1}\qquad\text{and}\qquad d_j>\frac{1}{4}(\alpha_{j-1}-\alpha_j),
   \end{split}\end{equation}
   since otherwise \eqref{djdj-1} holds immediately. In view of $d_j>\frac{1}{4}(\alpha_{j-1}-\alpha_j)$,
   we infer from \eqref{def lj+1} that
   $A_j(l_{j+1})=A_j(\tilde l)=\chi$. According to \eqref{djdj-1proofxi}, we
  repeat the arguments from Step 3 and this gives
  \begin{equation*}\begin{split}
  \chi=A_j(l_{j+1})\leq&
  \gamma \frac{r_j^{p-n}}{d_j^p}\tilde\mu^{1-m}\int_{B_j}g^{\frac{p}{p-1}}\,\mathrm {d}x+\gamma
  \frac{r_j^{p-n}}{d_j^{p-1}}\tilde\mu^{1-m}\int_{B_j}|f|\,\mathrm {d}x
  \\
  &+
   \gamma(\epsilon_1+
 \epsilon_2)\chi
 +\gamma(\epsilon_2)\chi^{1+\frac{p}{n}}\\&+\gamma(\epsilon_2)\left[
 \frac{r_j^{p-n}}{d_j^p}\tilde\mu^{1-m}\int_{B_j}g^{\frac{p}{p-1}}\,\mathrm {d}x+\gamma
 \frac{r_j^{p-n}}{d_j^{p-1}}\tilde\mu^{1-m}\int_{B_j}|f|\,\mathrm {d}x\right]^{1+\frac{p}{n}}.
  \end{split}\end{equation*}
According to the choices of $\epsilon_1$, $\epsilon_2$ and $\chi$ in \eqref{epsilon1epsilon2}-\eqref{chixi11}, we conclude with
    \begin{equation*}\begin{split}
 (\epsilon_1+
 \epsilon_2)\chi
 +\gamma(\epsilon_2)\chi^{1+\frac{p}{n}}\leq\frac{1}{4}\chi.
 \end{split}\end{equation*}
 Consequently, we infer that either
  \begin{equation*}\begin{split}
  d_j\leq \gamma\left(r_j^{p-n}\tilde\mu^{1-m}\int_{B_j}g^{\frac{p}{p-1}}\,\mathrm {d}x\right)^{\frac{1}{p}}
  \qquad\text{or}\qquad
  d_j\leq
  \gamma \left(r_j^{p-n}\tilde\mu^{1-m}\int_{B_j}|f|\,\mathrm {d}x\right)^{\frac{1}{p-1}},
   \end{split}\end{equation*}
   which proves the inequality \eqref{djdj-1xi}. Let $J>1$ be a fixed integer. We sum up the inequality \eqref{djdj-1xi} for $j=1,\cdots,J-1$
   and obtain
   \begin{equation*}\begin{split}
   l_1-l_J\leq& \frac{1}{3}d_0+\gamma\sum_{j=1}^{J-1}\left(r_{j-1}^{p-n}\tilde\mu^{1-m}\int_{B_{j-1}}|f(y)|
  \,\mathrm {d}y\right)^{\frac{1}{p-1}}
\\&+ \gamma\sum_{j=1}^{J-1}\left(r_{j-1}^{p-n}\tilde\mu^{1-m}\int_{B_{j-1}} g(y)^{\frac{p}{p-1}}
  \,\mathrm {d}y\right)^{\frac{1}{p}}.
   \end{split}\end{equation*}
   Recalling that $d_0=l_0-l_1$ and $l_0=\mu_-+\xi\omega$, we apply \eqref{omega1violated1} to obtain
  \begin{equation}\begin{split}\label{uupperboundxi}
  \mu_-+\xi\omega\leq
 & \frac{4}{3}d_0+l_J
  +\gamma\tilde\mu^{\frac{1-m}{p-1}}\int_0^{2R}\left(\frac{1}{r^{n-p}}\int_{B_r(x_1)}f(y)
\,\mathrm {d}y\right)^{\frac{1}{p-1}}\frac{1}{r}\,\mathrm {d}r
  \\&+\gamma\tilde\mu^{\frac{1-m}{p}}\int_0^{2R}\left(\frac{1}{r^{n-p}}\int_{B_r(x_1)}g(y)^{\frac{p}{p-1}}
\,\mathrm {d}y\right)^{\frac{1}{p}}\frac{1}{r}\,\mathrm {d}r
\\ \leq & \frac{4}{3}d_0+l_J+\gamma\frac{1}{\tilde B}\xi\omega.
   \end{split}\end{equation}
   Taking into account that $l_1\geq \bar l$ where $\bar l=\frac{1}{2}l_0+\frac{1}{2}\mu_-+\frac{1}{4}\tilde B\alpha_0+\frac{1}{8}\xi\omega$.
 We infer from \eqref{xil0} that $d_0\leq l_0-\bar  l\leq \frac{3}{8}\xi\omega$.
  Passing to the limit $J\to\infty$, we conclude from \eqref{uupperboundxi} that
   \begin{equation*}\begin{split}
   u(x_1,t_1)>\mu_-+\frac{1}{2}\xi\omega,
    \end{split}\end{equation*}
    provided that we choose $\tilde B>4\gamma$. This shows that the inequality \eqref{DeGiorgi2} holds for
    almost everywhere point in $\widetilde Q_2$,
    since $(x_1,t_1)$ is the Lebesgue point of $u$.
    The constant $\nu_1$ is determined via \eqref{nu1}, where the quantity $\chi$ in \eqref{nu1}
    is fixed in terms of \eqref{chixi11}.
    We have thus proved the lemma.
 \end{proof}
 With the help of the proceeding lemmas we can now establish a decay estimate for the oscillation of the weak solution
 in a smaller cylinder,
 and the following proposition is our main result in this section.
\begin{proposition}\label{1st proposition}
Let $\tilde Q_0=Q\left(\frac{1}{16}R,(\mu_+)^{1-m}\omega^{2-p}\left(\frac{1}{16}R\right)^p\right)$ and
let $u$ be a bounded nonnegative weak solution to \eqref{parabolic}-\eqref{A} in $\Omega_T$.
There exist $0<\xi_1<2^{-5}$ and $B_1>1$ depending only upon the data and $A$ such that
\begin{equation*}\begin{split}
\essosc_{\tilde Q_0} u\leq (1-2^{-1}\xi_1)\omega+B_1\xi_1^{-1}(
F_1(2R)^{\frac{p}{p+m-1}}+F_2(2R)^{\frac{p-1}{p+m-2}}).
 \end{split}\end{equation*}
\end{proposition}
\begin{proof}
To start with, we first assume that
\eqref{omega1}, \eqref{omega2} and \eqref{omega3} are violated.
We take $\nu_*=\nu_1$ according to \eqref{nu1} in Lemma \ref{timeexpandlemma}. Moreover,
we choose $s_*=2\nu_*^{-1}A^{p-2}$
and $\xi_1=2^{-s_*}$. It follows from Lemma \ref{timeexpandlemma} that
\begin{equation*}\left|\left\{(x,t)\in \widetilde Q_1:u<\mu_-+\xi_1\omega\right\}
\right|\leq \nu_1|\widetilde Q_1|,\end{equation*}
where $\hat Q_1=B_{\frac{R}{8}}\times(\hat t,0)$. This implies that the condition for Lemma \ref{lemmaDeGiorgi2} is satisfied and it
follows that
\begin{equation*}\begin{split}
\essosc_{\tilde Q_0} u\leq (1-2^{-1}\xi_1)\omega.
 \end{split}\end{equation*}
 On the other hand, if either \eqref{omega1}, \eqref{omega2} or \eqref{omega3} holds, then we conclude that either $\omega\leq (\mu_+)^{\frac{1-m}{p}}F_1(2R)$, $\omega\leq (\mu_+)^{\frac{1-m}{p-1}}F_2(2R)$,
 $\omega\leq\xi_1^{-\frac{2}{p}}(\mu_+)^{\frac{1-m}{p}}F_1(2R)$, $\omega\leq \xi_1^{-\frac{1}{p-1}}(\mu_+)^{\frac{1-m}{p-1}}F_2(2R)$,
\begin{equation*}\xi_1\omega\leq \tilde B(\mu_-+\xi_1\omega)^{\frac{1-m}{p}}F_1(2R)\qquad\text{or}\qquad\xi_1\omega\leq \tilde B(\mu_-+\xi_1\omega)^{\frac{1-m}{p-1}}F_2(2R).\end{equation*}
Noting that $\mu_+\geq \frac{1}{2}\omega$, $(\mu_-+\xi_1\omega)^{\frac{1-m}{p}}\leq (\xi_1\omega)^{\frac{1-m}{p}}$ and $(\mu_-+\xi_1\omega)^{\frac{1-m}{p-1}}\leq (\xi_1\omega)^{\frac{1-m}{p-1}}$, we deduce that the inequality
\begin{equation*}\begin{split}
\essosc_{\tilde Q_0} u\leq\omega\leq B_1\xi_1^{-1}(
F_1(2R)^{\frac{p}{p+m-1}}+F_2(2R)^{\frac{p-1}{p+m-2}}).
 \end{split}\end{equation*}
 holds for a constant $B_1=B_1(\text{data},A)>1$.
 It is now obvious that the proposition holds.
\end{proof}
  \section{The second alternative}
In this section, we will establish the decay estimate of the essential oscillation for the second alternative.
We begin with the following lemma which is a standard result that can be found in \cite[Lemma 7.1]{PV}.
 %For any $-(A^{p-2}-1)(\mu_+)^{1-m}\omega^{2-p}R^p\leq \bar t\leq0$, there holds
%\begin{equation*}\label{2nd}\left|\left\{(x,t)\in Q_{\frac{3}{4}R}^-(\bar t):u>\mu_+-\frac{\omega}{4}\right\}\right|
%\leq (1-\nu_0)|Q_{\frac{3}{4}R}^-(\bar t)|.\end{equation*}
\begin{lemma}Let $-(A^{p-2}-1)(\mu_+)^{1-m}\omega^{2-p}R^p\leq \bar t\leq0$
and let $u$ be a bounded nonnegative weak solution to \eqref{parabolic}-\eqref{A} in $\Omega_T$.
Then there exists a time level
 \begin{equation*}t^*\in \left[\bar t-(\mu_+)^{1-m}\omega^{2-p}\left(\frac{3}{4}R\right)^p,\bar t-\frac{1}{2}
 \nu_0(\mu_+)^{1-m}\omega^{2-p}\left(\frac{3}{4}R\right)^p\right]\end{equation*}
 such that
 \begin{equation}\label{2nd alternative begin}
 \left|\left\{x\in B_{\frac{3}{4}R}:u(x,t^*)>\mu_+-\frac{\omega}{4}\right\}\right|\leq
 \left(\frac{1-\nu_0}{1-\frac{1}{2}\nu_0}\right)|B_{\frac{3}{4}R}|,\end{equation}
 where $\nu_0$ is the constant claimed by Lemma \ref{lemmaDeGiorgi1}.
\end{lemma}
Next. we provide the following lemma regarding the time propagation of positivity.
\begin{lemma}\label{time 2nd}
Let $u$ be a bounded nonnegative weak solution to \eqref{parabolic}-\eqref{A} in $\Omega_T$.
There exists a positive constant $s_1$
that can be determined a priori only in terms of the data such that
 either
\begin{equation}\begin{split}\label{2ndomega assumption1}
\omega\leq 2^{\frac{s_1+n}{p-1}}(\mu_+)^{\frac{1-m}{p-1}}F_2(2R)
+2^{\frac{2s_1+n}{p}}(\mu_+)^{\frac{1-m}{p}}F_1(2R)
\end{split}\end{equation}
or
\begin{equation}\label{2nd measure estimate}
\left|\left\{x\in B_{\frac{3}{4}R}:u(x,t)>\mu_+-\frac{\omega}{2^{s_1}}\right\}\right|\leq
\left(1-\left(\frac{\nu_0}{2}\right)^2\right)|B_{\frac{3}{4}R}|
\end{equation}
for all $t\in \left[\bar t-\frac{1}{2}\nu_0(\mu_+)^{1-m}\omega^{2-p}(\frac{3}{4}R)^p,\bar t\right]$.
\end{lemma}
\begin{proof}
For simplicity, we write $\rho=\frac{3}{4}R$.
Let $k=\mu_+-\frac{1}{4}\omega$ and $c=2^{-2-l}\omega$ where $l\geq2$ is to be determined later.
Moreover, we set $H_k^+=\frac{1}{4}\omega$ and it follows that
\begin{equation*}
H_k^+\geq \esssup_{B_\rho\times[t^*,\bar t]}\left|\left(u-k\right)_+\right|.
\end{equation*}
We now consider the logarithmic function defined by
\begin{equation*}\begin{split}\psi^+=\ln^+\left(\frac{\frac{1}{4}\omega}
{\frac{1}{4}\omega-(u-k)_++c}\right).\end{split}\end{equation*}
Next, we take a smooth cutoff function $0<\zeta(x)\leq1$, defined in $B_\rho$, and satisfying
$\zeta\equiv 1$ in $B_{(1-\sigma)\rho}$ and $|D\zeta|\leq (\sigma \rho)^{-1}$, where $\sigma\in (0,1)$ is to be determined.
With these choices, we use the logarithmic estimate \eqref{lnCac} to obtain
\begin{equation}\begin{split}\label{psi+log}
&\int_{B_{(1-\sigma)\rho}\times\{t\}}[\psi^+(u)]^2\,\mathrm{d}x
\\&\leq
\int_{B_\rho\times\{t^*\}}[\psi^+(u)]^2\,\mathrm{d}x+\frac{\gamma}{\sigma^p \rho^p}\iint_{B_\rho\times[t^*,\bar t]}u^{m-1}\psi^+(u)
[(\psi^+)^\prime(u)]^{2-p}\,\mathrm{d}x\mathrm{d}t
\\&\quad+\gamma(\ln 2)l\frac{2^{2+l}}{\omega^{p-1}}(\mu_+)^{1-m}\rho^p\int_{B_\rho}|f|\,\mathrm{d}x+
\gamma(\ln 2)l\frac{2^{4+2l}}{\omega^p}(\mu_+)^{1-m}\rho^p\int_{B_\rho}g^{\frac{p}{p-1}}\,\mathrm{d}x
\\&=:I_1+I_2+I_3+I_4
\end{split}\end{equation}
for all $t\in[t^*,\bar t]$. Taking into account that $\psi^+(u)\leq l\ln2$, we conclude from \eqref{2nd alternative begin} that
\begin{equation*}\begin{split}
I_1\leq (l\ln2)^2\left|\left\{x\in B_\rho:u(x,t^*)>\mu_+-\frac{\omega}{4}\right\}\right|
\leq l^2(\ln^22)\left(\frac{1-\nu_0}{1-\frac{1}{2}\nu_0}\right)|B_\rho|.
\end{split}\end{equation*}
In view of $[(\psi^+)^\prime(u)]^{2-p}\leq(4^{-1}\omega)^{p-2}$, we have
\begin{equation*}\begin{split}
I_2\leq \gamma\frac{1}{\sigma^p\rho^p}(\mu_+)^{m-1}l(\ln2)\left(\frac{\omega}{4}\right)^{p-2}(\bar t-t^*)|B_\rho|\leq
\gamma \frac{l}{\sigma^p}|B_\rho|,
\end{split}\end{equation*}
where the constant $\gamma$ depends only upon the data.
At this stage, we assume that
\begin{equation}\begin{split}\label{omega assumption}
\omega>2^{\frac{l}{p-1}}(\mu_+)^{\frac{1-m}{p-1}}\left(\rho^{p-n}\int_{B_\rho}|f|\,\mathrm{d}x\right)^\frac{1}{p-1}
+2^{\frac{2l}{p}}(\mu_+)^{\frac{1-m}{p}}\left(\rho^{p-n}\int_{B_\rho}g^{\frac{p}{p-1}}\,\mathrm{d}x\right)^\frac{1}{p}
\end{split}\end{equation}
and hence
\begin{equation*}\begin{split}
I_3+I_4\leq \gamma l|B_\rho|\leq \gamma \frac{l}{\sigma^p}|B_\rho|.
\end{split}\end{equation*}
Combining the estimates for $I_1$-$I_4$, we arrive at
\begin{equation*}\begin{split}
\int_{B_{(1-\sigma)\rho}\times\{t\}}[\psi^+(u)]^2\,\mathrm{d}x\leq
l^2(\ln^22)\left(\frac{1-\nu_0}{1-\frac{1}{2}\nu_0}\right)|B_\rho|+\gamma \frac{l}{\sigma^p}|B_\rho|.
\end{split}\end{equation*}
On the other hand, the left-hand side of \eqref{psi+log} can be estimated below by integrating over the smaller set
\begin{equation*}\begin{split}
S=\left\{x\in B_{(1-\sigma)\rho}:u(x,t)>\mu_+-\frac{\omega}{2^{l+2}}\right\}
\end{split}\end{equation*}
and this implies that
\begin{equation*}\begin{split}
\int_{B_{(1-\sigma)\rho}\times\{t\}}[\psi^+(u)]^2\,\mathrm{d}x\geq (l-1)^2(\ln^22)|S|.
\end{split}\end{equation*}
Consequently, we infer that the estimate
\begin{equation*}\begin{split}
&\left|\left\{x\in B_\rho:u(x,t)>\mu_+-\frac{\omega}{2^{l+2}}\right\}\right|
\leq
|S|+|B_\rho\setminus B_{(1-\sigma)\rho}|
\\&\quad\leq \left[\left(\frac{l}{l-1}\right)^2\left(\frac{1-\nu_0}{1-\frac{1}{2}\nu_0}\right)+\gamma \frac{l}{\sigma^p(l-1)^2}
+\gamma\sigma\right]|B_\rho|
\end{split}\end{equation*}
holds  for all $t\in[t^*,\bar t]$. At this point, we choose $\gamma \sigma\leq\frac{3}{8}\nu_0^2$ and then $l$ so large that
\begin{equation*}\begin{split}
\left(\frac{l}{l-1}\right)^2<\left(1-\frac{1}{2}\nu_0\right)(1+\nu_0)\qquad\text{and}\qquad
\gamma\frac{1}{\sigma^pl}\leq\frac{3}{8}\nu_0^2.
\end{split}\end{equation*}
This proves the inequality \eqref{2nd measure estimate} with $s_1=l+2$. Moreover, if \eqref{omega assumption} is violated, then we get
\eqref{2ndomega assumption1} for such a choice of $s_1$. We have thus proved the lemma.
\end{proof}
Our task now is to establish a De Giorgi type lemma for the second alternative. To this end,
we introduce the concentric parabolic cylinders
\begin{equation*}\begin{split}
Q_A^{(1)}=B_{\frac{3}{4}R}\times\left(-\frac{1}{2}A^{p-2}(\mu_+)^{1-m}\omega^{2-p}R^p,0\right]
\end{split}\end{equation*}
and
\begin{equation*}\begin{split}
Q_A^{(2)}=B_{\frac{1}{2}R}\times\left(-\frac{1}{4}A^{p-2}(\mu_+)^{1-m}\omega^{2-p}R^p,0\right],
\end{split}\end{equation*}
where $A$ satisfies \eqref{first condition for A} and \eqref{second condition for A}.
The proof of the De Giorgi type lemma is based on the Kilpel\"ainen-Mal\'y technique.
Contrary to Lemma \ref{lemmaDeGiorgi1}, we deal with the estimates for $u$ near the supremum $\mu_+$ and
our proof makes no use of the cutoff function.
 \begin{lemma}\label{lemmaDeGiorgi1+}
 Let $0<\xi<4^{-1}p$ and
let $u$ be a bounded nonnegative weak solution to \eqref{parabolic}-\eqref{A} in $\Omega_T$.
There exist constants $\nu_2\in(0,1)$ and $\hat B>1$, depending only upon the data, such that if
\begin{equation*}\left|\left\{(x,t)\in Q_A^{(1)}:u\geq\mu_+-\xi\omega\right\}\right|\leq \nu_2|Q_A^{(1)}|,\end{equation*}
then either
\begin{equation}
\label{DeGiorgi1+}u(x,t)<\mu_+-\frac{1}{6}\xi\omega\qquad\text{for}\ \ \text{a.e.}\ \ (x,t)\in Q_A^{(2)}\end{equation}
or
\begin{equation}
\label{omega1+}\xi\omega\leq \hat B\left((\mu_+)^{\frac{1-m}{p}}F_1(2R)
+(\mu_+)^{\frac{1-m}{p-1}}F_2(2R)\right).\end{equation}
Here $A=\hat B\xi^{-1}$ and satisfies \eqref{first condition for A} and \eqref{second condition for A}.
 \end{lemma}
 \begin{proof}
 To start with, we first assume that
 \begin{equation}\label{hatB}
 \hat B>\max\left\{4^{102}B,\ \left(\frac{100p}{p-2}\right)^{\frac{1}{p-2}},\ A_p^\frac{1}{p-2}\right\},\end{equation}
 where $A_p>1$ is the constant such that for any $X>A_p$ there holds $X\leq 2^{(p-2)X}$. For such a choice of $\hat B$,
 the constant $A=\hat B\xi^{-1}$ satisfies \eqref{first condition for A} and \eqref{second condition for A}.
Next, we assume that \eqref{omega1} is violated, that is,
 \begin{equation}
\label{omega1violated+}\xi\omega>\frac{3}{4}\hat B\left(\frac{1}{3\hat B}\xi\omega
+(\mu_+)^{\frac{1-m}{p}}F_1(2R))+(\mu_+)^{\frac{1-m}{p-1}}F_2(2R)
\right).\end{equation}
We are reduced to proving \eqref{DeGiorgi1}.
Let $(x_1,t_1)\in Q_A^{(2)}$ be a fixed point and assume that $(x_1,t_1)$ is a Lebesgue point of the function $u$.
 We set $r_j=4^{-j}C^{-1}R$
 and $B_j=B_{r_j}(x_1)$ where $C>40$ is to be determined.
 For a sequence $\{l_j\}_{j=0}^\infty$ and
 a fixed $l>0$, we define
 $$Q_j(l)=B_j\times (t_1-(\mu_+)^{1-m}(l-l_j)^{2-p}r_j^p,t_1).$$
 Next, we set $\varphi_j(l)=\phi_j(x)\theta_{j,l}(t)$, where $\phi_j\in C_0^\infty(B_j)$,
 $\phi_j=1$ on $B_{j+1}$, $|D\phi_j|\leq r_j^{-1}$
 and $\theta_{j,l}(t)$ is a Lipschitz function
satisfies
 \begin{equation*}
 \theta_{j,l}(t)=1\qquad\text{in}\qquad t\geq t_1-\frac{4}{9}(\mu_+)^{1-m}(l-l_j)^{2-p}r_j^p,
 \end{equation*}
  \begin{equation*}
 \theta_{j,l}(t)=0\qquad\text{in}\qquad t\leq t_1-\frac{5}{9}(\mu_+)^{1-m}(l-l_j)^{2-p}r_j^p
 \end{equation*}
and
  \begin{equation*}
 \theta_{j,l}(t)=\frac{t-t_1-\frac{5}{9}(\mu_+)^{1-m}(l-l_j)^{2-p}r_j^p}{\frac{1}{9}(\mu_+)^{1-m}(l-l_j)^{2-p}r_j^p}
 \end{equation*}
in
$t_1-\frac{5}{9}(\mu_+)^{1-m}(l-l_j)^{2-p}r_j^p\leq t\leq t_1-\frac{4}{9}(\mu_+)^{1-m}(l-l_j)^{2-p}r_j^p$.
It is easy to show that $\varphi_j(l)=0$ on $\partial_PQ_j(l)$.
 Moreover, for $j=-1,0,1,2,\cdots$, we construct the sequence $\{\alpha_j\}$ similar to \eqref{alpha}, that is,
  \begin{equation}\begin{split}\label{alpha+}\alpha_j=&\frac{4^{-j-100}}
  {3\hat B}\xi\omega+\frac{3}{4}\int_0^{r_j}\left(r^{p-n}(\mu_+)^{1-m}\int_{B_r(x_1)}
  g(y)^{\frac{p}{p-1}} \,\mathrm {d}y
  \right)^{\frac{1}{p}}\frac{\mathrm {d}r}{r}\\&+\frac{3}{4}\int_0^{r_j}\left(r^{p-n}(\mu_+)^{1-m}\int_{B_r(x_1)}|f(y)| \,\mathrm {d}y
  \right)^{\frac{1}{p-1}}\frac{\mathrm {d}r}{r}.
  \end{split}\end{equation}
  According to the definition of $\alpha_j$, we see that $\alpha_j\to 0$ as $j\to\infty$ and the sequence of $\alpha_j$ satisfy
  \begin{equation}\begin{split}\label{alpha1+}
  \alpha_{j-1}-\alpha_j\geq &\frac{4^{-j-100}}{\hat B}\xi\omega+\gamma\left(r_j^{p-n}(\mu_+)^{1-m}\int_{B_j} g(y)^{\frac{p}{p-1}}
  \,\mathrm {d}y\right)^{\frac{1}{p}}\\&+\gamma\left(r_j^{p-n}(\mu_+)^{1-m}\int_{B_j}|f(y)|
  \,\mathrm {d}y\right)^{\frac{1}{p-1}},
  \end{split}\end{equation}
   \begin{equation}\begin{split}\label{alpha2+}
  \alpha_{j-1}-\alpha_j\leq & \frac{4^{-j-100}}{\hat B}\xi\omega+\gamma\left(r_{j-1}^{p-n}(\mu_+)^{1-m}\int_{B_{j-1}} g(y)^{\frac{p}{p-1}}
  \,\mathrm {d}y\right)^{\frac{1}{p}}\\&+\gamma\left(r_{j-1}^{p-n}(\mu_+)^{1-m}\int_{B_{j-1}}|f(y)|
  \,\mathrm {d}y\right)^{\frac{1}{p-1}}
  \end{split}\end{equation}
  and $\hat B\alpha_{j-1}\leq\xi\omega$ for all $j=0,1,2,\cdots$.
 Furthermore, we introduce a quantity $K_j(l)$ by
 \begin{equation*}\begin{split}
 K_j(l)=&\frac{(\mu_+)^{m-1}(l-l_j)^{p-2}}{r_j^{n+p}}\iint_{L_j(l)}\left(\frac{u-l_j}{l-l_j
 }\right)^{(1+\lambda)(p-1)}\varphi_j(l)^{k-p}
 \,\mathrm {d}x\mathrm {d}t
 \\&+\esssup_t\frac{1}{r_j^n}\int_{B_j\times\{t\}}G\left(\frac{u-l_j}{l-l_j}\right)\varphi_j(l)^{k}
 \,\mathrm {d}x,
 \end{split}\end{equation*}
 where $G$ is the function defined in \eqref{G} and $L_j(l)=Q_j(l)\cap \{u\geq l_j\}\cap\Omega_T$.
 We observe that $K_j(l)$ is continuous in $l>l_j$
 and $K_j(l)\downarrow0$ as $l\to\infty$. Let $l_0=\mu_+-\xi\omega$ and $\kappa>0$ is to be determined later.
 We construct the sequence $l_j$ in the following way. If $K_j(l_j+\frac{1}{4}(\alpha_{j-1}-\alpha_j))<\kappa$, then we set $l_{j+1}=l_j+\frac{1}{4}(\alpha_{j-1}-\alpha_j)$.
 If $K_j(l_j+\frac{1}{4}(\alpha_{j-1}-\alpha_j))\geq\kappa$, then we choose $l_{j+1}>l_j+\frac{1}{4}(\alpha_{j-1}-\alpha_j)$
 such that $K_j(l_{j+1})=\kappa$. From the definition of $l_j$, we find that
  \begin{equation}\begin{split}\label{Kj}K_j(l_{j+1})\leq\kappa \end{split}\end{equation}
  for any $j\geq0$.
We proceed similarly as in Lemma \ref{lemmaDeGiorgi1} and divide
the proof of \eqref{DeGiorgi1+} into three steps.

Step 1: \emph{We establish an iteration scheme.}
To start with, we set $\bar l=\frac{1}{2}(l_0+\mu_+)-\frac{1}{4}B\alpha_0$ and assert that $l_1\leq\bar l$.
 Observe that $\hat B\alpha_0< \xi\omega$ and hence
 \begin{equation}\begin{split}\label{l0+}
 \bar l-l_0=\frac{1}{2}\xi\omega-\frac{1}{4}\hat B\alpha_0\geq \frac{1}{4}\xi\omega
 \geq\frac{1}{4}(\alpha_{-1}-\alpha_0)\qquad\text{and}\qquad  \bar l-l_0\leq\frac{1}{2}\xi\omega.
 \end{split}\end{equation}
 It follows that
\begin{equation*}\begin{split}(\mu_+)^{1-m}(\bar l-l_0)^{2-p}r_0^p
\leq (\mu_+)^{1-m}4^{p-2}\hat B^{2-p}A^{p-2}\omega^{2-p}C^{-p}R^p
<\frac{1}{16}(\mu_+)^{1-m}A^{p-2}\omega^{2-p}R^p,
\end{split}\end{equation*}
 since $C>40$. This also yields that $Q_0(\bar l)\subset Q_A^{(1)}$.
 Since $u-l_0\leq \xi\omega$ in $L_0(\bar l)$, we conclude that
   \begin{equation*}\begin{split}
  & \frac{(\mu_+)^{m-1}(\bar l-l_0)^{p-2}}{r_0^{n+p}}\iint_{L_0(\bar l)}\left(\frac{u-l_0}{\bar l-l_0}
   \right)^{(1+\lambda)(p-1)}\varphi_0(\bar l)^{k-p}
 \,\mathrm {d}x\mathrm {d}t
 \\&\leq\frac{(\mu_+)^{m-1}(\bar l-l_0)^{p-2-(1+\lambda)(p-1)}}{r_0^{n+p}}
 \left(\xi\omega\right)^{(1+\lambda)(p-1)}|L_0(\bar l)|
 \\&\leq \hat\gamma_0 C^{n+p}\frac{(\mu_+)^{m-1}(\xi\omega)^{p-2}}{R^{n+p}}|L_0(\bar l)|
  \\&\leq \hat\gamma_0 C^{n+p}\hat B^{p-2}\frac{\left|Q_A^{(1)}\cap \left\{u\geq\mu_+-\xi\omega\right\}\right|}{|Q_A^{(1)}|}
   \leq \hat\gamma_0 C^{n+p}\hat B^{p-2}\nu_2,
    \end{split}\end{equation*}
    where the constant $\hat\gamma_0$ depends only upon the data.
    Moreover, we apply Lemma \ref{Cac2}
    with $(a,d,l)$ replaced by $(\mu_+,\bar l-l_0,l_0)$
    to obtain
     \begin{equation*}\begin{split}
   & \esssup_t\frac{1}{r_0^n}\int_{B_0\times\{t\}}G\left(\frac{u-l_0}{\bar l-l_0}\right)\varphi_0(\bar l)^{k}
 \,\mathrm {d}x \\ \leq &\gamma  \frac{(\bar l-l_0)^{p-2}}{r_0^{p+n}}\iint_{L_0(\bar l)}u^{m-1}\left(\frac{u-l_0}{\bar l-l_0}
 \right)^{(1+\lambda)(p-1)}
 \varphi_0(\bar l)^{k-p}\,\mathrm {d}x\mathrm {d}t
 \\&+\gamma\frac{1}{r_0^n} \iint_{L_0(\bar l)}\frac{u-l_0}{\bar l-l_0}|\partial_t\varphi_0(\bar l)|\,\mathrm {d}x\mathrm {d}t
   \\&+\gamma \frac{r_0^{p-n}}{(\bar l-l_0)^p}(\mu_+)^{1-m}\int_{B_0}g^{\frac{p}{p-1}}\,\mathrm {d}x
  +\gamma \frac{r_0^{p-n}}{(\bar l-l_0)^{p-1}}(\mu_+)^{1-m}\int_{B_0}|f|\,\mathrm {d}x,
 \end{split}\end{equation*}
 since $\varphi_0(\bar l)\in C_0^\infty(Q_0(\bar l))$ and the first term on the right-hand side of \eqref{Cacformula2} vanishes.
 Taking into account that $u\leq\mu_+$ in $L_0(\bar l)$ and $|\partial_t\varphi_0(\bar l)|\leq 9(\bar l-l_0)^{p-2}(\mu_+)^{m-1}r_0^{-p}$,
an argument similar
to the one used in the proof of Lemma \ref{lemmaDeGiorgi1} shows
that there exists a constant $\hat\gamma_1$ depending only upon the data such that
 \begin{equation*}\begin{split}
   & \esssup_t \frac{1}{r_0^n}\int_{B_0\times\{t\}}G\left(\frac{u-l_0}{\bar l-l_0}\right)\varphi_0(\bar l)^{k}
 \,\mathrm {d}x
 \leq \hat\gamma_1C^{n+p}\hat B^{p-2}\nu_2+\hat\gamma_1C^{n-p}(\hat B^{-p}+\hat B^{1-p}).
  \end{split}\end{equation*}
  Consequently, we infer that
   \begin{equation*}\begin{split}
  K_0(\bar l)\leq (\hat\gamma_0+\hat\gamma_1)C^{n+p}\hat B^{p-2}\nu_2+\hat\gamma_1C^{n-p}(\hat B^{-p}+\hat B^{1-p}).
    \end{split}\end{equation*}
    At this point, we choose $\nu_2<1$ and $\hat B>1$ be such that
     \begin{equation}\begin{split}\label{determine the value of nu2}
  (\hat\gamma_0+\hat\gamma_1)C^{n+p}\hat B^{p-2}
  \nu_2=\frac{\kappa}{4}\qquad\text{and}\qquad\hat\gamma_1C^{n-p}(\hat B^{-p}+\hat B^{1-p})<\frac{\kappa}{4}.
    \end{split}\end{equation}
Consequently, we deduce that $K_0(\bar l)\leq\frac{1}{2}\kappa$. It follows from \eqref{Kj} that $l_1\leq\bar l$.
    Next, we set $d_j=l_{j+1}-l_j$ and  $Q_j=B_j\times (t_1-(\mu_+)^{1-m}d_j^{2-p}r_j^p,t_1).$
We claim that for any $j\geq1$ there holds
  \begin{equation}\begin{split}\label{djdj-1+}
  d_j\leq &\frac{1}{4}d_{j-1}+\gamma \frac{4^{-j}}{\hat B}\xi\omega+
  \gamma\left(r_{j-1}^{p-n}(\mu_+)^{1-m}\int_{B_{j-1}} g(y)^{\frac{p}{p-1}}
  \,\mathrm {d}y\right)^{\frac{1}{p}}\\&+\gamma\left(r_{j-1}^{p-n}(\mu_+)^{1-m}\int_{B_{j-1}}|f(y)|
  \,\mathrm {d}y\right)^{\frac{1}{p-1}}.
  \end{split}\end{equation}

Step 2: \emph{Proof of the inequality \eqref{djdj-1+}.}
  To start with, we first assume that for any fixed $j\geq 1$ there holds
  \begin{equation}\begin{split}\label{djdj-1proof+}
  d_j>\frac{1}{4}d_{j-1}\qquad\text{and}\qquad d_j>\frac{1}{4}(\alpha_{j-1}-\alpha_j),
   \end{split}\end{equation}
   since otherwise \eqref{djdj-1+} holds immediately. Since $d_j>\frac{1}{4}(\alpha_{j-1}-\alpha_j)$, we infer
   from the construction of $K_j(l_{j+1})$ that
   $K_j(l_{j+1})=\kappa$. To simplify the notation,  we set $\varphi_i=\varphi_i(l_{i+1})$.
   In view of $d_j>\frac{1}{4}d_{j-1}$, we deduce that
    \begin{equation*}\begin{split}
 (\mu_+)^{1-m}d_j^{2-p}r_j^p\leq (\mu_+)^{1-m}\frac{r_{j-1}^p}{4^p}\left(\frac{d_{j-1}}{4}\right)^{2-p}=\frac{1}{16}
 (\mu_+)^{1-m}d_{j-1}
 ^{2-p}r_{j-1}^p,
  \end{split}\end{equation*}
  which yields $Q_j\subset Q_{j-1}$ and  $\varphi_{j-1}(x,t)=1$ for $(x,t)\in Q_j$.
We also remark that for any $i\geq0$, the inclusion $Q_i\subset \hat Q$ holds, where
$\hat Q$ is defined in \eqref{hat Q}. Taking into account that $A=\hat B\xi^{-1}$, we infer from \eqref{alpha1+} that
  \begin{equation*}\begin{split}
  (\mu_+)^{1-m}d_i^{2-p}r_i^p&\leq  (\mu_+)^{1-m}4^{p-2}(\alpha_{i-1}-\alpha_i)^{2-p}r_i^p
  \\&\leq (\mu_+)^{1-m}4^{-2i}4^{101(p-2)}\left(\frac{\xi\omega}{\hat B}\right)^{2-p}C^{-p}R^p
  \leq (\mu_+)^{1-m}\left(\frac{\omega}{A}\right)^{2-p}R^p,
  \end{split}\end{equation*}
  provided that we choose
  $C=4^{101}.$
  This  proves the inclusion $Q_i\subset \hat Q$ for any  $i\geq0$. We now turn our attention to the
  proof of \eqref{djdj-1+}. To this end,
   we set $L_j=Q_j\cap \{u\geq l_j\}\cap\Omega_T$ and decompose
   $L_j=L^\prime_j\cup L^{\prime\prime}_j$, where
   \begin{equation}\begin{split}\label{Ldecomposition+}
   L^\prime_j=L_j\cap \left\{\frac{u-l_j}{d_j}\leq\epsilon_1\right\}\qquad\text{and}\qquad
   L^{\prime\prime}_j=L_j\setminus L^\prime_j.
   \end{split}\end{equation}
  In view of $u\leq l_j$ on $L_j$, we use \eqref{Kj} to deduce that
  \begin{equation}\begin{split}\label{0st estimate+}
 \frac{(\mu_+)^{m-1}d_j^{p-2}}{r_j^{n+p}}|L_j|
\leq \frac{4^n}{r_{j-1}^n}\esssup_t\int_{B_{j-1}}G\left(\frac{u-l_{j-1}}{d_{j-1}}\right)
  \varphi_{j-1}^k
 \,\mathrm {d}x\leq 4^n\kappa.
 \end{split}\end{equation}
 It follows from \eqref{0st estimate+} that
 \begin{equation}\begin{split}\label{1st estimate+}
 &\frac{(\mu_+)^{m-1}d_j^{p-2}}{r_j^{n+p}}\iint_{L_j^\prime}
 \left(\frac{u-l_j}{d_j}\right)^{(1+\lambda)(p-1)}\varphi_j^{k-p}
 \,\mathrm {d}x\mathrm {d}t
 \\&\leq \frac{(\mu_+)^{m-1}d_j^{p-2}}{r_j^{n+p}}\epsilon_1^{(1+\lambda)(p-1)}|L_j|
\leq 4^n\epsilon_1^{(1+\lambda)(p-1)}\kappa.
 \end{split}\end{equation}
 Let parameters $h$ and $q$ be as in \eqref{hq}. For a fixed  $\epsilon_2<1$, we apply Young's inequality to conclude that
  \begin{equation*}\begin{split}
 &\frac{(\mu_+)^{m-1}d_j^{p-2}}{r_j^{n+p}}\iint_{L_j^{\prime\prime}}
 \left(\frac{u-l_j}{d_j}\right)^{(1+\lambda)(p-1)}\varphi_j^{k-p}
 \,\mathrm {d}x\mathrm {d}t
 \\&\leq\epsilon_2\frac{(\mu_+)^{m-1}d_j^{p-2}}{r_j^{n+p}}|L_j|
+\gamma(\epsilon_2)\frac{(\mu_+)^{m-1}d_j^{p-2}}{r_j^{n+p}}
 \iint_{L_j^{\prime\prime}}\left(\frac{u-l_j}{d_j}\right)^{p\frac{n+h}{nh}}\varphi_j^{(k-p)q}
 \,\mathrm {d}x\mathrm {d}t
 \\&=:T_1+T_2,
  \end{split}\end{equation*}
  with the obvious meanings of $T_1$ and $T_2$.
According to \eqref{0st estimate+}, we see that $T_1\leq 4^n\epsilon_2\kappa$.
To estimate $T_2$, we use Lemma \ref{lemmainequalitypsi+} with $(l,d)$ replaced by $(l_j,d_j)$ to deduce that
   \begin{equation}\begin{split}\label{T21+}
   T_2\leq \gamma\frac{(\mu_+)^{m-1}(l_j-\bar l)^{p-2}}{r_j^{n+p}}\iint_{L_j^{\prime\prime}}
   \psi_j^{p\frac{n+h}{n}}\varphi_j^{(k-p)q} \,\mathrm {d}x\mathrm {d}t,
   \end{split}\end{equation}
   where
\begin{equation*}\begin{split}\psi_j(x,t)=\frac{1}{d_j}\left[\int_{l_j}^u
   \left(1+\frac{s-l_j}{d_j}\right)^{-\frac{1}{p}-\frac{\lambda}{p}}\,\mathrm {d}s\right]_+.\end{split}\end{equation*}
   Let $v=\psi_j\varphi_j^{k_1}$, where $k_1=\frac{(k-p)nq}{p(n+h)}$. Recalling that $p<n$,
   we use H\"older's inequality, Sobolev's inequality and Lemma \ref{lemmainequalitypsi+}
   to deduce
   \begin{equation}\begin{split}\label{T22+}
&\iint_{L_j^{\prime\prime}}
 \psi_j^{p\frac{n+h}{n}}\varphi_j^{(k-p)q} \,\mathrm {d}x\mathrm {d}t
\\& \leq \int_{t_1- (\mu_+)^{1-m}d_j^{2-p}r_j^p}^{t_1}\left(\int_{B_j}
v^{\frac{np}{n-p}} \,\mathrm {d}x\right)^{\frac{n-p}{n}}
\left(\int_{L_j^{\prime\prime}(t)}
v^{h} \,\mathrm {d}x\right)^{\frac{p}{n}}
\mathrm {d}t
\\&\leq  \gamma\esssup_t\left(\int_{L_j^{\prime\prime}(t)}
\frac{u-l_j}{d_j}\varphi_j^{k_1h} \,\mathrm {d}x\right)^{\frac{p}{n}}\iint_{Q_j}
|Dv|^p \,\mathrm {d}x\mathrm {d}t,
   \end{split}\end{equation}
   where
     $$L_j^{\prime\prime}(t)=\{x\in B_j:u(\cdot,t)\geq l_j\}\cap \left\{x\in B_j:\frac{u(\cdot,t)-l_j}{d_j}>
   \epsilon_1\right\}.$$
   According to Lemma \ref{Gk}, we find that
   \begin{equation}\begin{split}\label{T23+}
&   \int_{L_j^{\prime\prime}(t)}
\frac{u-l_j}{d_j}\varphi_j^{k_1h} \,\mathrm {d}x
\leq c(\epsilon_1)
\int_{L_j(t)}
G\left(\frac{u-l_j}{d_j} \right)\varphi_j^{k_1h}\,\mathrm {d}x.
    \end{split}\end{equation}
    At this point,
    we introduce the quantities $M_1$ and $M_2$ by
\begin{equation*}\begin{split}
M_1= \frac{r_j^{p-n}}{d_j^p}(\mu_+)^{1-m}\int_{B_j}g^{\frac{p}{p-1}}\,\mathrm {d}x\qquad\text{and}\qquad M_2=
 \frac{r_j^{p-n}}{d_j^{p-1}}(\mu_+)^{1-m}\int_{B_j}|f|\,\mathrm {d}x.
 \end{split}\end{equation*}
We now apply Lemma \ref{Cac2} with $(a,d,l)$ replaced by  $(\mu_+,d_j,l_j)$ to conclude that
     \begin{equation*}\begin{split}
  \esssup_t&
  \frac{1}{r_j^n}
  \int_{B_jp}
G\left(\frac{u-l_j}{d_j} \right)\varphi_j^{k_1h}\,\mathrm {d}x
  \\
 \leq &
 \gamma  \frac{d_j^{p-2}}{r_j^{p+n}}\iint_{L_j}u^{m-1}\left(\frac{u-l_j}{d_j}\right)^{(1+\lambda)(p-1)}
  \varphi_j^{k_1h-p}\,\mathrm {d}x\mathrm {d}t
  \\&+\gamma
    \frac{1}{r_j^n}\iint_{L_j}\frac{u-l_j}{d_j}|\partial_t\varphi_j|\,\mathrm {d}x\mathrm {d}t
   \\&+\gamma \frac{r_j^{p-n}}{d_j^p}(\mu_+)^{1-m}\int_{B_j}g^{\frac{p}{p-1}}\,\mathrm {d}x
  +\gamma \frac{r_j^{p-n}}{d_j^{p-1}}(\mu_+)^{1-m}\int_{B_j}|f|\,\mathrm {d}x
  \\=&:T_3+T_4+\gamma M_1+\gamma M_2,
   \end{split}\end{equation*}
   since $\varphi_j=0$ on $\partial_PQ_j$.
   We first consider the estimate for $T_3$.
According to \eqref{Kj} and \eqref{djdj-1proof+}, we conclude that
    \begin{equation*}\begin{split}
    T_3\leq &\gamma\frac{d_j^{p-2-(1+\lambda)(p-1)}(\mu_+)^{m-1}}{r_j^{p+n}}\iint_{L_j}
    (u-l_j)^{(1+\lambda)(p-1)}
  \varphi_j^{k_1h-p}\,\mathrm {d}x\mathrm {d}t
  \\ \leq &\gamma\frac{d_{j-1}^{p-2}(\mu_+)^{m-1}}{r_{j-1}^{p+n}}\iint_{L_{j-1}}
  \left(\frac{u-l_{j-1}}{d_{j-1}}\right)^{(1+\lambda)(p-1)}
  \varphi_{j-1}^{k-p}\,\mathrm {d}x\mathrm {d}t
  \\ \leq & \gamma K_{i-1}(l_i)\leq \gamma\kappa.
  \end{split}\end{equation*}
Taking into account that $(1+\lambda)(p-1)>1$, $l_{j-1}\leq l_j$, $d_j>4^{-1}d_{j-1}$ and
  $|\partial_t\varphi_j|\leq 9d_j^{p-2}(\mu_+)^{m-1}r_j^{-p}$, we infer from \eqref{0st estimate+} that
     \begin{equation*}\begin{split}
    T_4\leq &\gamma
    \frac{d_j^{p-2}(\mu_+)^{m-1}}{r_j^{n+p}}\iint_{L_j}\frac{u-l_j}
    {d_j}\,\mathrm {d}x\mathrm {d}t
    \\ \leq &\gamma\frac{d_j^{p-2}(\mu_+)^{m-1}}{r_j^{n+p}}|L_j|
  +\gamma
    \frac{d_j^{p-2}(\mu_+)^{m-1}}{r_j^{n+p}}\iint_{L_j}\left(\frac{u-l_j}
    {d_j}\right)^{(1+\lambda)(p-1)}\,\mathrm {d}x\mathrm {d}t
  \\ \leq & \gamma\kappa+K_{i-1}(l_i)\leq \gamma\kappa.
    \end{split}\end{equation*}
Combining the above estimates,
    we infer that there exists a constant $\gamma$ depending only upon the data such that
    \begin{equation}\begin{split}\label{important G+}
  \esssup_t&
  \frac{1}{r_j^n}
  \int_{B_j}
G\left(\frac{u-l_j}{d_j} \right)\varphi_j^{k_1h}\,\mathrm {d}x\leq\gamma\kappa+\gamma(M_1+M_2).
    \end{split}\end{equation}
We now turn our attention to the estimate of $T_2$.
    Combining \eqref{T21+}-\eqref{important G+}, we can rewrite the upper bound for $T_2$ by
    \begin{equation*}\begin{split}
    T_2\leq & \gamma \frac{(\mu_+)^{m-1}d_j^{p-2}}{r_j^n}(\kappa+M_1+M_2)^{\frac{p}{n}}
    \\&\times\left[\iint_{Q_j}
\varphi_j^{k_1p}|D\psi_j|^p \,\mathrm {d}x\mathrm {d}t
+\iint_{Q_j}
\varphi_j^{(k_1-1)p}\psi_j^p|D\varphi_j|^p \,\mathrm {d}x\mathrm {d}t\right]
\\=&:\gamma(\kappa+M_1+M_2)^{\frac{p}{n}}(T_7+T_8),
    \end{split}\end{equation*}
    with the obvious meanings of $T_7$ and $T_8$. We first consider the  estimate for $T_7$. Noting that $\xi<4^{-1}$
    and $\omega\leq 2\mu_+$, we have
$u\geq l_j=\mu_+-\xi\omega\geq \mu_+-2\xi\mu_+>\frac{1}{2}\mu_+$ on $L_j$
    and hencep
    \begin{equation*}\begin{split}
    T_7&=\frac{(\mu_+)^{m-1}d_j^{p-2}}{r_j^n}
    \iint_{L_j}
\varphi_j^{k_1p}|D\psi_j|^p \,\mathrm {d}x\mathrm {d}t
\leq \gamma \frac{d_j^{p-2}}{r_j^n}
    \iint_{Q_j}u^{m-1}
|D\psi_j|^p \varphi_j^{k_1p}\,\mathrm {d}x\mathrm {d}t.
    \end{split}\end{equation*}
To proceed further, we use Lemma \ref{Cac1}
with $(a,d,l)$ replaced by  $(\mu_+,d_j,l_j)$
and taking into account the estimates for $T_3$-$T_6$. Since $\varphi_j=0$ on $\partial_PQ_j$, we conclude that
  \begin{equation*}\begin{split}
  T_7\leq &\gamma \frac{d_j^{p-2}}{r_j^n}
    \iint_{Q_j}u^{m-1}
|D\psi_j|^p \varphi_j^{k_1p}\,\mathrm {d}x\mathrm {d}t
  \\
 \leq &\gamma  \frac{d_j^{p-2}}{r_j^p}\iint_{L_j}u^{m-1}\left(\frac{u-l_j}{d_j}\right)^{(1+\lambda)(p-1)}
 \varphi_j^{(k_1-1)p}\,\mathrm {d}x\mathrm {d}t
  +\gamma \iint_{L_j}\frac{u-l_j}{d_j}|\partial_t\varphi_j|\,\mathrm {d}x\mathrm {d}t
\\&+\gamma \frac{r_j^p}{d_j^p}(\mu_+)^{1-m}\int_{B_j}g^{\frac{p}{p-1}}\,\mathrm {d}x
 +\gamma \frac{r_j^p}{d_j^{p-1}}(\mu_+)^{1-m}\int_{B_j}|f|\,\mathrm {d}x
  \\ \leq &
  \gamma\frac{d_{j-1}^{p-2}(\mu_+)^{m-1}}{r_{j-1}^{p+n}}\iint_{L_{j-1}}\left(\frac{u-l_{j-1}}{d_{j-1}}\right)^{(1+\lambda)(p-1)}
  \varphi_{j-1}^{k-p}\,\mathrm {d}x\mathrm {d}t
    \\&+\gamma\frac{d_j^{p-2}(\mu_+)^{m-1}}{r_j^{n+p}}|L_j|+M_1+M_2
 \\ \leq&\gamma (\kappa+M_1+M_2).
  \end{split}\end{equation*}
  To estimate $T_8$, we first note that
   \begin{equation*}\begin{split}\psi_j(x,t)=\frac{1}{d_j}\left[\int_{l_j}^u
   \left(1+\frac{s-l_j}{d_j}\right)^{-1-\lambda}\,\mathrm {d}s\right]^{\frac{1}{p}}(u-l_j)_+^{\frac{1}{p^\prime}}
   \leq \frac{(u-l_j)_+^{\frac{1}{p^\prime}}}{d_j^{\frac{1}{p^\prime}}}.
   \end{split}\end{equation*}
 Consequently, we infer that
    \begin{equation*}\begin{split}
  T_8&\leq \frac{(\mu_+)^{m-1}d_j^{p-2}}{r_j^n}\iint_{Q_j}
\varphi_j^{(k_1-1)p}\psi_j^p|D\varphi_j|^p \,\mathrm {d}x\mathrm {d}t
\\&\leq  \frac{(\mu_+)^{m-1}d_j^{p-2}}{r_j^{n+p}}\iint_{L_j}
\left(\frac{u-l_j}{d_j}\right)^{p-1}
\varphi_j^{(k_1-1)p} \,\mathrm {d}x\mathrm {d}t
\\&\leq \frac{(\mu_+)^{m-1}d_j^{p-2}}{r_j^{n+p}}|L_j|
+\frac{(\mu_+)^{m-1}d_j^{p-2}}{r_j^{n+p}}\iint_{L_j}
\left(\frac{u-l_j}{d_j}\right)^{(1+\lambda)(p-1)}
\varphi_j^{(k_1-1)p} \,\mathrm {d}x\mathrm {d}t
\\&\leq \gamma \kappa
  \end{split}\end{equation*}
  and hence we arrive at
 $T_2\leq \gamma (\kappa+M_1+M_2)^{1+\frac{p}{n}}$.
Combining the above estimates, we conclude with
   \begin{equation}\begin{split}\label{Lprimeprime+}
  &\frac{(\mu_+)^{m-1}d_j^{p-2}}{r_j^{n+p}}\iint_{L_j^{\prime\prime}}
  \left(\frac{u-l_j}{d_j}\right)^{(1+\lambda)(p-1)}\varphi_j^{k-p}
 \,\mathrm {d}x\mathrm {d}t
 \\&\leq 4^n\epsilon_2\kappa+\gamma(\epsilon_2)(\kappa+M_1+M_2)^{1+\frac{p}{n}}.
 \end{split}\end{equation}
 This also yields that
  \begin{equation}\begin{split}\label{AA1+}
  &\frac{(\mu_+)^{m-1}d_j^{p-2}}{r_j^{n+p}}\iint_{L_j}\left(\frac{u-l_j}{d_j}
  \right)^{(1+\lambda)(p-1)}\varphi_j^{k-p}
 \,\mathrm {d}x\mathrm {d}t
 \\&\leq 4^n\epsilon_1^{(1+\lambda)(p-1)}\kappa+
 4^n\epsilon_2\kappa+\gamma(\epsilon_2)(\kappa+M_1+M_2)^{1+\frac{p}{n}}.
 \end{split}\end{equation}
 Furthermore,
we aim to improve the estimate \eqref{important G+}. To this end,
 we apply Lemma \ref{Cac2} to obtain
 \begin{equation}\begin{split}\label{estimateforG+}
 & \esssup_t
  \frac{1}{r_j^n}
  \int_{L_j(t)}
G\left(\frac{u-l_j}{d_j} \right)\varphi_j^k\,\mathrm {d}x
  \\
& \leq
 \gamma  \frac{d_j^{p-2}}{r_j^{p+n}}\iint_{L_j}u^{m-1}\left(\frac{u-l_j}{d_j}\right)^{(1+\lambda)(p-1)}
  \varphi_j^{k-p}\,\mathrm {d}x\mathrm {d}t
  \\&+\gamma
    \frac{1}{r_j^n}\iint_{L_j}\frac{u-l_j}{d_j}\varphi_j^{k-1}
    |\partial_t\varphi_j|\,\mathrm {d}x\mathrm {d}t
   \\&+\gamma \frac{r_j^{p-n}}{d_j^p}(\mu_+)^{1-m}\int_{B_j}g^{\frac{p}{p-1}}\,\mathrm {d}x
  +\gamma \frac{r_j^{p-n}}{d_j^{p-1}}(\mu_+)^{1-m}\int_{B_j}|f|\,\mathrm {d}x
  \\&=:S_1+S_2+\gamma M_1+\gamma M_2,
   \end{split}\end{equation}
   with the obvious meanings of $S_1$ and $S_2$. To estimate $S_1$, we apply \eqref{AA1+} to deduce that
   \begin{equation*}\begin{split}
   S_1&\leq 2^{(1+\lambda)(p-1)}
   \gamma\frac{d_j^{p-2}(\mu_+)^{m-1}}{r_j^{p+n}}\iint_{L_j}\left(\frac{u-l_j}{d_j}\right)^{(1+\lambda)(p-1)}
  \varphi_j^{k-p}\,\mathrm {d}x\mathrm {d}t
  \\&\leq 2^{(1+\lambda)(p-1)}
   \gamma\left[4^n\epsilon_1^{(1+\lambda)(p-1)}\kappa+
 4^n\epsilon_2\kappa+\gamma(\epsilon_2)(\kappa+M_1+M_2)^{1+\frac{p}{n}}\right].
   \end{split}\end{equation*}
   Finally, we consider the estimate for $S_2$.
To this end, we decompose
   $L_j=L^\prime_j\cup L^{\prime\prime}_j$, where $L^\prime_j$ and $L^{\prime\prime}_j$
   satisfy \eqref{Ldecomposition+}. In view of $|\partial_t\varphi_j|\leq 9d_j^{p-2}(\mu_+)^{m-1}r_j^{-p}$, we
   use \eqref{0st estimate+}
   and \eqref{Lprimeprime+} to conclude that
    \begin{equation*}\begin{split}
   S_2\leq & \gamma
    \frac{d_j^{p-2}(\mu_+)^{m-1}}{r_j^{n+p}}\iint_{L_j}\frac{u-l_j}{d_j}\varphi_j^{k-1}
    \,\mathrm {d}x\mathrm {d}t
   \\
   \leq &\gamma\epsilon_1\frac{(\mu_+)^{m-1}d_j^{p-2}}{r_j^{n+p}}|L_j^\prime|
   +\gamma
   \frac{d_j^{p-2}(\mu_+)^{m-1}}{r_j^{p+n}}\iint_{L_j^{\prime\prime}}
   \left(\frac{u-l_j}{d_j}\right)^{(1+\lambda)(p-1)}
  \varphi_j^{k-p}\,\mathrm {d}x\mathrm {d}t
  \\ \leq & 4^n\epsilon_1\kappa+\gamma\left[4^n\epsilon_2\kappa+
  \gamma(\epsilon_2)(\kappa+M_1+M_2)^{1+\frac{p}{n}}\right].
     \end{split}\end{equation*}
Inserting the estimates for $S_1$ and $S_2$
     into \eqref{estimateforG+} and taking into account \eqref{AA1+}, we arrive at
     \begin{equation*}\begin{split}
    K_j(l_{j+1})\leq &\gamma (M_1+M_2)+2^{(1+\lambda)(p-1)}
   \gamma4^n(\epsilon_1+
 \epsilon_2)\kappa
 +\gamma(\epsilon_2)(\kappa+M_1+M_2)^{1+\frac{p}{n}}.
      \end{split}\end{equation*}
Recalling that $  \kappa=K_j(l_{j+1})$, we conclude that there exist constants $\hat \gamma_2=\hat\gamma_2(\text{data})$
and $\hat \gamma_3=\hat \gamma_3(\text{data},\epsilon_1,\epsilon_2)$ such that
  \begin{equation*}\begin{split}
  \kappa=K_j(l_{j+1})\leq&
  \hat\gamma_3 \frac{r_j^{p-n}}{d_j^p}(\mu_+)^{1-m}\int_{B_j}g^{\frac{p}{p-1}}\,\mathrm {d}x+
  \hat\gamma_3 \frac{r_j^{p-n}}{d_j^{p-1}}(\mu_+)^{1-m}\int_{B_j}|f|\,\mathrm {d}x
  \\
  &+\hat\gamma_2(\epsilon_1+
 \epsilon_2)\kappa
 +\hat\gamma_3\kappa^{1+\frac{p}{n}}\\&+\hat\gamma_3\left[
 \frac{r_j^{p-n}}{d_j^p}(\mu_+)^{1-m}\int_{B_j}g^{\frac{p}{p-1}}\,\mathrm {d}x+
 \frac{r_j^{p-n}}{d_j^{p-1}}(\mu_+)^{1-m}\int_{B_j}|f|\,\mathrm {d}x\right]^{1+\frac{p}{n}}.
  \end{split}\end{equation*}
  At this point, we first determine the values of $\epsilon_1$ and $\epsilon_2$ by
   $\epsilon_1=\epsilon_2=(8\hat\gamma_2)^{-1}.$
    The choices of $\epsilon_1$ and $\epsilon_2$ determine the value of $\hat\gamma_3$. Moreover, we choose $\kappa$ be such that
  \begin{equation}\begin{split}\label{kappa}
  \kappa=4^{-\frac{n}{p}}\hat\gamma_3^{-\frac{n}{p}}.
   \end{split}\end{equation}
 The choices of $\epsilon_1$, $\epsilon_2$ and $\chi$ ensuring
$ \hat\gamma_2(\epsilon_1+
 \epsilon_2)\kappa
 +\hat\gamma_3\kappa^{1+\frac{p}{n}}\leq\frac{1}{2}\kappa$
and we conclude that either
  \begin{equation*}\begin{split}
  d_j\leq \gamma\left(r_j^{p-n}(\mu_+)^{1-m}\int_{B_j}g^{\frac{p}{p-1}}\,\mathrm {d}x\right)^{\frac{1}{p}}
  \qquad\text{or}\qquad
  d_j\leq
  \gamma \left(r_j^{p-n}(\mu_+)^{1-m}\int_{B_j}|f|\,\mathrm {d}x\right)^{\frac{1}{p-1}}
   \end{split}\end{equation*}
   holds, which proves the inequality \eqref{djdj-1proof+}.

Step 3: \emph{Proof of the inequality \eqref{DeGiorgi1+}.}
   Let $J>1$ be a fixed integer. We sum up the inequality \eqref{djdj-1proof+} for $j=1,\cdots,J-1$
   and obtain
   \begin{equation*}\begin{split}
   l_J-l_1\leq& \frac{1}{3}d_0+\gamma\frac{\xi\omega}{\hat B}
   \sum_{j=1}^{J-1} 4^{-j}+\gamma\sum_{j=1}^{J-1}\left(r_{j-1}^{p-n}(\mu_+)^{1-m}\int_{B_{j-1}}|f(y)|
  \,\mathrm {d}y\right)^{\frac{1}{p-1}}
\\&+ \gamma\sum_{j=1}^{J-1}\left(r_{j-1}^{p-n}(\mu_+)^{1-m}\int_{B_{j-1}} g(y)^{\frac{p}{p-1}}
  \,\mathrm {d}y\right)^{\frac{1}{p}}.
   \end{split}\end{equation*}
   In view of $d_0=l_1-l_0$ and $l_0=\mu_+-\xi\omega$, we apply \eqref{omega1violated+} to conclude that
   there exists a constant $\hat\gamma_4=\hat\gamma_4(\text{data})>1$ such that
  \begin{equation}\begin{split}\label{uupperbound+}
l_J\leq
 & \frac{4}{3}d_0+\mu_+-\xi\omega+\gamma \frac{\xi\omega}{\hat B}
  \\&+\gamma(\mu_+)^{\frac{1-m}{p-1}}\int_0^{4C^{-1}R}\left(\frac{1}{r^{n-p}}\int_{B_r(x_1)}f(y)
\,\mathrm {d}y\right)^{\frac{1}{p-1}}\frac{1}{r}\,\mathrm {d}r
  \\&+\gamma(\mu_+)^{\frac{1-m}{p}}\int_0^{4C^{-1}R}\left(\frac{1}{r^{n-p}}\int_{B_r(x_1)}g(y)^{\frac{p}{p-1}}
\,\mathrm {d}y\right)^{\frac{1}{p}}\frac{1}{r}\,\mathrm {d}r
\\ \leq & \frac{4}{3}d_0+\mu_+-\xi\omega+\hat\gamma_4\frac{\xi\omega}{\hat B}.
   \end{split}\end{equation}
According to \eqref{l0+}, we find that $d_0\leq \bar l-l_0\leq\frac{1}{2}\xi\omega$.
  Passing to the limit $J\to\infty$, we infer from \eqref{uupperbound+} that
   \begin{equation*}\begin{split}
   u(x_1,t_1)<\mu_+-\frac{1}{6}\xi\omega,
    \end{split}\end{equation*}
    provided that we choose $\hat B>6\hat\gamma_4$.
    This shows that the estimate \eqref{DeGiorgi1+} holds for
    almost everywhere point in $Q_A^{(2)}$,
    since $(x_1,t_1)$ is the Lebesgue point of $u$.
    Taking into account that $C=4^{101}$, \eqref{hatB} and \eqref{determine the value of nu2}, we determine the constant $\hat B$ by
   \begin{equation}\label{hatBdetermine}
   \hat B=\max\left\{\left(\frac{\kappa}{4^{101(n-p)+2}\hat\gamma_1}\right)^{\frac{1}{1-p}},\  6\hat\gamma_4,\
   4^{102}B,\ \left(\frac{100p}{p-2}\right)^{\frac{1}{p-2}},\ A_p^\frac{1}{p-2}\right\},
   \end{equation}
   where the constant $\kappa$ is determined via \eqref{kappa}.
    For such a choice of $\hat B$, we determine the value of $\nu_2$ via \eqref{determine the value of nu2} by
    \begin{equation}\label{nu2determine}
    \nu_2=\frac{\kappa}{4^{101(n-p)+1}(\hat\gamma_0+\hat\gamma_1)\hat B^{p-2}}.
     \end{equation}
     We find that the constant $\nu_2$ depends only upon the data and independent of $\xi$.
    The proof of the Lemma is now complete.
\end{proof}
The crucial step
to obtain the main result proved in this section is the following lemma concerning the estimate of the measure of level sets.
\begin{lemma}\label{DeGiorgi3}
Let $u$ be a bounded nonnegative weak solution to \eqref{parabolic}-\eqref{A} in $\Omega_T$.
For every $\bar\nu\in(0,1)$, there exists a positive integer $q_*=q_*(\text{data},\bar\nu)$ such that either
\begin{equation}\begin{split}\label{2st omega assumption}
\omega\leq 2^{s_1+q_*}(\mu_+)^{\frac{1-m}{p}}F_1(2R)+2^{s_1+q_*}(\mu_+)^{\frac{1-m}{p-1}}F_2(2R)
\end{split}\end{equation}
or
\begin{equation}\label{measure estimate 2nd}\left|\left\{(x,t)\in Q_A^{(1)}:u\geq\mu_+-\frac{\omega}{2^{s_1+q_*}}\right\}
\right|\leq \bar\nu|Q_A^{(1)}|,\end{equation}
where
$s_1>2$ is the constant claimed by Lemma \ref{time 2nd} and
$A=2^{s_1+q_*}\hat B$.
\end{lemma}
\begin{proof}
For simplicity, we abbreviate $Q_A=Q_A^{(1)}$.
To start with, we first assume that \eqref{2ndomega assumption1} is violated, that is,
\begin{equation}\begin{split}\label{omegaviolated++}
\omega> 2^{\frac{s_1+n}{p-1}}(\mu_+)^{\frac{1-m}{p-1}}F_2(2R)
+2^{\frac{2s_1+n}{p}}(\mu_+)^{\frac{1-m}{p}}F_1(2R).
\end{split}\end{equation}
According to Lemma \ref{time 2nd}, we infer that the slice-wise estimate
\begin{equation}\label{2nd measure estimate++}
\left|\left\{x\in B_{\frac{3}{4}R}:u(x,t)>\mu_+-\frac{\omega}{2^{s_1}}\right\}\right|\leq
\left(1-\left(\frac{\nu_0}{2}\right)^2\right)|B_{\frac{3}{4}R}|
\end{equation}
holds
for all $t\in \left(-\frac{1}{2}A^{p-2}(\mu_+)^{1-m}\omega^{2-p}R^p,0\right]$.
Let $q_*>1$ to be determined in the course of the proof.
For $j>2$ and $A=2^{s_1+q_*}\hat B$, we define
\begin{equation*}\begin{split}
%Q_A&=B_{\frac{3}{4}R}\times\left(-\frac{1}{2}A^{p-2}(\mu_+)^{1-m}\omega^{2-p}R^p,0\right],
%\\
Q_A^\prime&=B_{R}\times\left(-A^{p-2}(\mu_+)^{1-m}\omega^{2-p}R^p,0\right],
\\
A_j(t)&=\left\{x\in B_{\frac{3}{4}R}:u(x,t)>\mu_+-\frac{\omega}{2^j}\right\}
\end{split}
\end{equation*}
and
\begin{equation*}Q_j^+=\left\{(x,t)\in Q_A:u(x,t)>\mu_+-\frac{\omega}{2^j}\right\}.\end{equation*}
Now for $q_*\in\mathbb{N}$, we set $k_j=\mu_+-2^{-j}\omega$ for $j=s_1,s_1+1,\cdots, s_1+q_*$.
Take a cutoff function $0\leq\varphi\leq1$,
such that $\varphi=1$ in $Q_A$, $\varphi=0$ on $\partial_PQ_A^\prime$,
\begin{equation*}\begin{split}
|D\varphi|\leq \frac{4}{R}\qquad\text{and}\qquad 0\leq
\frac{\partial \varphi}{\partial t}\leq2A^{2-p}(\mu_+)^{m-1}\omega^{p-2}R^{-p}.
\end{split}\end{equation*}
We consider the Caccioppoli estimate \eqref{Caccioppoli} for the truncated
functions $(u-k_j)_+$ over the cylinder $Q_A^\prime$
and obtain
\begin{equation*}\begin{split}
&\iint_{Q_A^\prime}u^{m-1}|D(u-k_j)_+\varphi|^p \,\mathrm {d}x\mathrm {d}t\\
&\leq \gamma\iint_{Q_A^\prime} u^{m-1}(u-k_j)_+^p|D\varphi|^p
\,\mathrm {d}x\mathrm {d}t+\gamma
\iint_{Q_A^\prime} (u-k_j)_+^2|\partial_t\varphi|\,\mathrm {d}x\mathrm {d}t
\\&+\gamma
\iint_{Q_A^\prime} |f|(u-k_j)_+\,\mathrm {d}x\mathrm {d}t+\gamma
\iint_{Q_A^\prime} g^{\frac{p}{p-1}}\,\mathrm {d}x\mathrm {d}t
.\end{split}\end{equation*}
Taking into accouont that $j>2$ and $\omega<2\mu_+$, we have $u\geq k_j>\mu_+-\frac{1}{4}\omega>\frac{1}{2}
\mu_+$ on the set $\{u\geq k_j\}$.
Then, we deduce
\begin{equation*}\begin{split}
\iint_{Q_A^\prime}u^{m-1}|D(u-k_j)_+\varphi|^p \,\mathrm {d}x\mathrm {d}t\geq
\left(\frac{\mu_+}{2}\right)^{m-1}\iint_{Q_j^+}|Du|^p \,\mathrm {d}x\mathrm {d}t.
\end{split}\end{equation*}
In view of $(u-k_j)_+\leq 2^{-j}\omega$ and $A=2^{s_1+q_*}\hat B$, we conclude that
\begin{equation*}\begin{split}
\iint_{Q_A^\prime} u^{m-1}(u-k_j)_+^p|D\varphi|^p
\,\mathrm {d}x\mathrm {d}t\leq \gamma\frac{(\mu_+)^{m-1}}{R^p}\left(\frac{\omega}{2^j}\right)^p|Q_A|
\end{split}\end{equation*}
and
\begin{equation*}\begin{split}
\iint_{Q_A^\prime}& (u-k_j)_+^2|\partial_t\varphi|\,\mathrm {d}x\mathrm {d}t\leq
\gamma\frac{(\mu_+)^{m-1}\omega^{p-2}}{A^{p-2}R^p}\left(\frac{\omega}{2^j}\right)^2|Q_A|
\\&=\gamma\frac{(\mu_+)^{m-1}}{R^p}\left(\frac{\omega}{2^{s_1+q_*}\hat B
}\right)^{p-2}\left(\frac{\omega}{2^j}\right)^2|Q_A|
\leq \gamma\frac{(\mu_+)^{m-1}}{R^p}\left(\frac{\omega}{2^j}\right)^p|Q_A|,
\end{split}\end{equation*}
since $j\leq s_1+q_*$ and $\hat B>1$. Furthermore, we deal with the estimates for the lower
order terms. In view of $u-k_j\leq 2^{-j}\omega$, we find that
\begin{equation*}\begin{split}
\iint_{Q_A^\prime}& |f|(u-k_j)_+\,\mathrm {d}x\mathrm {d}t\leq \gamma\left(\frac{\omega}{2^jR^p}\right)
\left(R^{p-n}\int_{B_R}|f|\,\mathrm{d}x
\right)|Q_A|
\\&=\gamma\left(\frac{(\mu_+)^{m-1}}{R^p}\left(\frac{\omega}{2^j}\right)^p|Q_A|\right)\left[
\left(\frac{\omega}{2^j}\right)^{1-p}(\mu_+)^{1-m}
R^{p-n}\int_{B_R}|f|\,\mathrm{d}x\right]
\\&\leq \gamma\frac{(\mu_+)^{m-1}}{R^p}\left(\frac{\omega}{2^j}\right)^p|Q_A|,
\end{split}\end{equation*}
provided that we assume
\begin{equation}\begin{split}\label{1st omega}
\omega\geq 2^j(\mu_+)^{\frac{1-m}{p-1}}\left(R^{p-n}\int_{B_R}|f|\,\mathrm{d}x
\right)^{\frac{1}{p-1}}.
\end{split}\end{equation}
Finally, we deduce that
\begin{equation*}\begin{split}
\iint_{Q_A^\prime} g^{\frac{p}{p-1}}\,\mathrm {d}x\mathrm {d}t&\leq
\gamma\left(\frac{(\mu_+)^{m-1}}{R^p}\left(\frac{\omega}{2^j}\right)^p|Q_A|\right)\left[
\left(\frac{\omega}{2^j}\right)^{-p}(\mu_+)^{1-m}
R^{p-n}\int_{B_R}g^{\frac{p}{p-1}}\,\mathrm{d}x\right]
\\&\leq \gamma\frac{(\mu_+)^{m-1}}{R^p}\left(\frac{\omega}{2^j}\right)^p|Q_A|,
\end{split}\end{equation*}
provided that we assume
\begin{equation}\begin{split}\label{2st omega}
\omega\geq 2^j(\mu_+)^{\frac{1-m}{p}}\left(R^{p-n}\int_{B_R}g^{\frac{p}{p-1}}\,\mathrm{d}x
\right)^{\frac{1}{p}}.
\end{split}\end{equation}
Combining the above estimates, we conclude that there exists a constant $\gamma$ depending only upon the data such that
\begin{equation}\begin{split}\label{Du}
\iint_{Q_j^+}|Du|^p \,\mathrm {d}x\mathrm {d}t\leq \gamma\frac{1}{R^p}\left(\frac{\omega}{2^j}\right)^p|Q_A|.
\end{split}\end{equation}
To proceed further, we apply a De Giorgi type Lemma (see \cite[Lemma 2.2]{Di93}) to the function $u(\cdot,t)$,
for all $-\frac{1}{2}A^{p-2}(\mu_+)^{1-m}\omega^{2-p}R^p\leq t\leq0$, and with $l=\mu_+-2^{-(j+1)}\omega$
and $k=\mu_+-2^{-j}\omega$. According to \eqref{2nd measure estimate++}, we deduce that
\begin{equation*}\begin{split}
\left|\left\{x\in B_{\frac{3}{4}R}:u(x,t)\leq \mu_+-\frac{\omega}
{2^j}\right\}\right|=|B_{\frac{3}{4}R}|-|A_j(t)|\geq \left(\frac{\nu_0}{2}\right)^2|B_{\frac{3}{4}R}|
\end{split}\end{equation*}
holds
for all $t\in \left(-\frac{1}{2}A^{p-2}(\mu_+)^{1-m}\omega^{2-p}R^p,0\right]$.
This yields that
the estimate
\begin{equation*}\begin{split}
\frac{\omega}{2^{j+1}}|A_{j+1}(t)|\leq \frac{\gamma}{\nu_0^2}\frac{R^{n+1}}{|B_R|}
\int_{A_s(t)\setminus A_{s+1}(t)}|Du|\,\mathrm{d}x
\leq \gamma \frac{R}{\nu_0^2}
\int_{A_s(t)\setminus A_{s+1}(t)}|Du|\,\mathrm{d}x
\end{split}\end{equation*}
holds for any $-\frac{1}{2}A^{p-2}(\mu_+)^{1-m}\omega^{2-p}R^p\leq t\leq0$.
Integrating the preceding inequality
over the interval $\left(-\frac{1}{2}A^{p-2}(\mu_+)^{1-m}\omega^{2-p}R^p,0\right]$, we conclude from \eqref{Du} that
\begin{equation*}\begin{split}
\frac{\omega}{2^{j+1}}|Q_{j+1}^+|&\leq \frac{\gamma R}{\nu_0^2}\iint_{Q_j^+\setminus Q_{j+1}^+}|Du|\,\mathrm{d}x\mathrm{d}t
\\&\leq \frac{\gamma R}{\nu_0^2}\left(\iint_{Q_j^+\setminus Q_{j+1}^+}|Du|^p\,\mathrm{d}x\mathrm{d}t\right)^{\frac{1}{p}}
|Q_j^+\setminus Q_{j+1}^+|^{\frac{p-1}{p}}
\\&\leq \frac{\gamma}{\nu_0^2}\left(\frac{\omega}{2^j}\right)|Q_A|^{\frac{1}{p}}
|Q_j^+\setminus Q_{j+1}^+|^{\frac{p-1}{p}}.
\end{split}\end{equation*}
It follows that
\begin{equation*}\begin{split}
|Q_{j+1}^+|^{\frac{p}{p-1}}\leq \gamma\nu_0^{-\frac{2p}{p-1}}|Q_A|^{\frac{1}{p-1}}|Q_j^+\setminus Q_{j+1}^+|.
\end{split}\end{equation*}
Furthermore, we add up these inequalities for $j=s_1,s_1+1,\cdots,s_1+q_*-1$ and deduce that
\begin{equation*}\begin{split}
(q_*-1)|Q_{s_1+q_*}^+|^{\frac{p}{p-1}}\leq \gamma \nu_0^{-\frac{2p}{p-1}}|Q_A|^{\frac{p}{p-1}},
\end{split}\end{equation*}
where the constant $\gamma$ depends only upon the data.
From the definition of $Q_j^+$, the above inequality reads
\begin{equation*}\begin{split}
\left|\left\{(x,t)\in Q_A:u(x,t)>\mu_+-\frac{\omega}{2^{s_1+q_*}}\right\}\right|\leq \frac{\gamma}{\nu_0^2}
\frac{1}{(q_*-1)^{\frac{p-1}{p}}}|Q_A|.
\end{split}\end{equation*}
At this point, we take $q_*=q_*(\text{data},\bar\nu)\geq2$ according to
\begin{equation}\begin{split}\label{q*}
q_*=q_*(\bar\nu)=n+2+\left(\frac{\gamma}{\nu_0^2\bar\nu}\right)^\frac{p}{p-1}.
\end{split}\end{equation}
For such a choice of $q_*$, the inequality \eqref{measure estimate 2nd} holds true.
We also determine the value of $j$ in \eqref{1st omega} and \eqref{2st omega}
by $j=s_1+q_*$. On the other hand, if the  inequality \eqref{omegaviolated++},
\eqref{1st omega} or \eqref{2st omega} is violated,
then we deduce the inequality \eqref{2st omega assumption}. We have thus proved the lemma.
\end{proof}
We are now in a position to prove the decay estimate for the oscillation
of the weak solution for the second alternative.
The next proposition is our main result in this section.
\begin{proposition}\label{2nd proposition}
Let $\tilde Q_0=Q\left(\frac{1}{16}R,(\mu_+)^{1-m}\omega^{2-p}\left(\frac{1}{16}R\right)^p\right)$
and let $u$ be a bounded nonnegative weak solution to \eqref{parabolic}-\eqref{A} in $\Omega_T$.
There exist $0<\xi_2<2^{-5}$ and $B_2>1$ depending only upon the data such that
\begin{equation*}\begin{split}
\essosc_{\tilde Q_0} u\leq (1-6^{-1}\xi_2)\omega+B_2\xi_2^{-1}(
F_1(2R)^{\frac{p}{p+m-1}}+F_2(2R)^{\frac{p-1}{p+m-2}}).
 \end{split}\end{equation*}
\end{proposition}
\begin{proof}
Let $\nu_2$ be the constant claimed in Lemma \ref{lemmaDeGiorgi1+}.
We take
$\bar\nu=\nu_2$ in Lemma \ref{DeGiorgi3} and this fixes $q_*$ via \eqref{q*}.
Moreover, we fix
$\xi_2=2^{-(s_1+q_*)}$ and $A=\xi_2^{-1}\hat B$.
We first assume that \eqref{2ndomega assumption1}, \eqref{omega1+} and \eqref{2st omega assumption}
are violated, then we infer from Lemma \ref{lemmaDeGiorgi1+} that
\begin{equation*}\begin{split}
\essosc_{\tilde Q_0} u\leq (1-6^{-1}\xi_2)\omega.
 \end{split}\end{equation*}
 On the other hand, if either \eqref{2ndomega assumption1}, \eqref{omega1+} or \eqref{2st omega assumption} holds,
 then we deduce that either
 \begin{equation*}\begin{split}
 \omega\leq 2^{\frac{2s_1+n}{p}}(\mu_+)^{\frac{1-m}{p}}F_1(2R),\quad
  \omega\leq 2^{\frac{s_1+n}{p-1}}(\mu_+)^{\frac{1-m}{p-1}}F_2(2R),\quad
  \xi_2\omega\leq\hat B(\mu_+)^{\frac{1-m}{p}}F_1(2R),
  \end{split}\end{equation*}
   \begin{equation*}\begin{split}
    \xi_2\omega\leq\hat B(\mu_+)^{\frac{1-m}{p-1}}F_2(2R),\quad
    \omega\leq 2^{s_1+q_*}(\mu_+)^{\frac{1-m}{p}}F_1(2R)\quad\text{or}\quad
    \omega\leq 2^{s_1+q_*}(\mu_+)^{\frac{1-m}{p-1}}F_2(2R).
   \end{split}\end{equation*}
In view of $\mu_+\geq \frac{1}{2}\omega$, we conclude that there exists a constant $B_2>1$ depending only upon the data such that
\begin{equation*}\begin{split}
\essosc_{\tilde Q_0} u\leq\omega\leq B_2\xi_2^{-1}(
F_1(2R)^{\frac{p}{p+m-1}}+F_2(2R)^{\frac{p-1}{p+m-2}}).
 \end{split}\end{equation*}
 The proof of the proposition is now complete.
\end{proof}
\section{Proof of the main result}
This section is devoted to give the final part of the proof of Theorem \ref{main}.
Our proof follows the idea from
\cite{BDG1} and \cite{PV} but the argument is considerably more delicate in our setting.
First, we set up an iteration scheme.
Let $\{\xi_1,B_1\}$ and $\{\xi_2,B_2\}$ are the constants claimed by Proposition \ref{1st proposition} and Proposition \ref{2nd proposition},
respectively.
Let $\omega_0=\omega$, $\mu_0=\mu_+$, $R_0=R$ and $Q_0=Q_R=B_R\times (-R^{p-\epsilon_0},0)$. Moreover, we set
\begin{equation}\begin{split}\label{define omega1}
\omega_1= \eta\omega_0+\gamma[
F_1(2R_0)^{\frac{p}{p+m-1}}+F_2(2R_0)^{\frac{p-1}{p+m-2}}]\quad\text{and}\quad\mu_1=\esssup_{Q_1}u,
 \end{split}\end{equation}
 where $\eta=\max\left\{1-2^{-1}\xi_1,\ 1-6^{-1}\xi_2\right\}<1$,
 \begin{equation*}\begin{split}
 \gamma=\max\left\{B_1\xi_1^{-1},\ B_2\xi_2^{-1}\right\},\quad\text{and}\quad
 Q_1=B_{\frac{1}{16}R_0}\times\left(-\mu_0^{1-m}\omega_0^{2-p}\left(\frac{1}{16}R_0\right)^p,0\right).
  \end{split}\end{equation*}
  According to Proposition \ref{1st proposition} and Proposition \ref{2nd proposition}, we conclude that
   \begin{equation}\begin{split}\label{osc1}
   \essosc_{Q_1}u=\esssup_{Q_1}u-\essinf_{Q_1}u\leq\omega_1.
    \end{split}\end{equation}
    Our task now is to construct a smaller cylinder $Q_2$ under the assumption
that $\essosc_{Q_1}u\geq \frac{1}{2}\omega_1$.
    We first claim that $\mu_0\leq\frac{4}{\eta}\mu_1$.
    In the case $\mu_-\leq\frac{1}{2}\mu_+$, we have $\mu_+\leq\omega+\mu_-\leq\omega+\frac{1}{2}\mu_+$ and hence
    $\mu_+\leq 2\omega=2\omega_0$. In view of \eqref{define omega1}, we find that
   $\mu_0=\mu_+\leq \frac{2}{\eta}\omega_1\leq\frac{4}{\eta}\mu_1$. In the case $\mu_->\frac{1}{2}\mu_+$, we have
   $\mu_1\geq\mu_->\frac{1}{2}\mu_+$. At this point, we set
    \begin{equation*}\begin{split}
    \delta=\frac{1}{16}2^{-\frac{3m}{p}}\eta^{\frac{m+p-3}{p}}A^{\frac{2-p}{p}},\qquad R_1=\delta R_0\qquad\text{and}\qquad
    Q_2=B_{\frac{1}{16}R_1}\times\left(-\mu_1^{1-m}\omega_1^{2-p}\left(\frac{1}{16}R_1\right)^p,0\right).
    \end{split}\end{equation*}
    In view of $\mu_0\leq\frac{4}{\eta}\mu_1$ and $\omega_1\geq\eta\omega_0$, we conclude that
    \begin{equation}\begin{split}\label{hatcontained}
    \widehat Q_2:=B_{R_1}\times\left(-\mu_1^{1-m}\left(\frac{\omega_1}{A}\right)^{2-p}R_1^p,0\right)
    \subseteq B_{\frac{1}{16}R_0}\times\left(-\frac{1}{2}\mu_0^{1-m}\omega_0^{2-p}\left(\frac{1}{16}R_0\right)^p,0\right).
     \end{split}\end{equation}
Next, we claim that
       \begin{equation}\begin{split}\label{oscillation}
   \essosc_{Q_2}u\leq\omega_2,\qquad\text{where}\qquad
 \omega_2=\eta\omega_1+\gamma[
F_1(2R_1)^{\frac{p}{p+m-1}}+F_2(2R_1)^{\frac{p-1}{p+m-2}}].
    \end{split}\end{equation}
    The claim can be proved in the same way as shown before.
    In fact, we observe that the nonuniformly parabolic cylinders
     \begin{equation*}\begin{split}
     &Q_j(\bar l)=B_j\times (t_1-(\mu_+)^{1-m}(l_j-\bar l)^{2-p}r_j^p,t_1),
     \qquad \bar l=\frac{1}{2}l_j+\frac{1}{16}B\alpha_j+\frac{1}{32}\omega
     +\frac{1}{2}\mu_-
     \\&\quad\text{and}\qquad
     Q_j=Q_j(l_{j+1})=B_j\times (t_1-(\mu_+)^{1-m}(l_j-l_{j+1})^{2-p}r_j^p,t_1)
     \end{split}\end{equation*}
  constructed in the proof of Lemma \ref{lemmaDeGiorgi1} are contained in the cylinder $\hat Q$, where
$\hat Q=B_R\times\left(-(\mu_+)^{1-m}A^{p-2}\omega^{2-p}R^p,0\right).$
Similarly, we find that the cylinders
     \begin{equation*}\begin{split}
     &Q_0(\bar l)=B_0\times (t_1-(\mu_+)^{1-m}(\bar l-l_0)^{2-p}r_0^p,t_1),
     \qquad \bar l=\mu_+-\frac{1}{2}\xi\omega-\frac{1}{4}\hat B\alpha_0
     \\&\quad\text{and}\qquad
     Q_j=Q_j(l_{j+1})=B_j\times (t_1-(\mu_+)^{1-m}(l_{j+1}-l_j)^{2-p}r_j^p,t_1)
     \end{split}\end{equation*}
  considered in the proof of Lemma \ref{lemmaDeGiorgi1+} are also contained in the cylinder $\hat Q$.
    Consequently, the statements of Proposition \ref{1st proposition} and Proposition \ref{2nd proposition} are still valid if
    we modify the settings
    of $\mu_+$ and $\mu_-$ by
    \begin{equation*}\begin{split}
\mu_+=\esssup_{\hat Q}u\qquad\text{and}\qquad \mu_-=\essinf_{\hat Q}u.
\end{split}\end{equation*}
Thus we arrive at the conclusion that \eqref{oscillation} holds.
    At this point, we define $R_k=\delta R_{k-1}=\delta^kR_0$ and
    \begin{equation*}\begin{split}\omega_k=
    \eta\omega_{k-1}+\gamma[
F_1(2R_{k-1})^{\frac{p}{p+m-1}}+F_2(2R_{k-1})^{\frac{p-1}{p+m-2}}].
\end{split}\end{equation*}
We first observe that $\omega_{k+1}\geq\frac{1}{2}\omega_k$ for any $k\geq0$.
Another step in the proof is to show that $\omega_k\leq\omega_0$ and $\omega_k\to0$ as $k\to\infty$.
For simplicity of notation, we write $F_k=
F_1(2R_k)^{\frac{p}{p+m-1}}+F_2(2R_k)^{\frac{p-1}{p+m-2}}$. We see that $F_k\downarrow0$ as $k\to\infty$. For any fixed integer
$k\geq0$, we infer that for
$n>k$ there holds
\begin{equation}\begin{split}\label{omegakomgea0}
\omega_n=\eta^{n-k}\omega_k+\gamma\eta^n\sum_{j=k+1}^n\eta^{-j}F_{j-1}\leq \eta^{n-k}\omega_k+\gamma\frac{1}{1-\eta}F_k.
\end{split}\end{equation}
Then, we have $\omega_n\leq\omega_0$, provided that we choose $R>0$ so small that $F_0\leq (1-\eta)\omega_0$.
Passing to the limit $n\to\infty$, we conclude from \eqref{omegakomgea0} that
\begin{equation}\begin{split}\label{passing to the limit}
\lim_{n\to\infty}\omega_n\leq\gamma\frac{1}{1-\eta}F_k
\end{split}\end{equation}
holds for any fixed $k>1$. Passing to the limit $k\to\infty$ in \eqref{passing to the limit}, we conclude that $\omega_k\to0$ as $k\to\infty$.
Next, we consider the case $\essosc_{Q_1}u<\frac{1}{2}\omega_1$. Let $k_1\geq1$ be a smallest integer such that \begin{equation}\begin{split}\label{k1half}\frac{1}{2}\omega_{k_1}<\essosc_{Q_1}u<\omega_{k_1}.\end{split}\end{equation}
At this point, we first define $Q_{k_1+1}$ by
    \begin{equation*}\begin{split}
    Q_{k_1+1}=B_{\frac{1}{16}R_{k_1}}\times\left(-\mu_1^{1-m}\omega_{k_1}^{2-p}\left(\frac{1}{16}R_{k_1}\right)^p,0\right).
    \end{split}\end{equation*}
    Taking into account that $\omega_{k_1}\geq\eta^{k_1}\omega_0$, $R_{k_1}=\delta^{k_1}R_0$ and
     \begin{equation*}
     \begin{cases}
     \mu_0\leq 2\omega_0\leq 2\eta^{-k_1}\omega_{k_1}<4\eta^{-k_1}\mu_1\qquad&\text{if}\qquad \mu_-\leq\frac{1}{2}\mu_+,
     \\ \mu_0<2\mu_-\leq2\mu_1\qquad&\text{if}\qquad \mu_->\frac{1}{2}\mu_+.
     \end{cases}
     \end{equation*}
     It follows that
     \begin{equation}\begin{split}\label{k_1time}
     \mu_1^{1-m}\left(\frac{\omega_{k_1}}{A}\right)^{2-p}R_{k_1}^p
     &\leq 4^{m-1}\eta^{k_1(1-m)}\eta^{k_1(2-p)}A^{p-2}\mu_0^{1-m}\omega_0^{2-p}\left(\delta^{k_1}R_0\right)^p
     \\&\leq \frac{1}{2}\mu_0^{1-m}\omega_0^{2-p}\left(\frac{1}{16}R_0\right)^p,
     \end{split}\end{equation}
     since $\delta=\frac{1}{16}2^{-\frac{3m}{p}}\eta^{\frac{m+p-3}{p}}A^{\frac{2-p}{p}}$.
     Consequently, we infer that
      \begin{equation*}\begin{split}
    \widehat Q_{k_1+1}:=B_{R_{k_1}}\times\left(- \mu_1^{1-m}\left(\frac{\omega_{k_1}}{A}\right)^{2-p}R_{k_1}^p
    ,0\right)
    \subseteq B_{\frac{1}{16}R_0}\times\left(-\frac{1}{2}\mu_0^{1-m}\omega_0^{2-p}\left(\frac{1}{16}R_0\right)^p,0\right).
     \end{split}\end{equation*}
Taking into account \eqref{k1half} and \eqref{k_1time}, and using the same argument as in the proof of \eqref{oscillation},
we deduce that
$\essosc_{Q_{k_1+1}}u\leq \omega_{k_1+1}.$
     Furthermore, we introduce the cylinders $Q_2$, $\cdots$, $Q_{k_1}$ by
     \begin{equation*}\begin{split}
    Q_2&=B_{\frac{1}{16}R_1}\times\left(-(\mu^{(2)})^{1-m}(\omega^{(2)})^{2-p}\left(\frac{1}{16}R_1\right)^p,0\right),
    \\Q_3&=B_{\frac{1}{16}R_2}\times\left(-(\mu^{(3)})^{1-m}(\omega^{(3)})^{2-p}\left(\frac{1}{16}R_2\right)^p,0\right),
    \\ &\qquad\qquad \qquad\cdots\cdots
    \\ Q_{k_1}&=B_{\frac{1}{16}R_{k_1-1}}\times\left(-(\mu^{(k_1)})^{1-m}(\omega^{(k_1)})
    ^{2-p}\left(\frac{1}{16}R_{k_1-1}\right)^p,0\right),
    \end{split}\end{equation*}
    where $\mu^{(2)}=\eta\mu_0$, $\mu^{(3)}=\eta^2\mu_0$, $\cdots$, $\mu^{(k_1)}=\eta^{k_1-1}\mu_0$ and
    $\omega^{(2)}=\eta\omega_0$, $\omega^{(3)}=\eta^2\omega_0$, $\cdots$, $\omega^{(k_1)}=\eta^{k_1-1}\omega_0$.
    It is easy to check that $Q_{k_1+1}\subset Q_{k_1}\subset\cdots\subset Q_1$.
    Taking into account that $k_1\geq1$ is the smallest integer such that \eqref{k1half} holds and
     $\omega_{j+1}\geq\frac{1}{2}\omega_j$ for any $j\geq0$,
    we conclude that $\essosc_{Q_1}u<\frac{1}{2}\omega_1\leq \omega_2$ and hence
     \begin{equation*}\begin{split}\essosc_{Q_1}u\leq
     \min\{\omega_1,\omega_2,\cdots,\omega_{k_1}\}.\end{split}\end{equation*}
Consequently, we infer that
      \begin{equation*}\begin{split}\essosc_{Q_2}u&\leq\essosc_{Q_1}u\leq\omega_2,\quad
      \essosc_{Q_3}u\leq\essosc_{Q_1}u\leq\omega_3,\quad
     \\& \cdots,\quad  \essosc_{Q_{k_1}}u\leq\essosc_{Q_1}u\leq\omega_{k_1}.
    \end{split}\end{equation*}
We may now repeat the same arguments as in the previous proof and conclude that there exist two
sequences $\omega^{(n)}$, $\mu^{(n)}$ and a sequence of cylinders
     \begin{equation*}\begin{split}
    Q_n=B_{\frac{1}{16}R_{n-1}}\times\left(-(\mu^{(n)})^{1-m}(
    \omega^{(n)})^{2-p}\left(\frac{1}{16}R_{n-1}\right)^p,0\right)
    \end{split}\end{equation*}
    such that $\omega^{(n)}\leq\omega_0$, $\mu^{(n)}\leq\mu_0$,
    $Q_0\supset Q_1\supset \cdots \supset Q_n\supset\cdots$ and there holds
    \begin{equation*}\begin{split}
    \essosc_{Q_n}u\leq \omega_n
    \end{split}\end{equation*}
    for all $n=0,1,\cdots$.
    We are now in a position to establish a continuity estimate for the weak solutions.
    To this end, we fix $\bar\rho\in(0,R_0)$, $\bar r\in (0,\bar\rho)$
    and $0<b<\min\{\delta,\eta\}$. Then, there exist two integers $k\in\mathbb{N}$ and $l\in\mathbb{N}$ such that
    \begin{equation}\label{kl}k-1<\frac{\ln\frac{\bar \rho}{R_0}}{\ln\delta}\leq k\qquad\text{and}\qquad
      l-1<\frac{\ln\frac{\bar r}{\bar \rho}}{\ln b}\leq l.
    \end{equation}
    Moreover, we set $n=k+l$ and observe that
    \begin{equation*}\begin{split}
    \eta^{n-k}=\exp(\alpha l\ln b)=b^{\alpha l}\leq \left(\frac{\bar r}{\bar \rho}\right)^\alpha,\qquad\text{where}\qquad
    \alpha=\frac{\ln \eta}{\ln b}.
     \end{split}\end{equation*}
It follows from \eqref{omegakomgea0} that the estimate
     \begin{equation}\begin{split}\label{omegan}
\omega_n\leq \left(\frac{\bar r}{\bar \rho}\right)^\alpha\omega_0+\gamma\frac{1}{1-\eta}
\left(F_1(2\bar \rho)^{\frac{p}{p+m-1}}+F_2(2\bar \rho)^{\frac{p-1}{p+m-2}}\right).
\end{split}\end{equation}
holds for any $0<\bar r<\bar \rho<R_0$ and $n=k+l$, where $k$ and $l$ satisfy \eqref{kl}. Furthermore, we take $0<\bar r<R_0$ and
specify $\bar \rho=R_0^{1-\beta}\bar r^\beta$. This also determines $k_*$ and $l_*$ via \eqref{kl} by
 \begin{equation*}k_*-1<\frac{\beta}{\ln\delta}\ln\frac{\bar r}{R_0}\leq k_*\qquad\text{and}\qquad
      l_*-1<\frac{(1-\beta)}{\ln b}\ln\frac{\bar r}{R_0}\leq l_*.
    \end{equation*}
    For $n_*=k_*+l_*$, we find that the estimate \eqref{omegan} reads
    \begin{equation*}\begin{split}
\omega_{n_*}\leq \left(\frac{\bar r}{R_0}\right)^{\alpha(1-\beta)}\omega_0+\gamma
\left(F_1(2R_0^{1-\beta}\bar r^\beta)^{\frac{p}{p+m-1}}+F_2(2R_0^{1-\beta}\bar r^\beta)^{\frac{p-1}{p+m-2}}\right)
\end{split}\end{equation*}
and $n_*$ satisfies
\begin{equation}\label{n*}
n_*-2<\left(\frac{\beta}{\ln\delta}+\frac{1-\beta}{\ln b}\right)\ln\frac{\bar r}{R_0}\leq n_*.
\end{equation}
At this stage, we set $a_0=\mu_0^{1-m}\omega_0^{2-p}$ and $\bar Q_n=B_{R_n}\times(-a_0R_n^p,0)$. Recalling that
$\omega^{(n)}\leq\omega_0$ and $\mu^{(n)}\leq\mu_0$, we see that $\bar Q_n\subset Q_n$. Moreover, we define
 \begin{equation*}\begin{split}\sigma=\beta+(1-\beta)\frac{\ln\delta}{\ln b}\qquad\text{and}\qquad
 r=\delta^2R_0\left(\frac{\bar r}{R_0}\right)^\sigma.\end{split}\end{equation*}
 This yields that
 \begin{equation*}\begin{split}
 \bar r=R_0\left(\delta^{-2}R_0^{-1}r\right)^{\frac{1}{\sigma}}\qquad\text{and}\qquad
  R_0^{1-\beta}\bar r^\beta=\delta^{-2\frac{\beta}{\beta+\frac{1-\beta}{\ln b}\ln\delta}}r^{\frac{\beta}{\sigma}}
  R_0^{\frac{(1-\beta)\ln \delta}{\sigma\ln b}}\leq c(\delta)r^{\frac{\beta}{\sigma}}.
 \end{split}\end{equation*}
 In view of \eqref{n*}, we conclude with
 \begin{equation*}\begin{split}
 r=R_0\exp\left(\sigma\ln\frac{\bar r}{R_0}+2\ln\delta\right)<R_0e^{n_*\ln\delta} =R_{n_*}
 \end{split}\end{equation*}
 and hence $Q(r,a_0r^p)\subset \bar Q_{n_*}\subset Q_{n_*}$.
 Consequently, we arrive at the continuity estimate as follows
 \begin{equation*}\begin{split}
 &\essosc_{Q(r,a_0r^p)}u\leq \essosc_{\bar Q_{n_*}}u\leq\essosc_{Q_{n_*}}u\leq \omega_{n_*}
 \\&\leq \gamma\left(\frac{r}{R_0}\right)^{\frac{\alpha(1-\beta)}{\sigma}}\omega_0+\gamma
\left(F_1(cr^{\frac{\beta}{\sigma}}
)^{\frac{p}{p+m-1}}+F_2(cr^{\frac{\beta}{\sigma}})^{\frac{p-1}{p+m-2}}\right).
 \end{split}\end{equation*}
 Thus we are led to the conclusion that the weak solution to \eqref{parabolic}-\eqref{A} is locally continuous in $\Omega_T$.
 \section*{Acknowledgements}
 The author wishes to thank the Institute for Advanced Study in Mathematics of Harbin Institute of Technology
 for the kind hospitality during his visit in July 2023.
\bibliographystyle{abbrv}

\end{document}